\documentclass[11pt]{amsart}
\usepackage{graphics,graphicx} % \rotatebox

\usepackage[T1]{fontenc}

\usepackage[cp1250]{inputenc}

\usepackage{amsmath,amsthm,amssymb, dsfont}
\usepackage{amscd}
\usepackage[arrow, matrix, curve]{xy}
\newenvironment{proof of}[1]{\emph{Proof of #1.}}{\hfill $\qquad \square$\par}

\DeclareMathOperator{\In}{In}
\DeclareMathOperator{\Proj}{Proj}
\DeclareMathOperator{\Ob}{Ob}

\DeclareMathOperator{\id}{id}

\DeclareMathOperator{\Aut}{Aut}

\DeclareMathOperator{\spane}{span}
\DeclareMathOperator{\clsp}{\overline{span}}
\newcommand{\I}{\mathcal I}

\newcommand{\K}{\mathcal K}

\newcommand{\LL}{\mathcal{L}}
\newcommand{\KK}{\mathcal K}

\newcommand{\FF}{\mathcal F}

\newcommand{\OO}{\mathcal O}

\renewcommand{\P}{\mathcal{P}}

\renewcommand{\SS}{\mathcal S}
\newcommand{\B}{\mathcal B}

\newcommand{\R}{\mathbb R}
\newcommand{\JJ}{\mathcal J}
\newcommand{\Z}{\mathbb Z}
\newcommand{\N}{\mathbb N}

\newcommand{\TT}{\mathcal T}
\renewcommand{\LL}{\mathcal L}
\newcommand{\NT}{\mathcal{NT}}

\numberwithin{equation}{section}

\newtheorem{thm}{Theorem}[section]\newtheorem{lem}[thm]{Lemma}
\newtheorem{prop}[thm]{Proposition}
\newtheorem{cor}[thm]{Corollary}
\theoremstyle{definition}
\newtheorem{defn}[thm]{Definition}
\newtheorem{ex}[thm]{Example}
\newtheorem{rem}[thm]{Remark}

\title[Nica-Toeplitz algebras for right tensor $C^*$-precategories]{Nica-Toeplitz algebras associated with right tensor $C^*$-precategories over right LCM semigroups}

\author{Bartosz K.  Kwa\'sniewski}
 \email{bartoszk@math.uwb.edu.pl}
 \address{Institute of Mathematics,  University  of Bialystok
ul. Ciolkowskiego  1M, 15-245  Bialystok,   Poland // Department of Mathematics and Computer Science, The University of Southern Denmark,
Campusvej 55, DK-5230 Odense M, Denmark}

\author{Nadia S.  Larsen}
\email{nadiasl@math.uio.no}
\address{Department of Mathematics, University of Oslo, PO Box 1053 Blindern, 0316 Oslo, Norway }

\date{25 November 2016. Revised 5 June 2018.}
\begin{document}

\begin{abstract} We introduce and analyze the full $\NT_{\LL}(\KK)$ and the reduced  $\NT_\LL^r(\KK)$ Nica-Toeplitz algebra associated to an ideal $\KK$ in a right tensor $C^*$-precategory $\LL$ over a right LCM semigroup $P$. These  $C^*$-algebras unify  cross-sectional $C^*$-algebras associated to Fell bundles over discrete groups and Nica-Toeplitz $C^*$-algebras associated to  product systems. They also allow a study of Doplicher-Roberts versions of the latter.

A new phenomenon is that when $P$ is not right cancellative then the canonical conditional expectation takes values outside the ambient algebra. 
Our main result is a uniqueness theorem that gives sufficient conditions for a representation  of $\KK$ to generate a $C^*$-algebra naturally lying between $\NT_{\LL}(\KK)$ and $\NT_\LL^r(\KK)$. 
We also characterise the situation when  $\NT_{\LL}(\KK)\cong \NT_\LL^r(\KK)$.
Unlike previous results for quasi-lattice monoids, $P$ is allowed to contain nontrivial invertible elements, and we accommodate this by identifying an assumption of aperiodicity of an action of the group of invertible elements in $P$. One prominent condition for uniqueness  is a geometric condition of Coburn's type,  exploited in the work of Fowler, Laca and Raeburn. Here we shed  new light on the role  of this condition by relating it to a $C^*$-algebra associated to $\LL$ itself.

\end{abstract}

\maketitle

 \setcounter{tocdepth}{1}
%\tableofcontents

\section{Introduction}

Tensor $C^*$-categories (or monoidal $C^*$-categories) and all the more right-tensor $C^*$-categories
(also called semitensor $C^*$-categories)   arise naturally  in quantum field theory and duality theory of compact (quantum) groups \cite{dr}, \cite{dpz}. In particular, these structures play a fundamental role in numerous recent
results with a flavor of geometric group theory,  see e.g. \cite{KRVV}, \cite{NT}, \cite{PV} and references therein.
Right tensor $C^*$-(pre)categories proved also to be a very natural framework allowing efficient description of the structure of Cuntz-Pimsner and Nica-Toeplitz algebras associated to product systems, see \cite{kwa-doplicher}, \cite{kwa-szym},  \cite{kwa-larII}. In the present paper we initiate a systematic study of $C^*$-algebras associated to right tensor $C^*$-precategories inspired by this last class of examples. In fact, already in the context of Nica-Toeplitz algebras associated to product systems,  our results extend substantially the existing theory, see  \cite{kwa-larII}.
We believe that the ``categorial language'' is well suited to the complicated analysis of $C^*$-algebras over semigroup structures, and has a potential to be used, for instance,  in the study of Cuntz-Nica-Pimsner algebras \cite{SY}  or Doplicher-Roberts algebras  \cite{dr}, \cite{dpz} and their generalizations.

Our initial data consists of  an ideal $\KK$ in a  right-tensor $C^*$-precategory $\LL$ over a discrete, left cancellative and unital semigroup $P$. A $C^*$-precategory, as introduced in \cite{kwa-doplicher}, is a non-unital version of a $C^*$-category. The important example that the reader may keep  in mind  is that of Banach spaces $\LL(X_p, X_q)$ of adjointable operators between Hilbert $A$-modules $X_p, X_q$ for $p,q\in P$, that form a product system $X=\bigsqcup_{p\in P}X_p$ over $P$. Then  the  right tensoring structure $\{\otimes 1_r\}_{r\in P}$
 on $\LL=\{\LL(X_p, X_q)\}_{p,q\in P}$ is given by tensoring on the right with the unit $1_r$ in $\LL(X_r)$.
 The Banach spaces $\KK(X_p, X_q)$, $p,q\in P$, of generalized compacts form a $C^*$-precategory $\KK$,  in fact an ideal in $\LL$, which need not be preserved by the `functors' $\otimes 1_r$, $r\in P$. In special cases such structures were considered in  \cite[Example 3.2]{kwa-doplicher}, \cite[Subsection 3.1]{kwa-szym}.  We give a  detailed analysis of this example in \cite{kwa-larII} where we also explain how  the results of the present paper give a new insight to $C^*$-algebras associated with $X$. Another somewhat trivial but important and instructive example is when $P=G$ is a group.
Then  as we explain (see Section \ref{groups section})  our framework is equivalent to the theory of Fell bundles over discrete groups.

 In general,  there is a natural notion of a right-tensor representation of $\KK$ which allows us to define the \emph{Toeplitz algebra} $\TT_{\LL}(\KK)$ of $\KK$ as the universal $C^*$-algebra for these representations. We construct a canonical injective right-tensor representation $T$ of $\KK$  on an appropriate Fock module $\FF_\KK$.
We call $T$ \emph{the Fock representation} of $\KK$. We wish to study $C^*$-algebras which in general are quotients of $\TT_{\LL}(\KK)$ obtained by considering additional relations  coming from the Fock representation. It was Nica \cite{N}, who first identified and used explicitly such relations
in the context of  $C^*$-algebras associated to positive cones in quasi-lattice ordered groups. Fowler \cite{F99} generalized these conditions to product systems over semigroups studied by Nica. In particular, he introduced notions of \emph{compact-alignment} and
 \emph{Nica-covariance} for such objects.
Recently, definitions of Nica covariant representations and the corresponding Nica-Toeplitz algebras  were generalized to product systems over right LCM semigroups in \cite{bls2}.

In order to define Nica covariance  we work under the assumption that $P$ is a \emph{right LCM semigroup}, a terminology introduced in \cite{BRRW}.  Such semigroups appear also under the name of semigroups satisfying Clifford's condition, see \cite{Law2} and \cite{No0}. We emphasize that passing from positive cones in quasi-lattice ordered groups to right LCM semigroups is not straightforward, and has a number of important consequences. First,
 it allows to develop a theory independent of  the ambient group. In fact, the semigroups we consider need not  be (right) cancellative, and hence they might not be embeddable into any group.
Second, LCM semigroups allow invertible elements. This makes a number of problems much more delicate, but also
allows us to cover a  larger class of interesting examples.  In particular LCM semigroups  could be viewed as a unification of  quasi-lattice ordered semigroups and groups.

Given a  right-tensor $C^*$-precategory $\LL$ over an LCM semigroup $P$ we say that an ideal $\KK$ in $\LL$ is \emph{well-aligned} if for every two `morphisms' in $\KK$ that can be tensored so that they can be composed, the composition  is again in $\KK$. This generalises the notion of compactly aligned product systems. For a well-aligned ideal $\KK$ in a right-tensor $C^*$-precategory $\LL$  we introduce  representations which we call \emph{Nica covariant}. We show that the Fock representation of $\KK$ is Nica covariant. Two $C^*$-algebras are then naturally associated to $\KK$: a \emph{Nica-Toeplitz algebra} $\NT_\LL(\KK)$ universal for Nica covariant representations, and a \emph{reduced Nica-Toeplitz} algebra $\NT_\LL^r(\KK)$ which is generated by the Fock representation of $\KK$.

One important tool to study Nica-Toeplitz type $C^*$-algebras is a conditional expectation onto a natural \emph{core} subalgebra. While core subalgebras of both $\NT_\LL(\KK)$ and $\NT_\LL^r(\KK)$ are easy to write down,  conditional expectations onto the respective cores do not obviously exist. In the generality of a right LCM semigroup, which need not be right cancellative, we find  new ingredients, namely a self-adjoint operator space $B_\KK$, which in general is not a subspace of $\NT_\LL^r(\KK)$,  and a faithful completely positive map $E^T$ from $\NT_\LL^r(\KK)$ onto $B_\KK$. We refer to $B_\KK$ as a \emph{transcendental core} and to $E^T$ as a \emph{transcendental conditional expectation}.
The map $E^T$ becomes a genuine conditional expectation onto the core subalgebra of $\NT_\LL^r(\KK)$ if and only if the semigroup $P$ is cancellative. We note that a similar phenomenon was recently discovered  in the context of $C^*$-algebras associated to actions of inverse semigroups of Hilbert bimodules in \cite{BEM}, where a notion of a weak conditional expectation is introduced (it is a genuine conditional expectation if and only if the space of units is closed in the dual transformation groupoid).  We define an \emph{exotic Nica-Toeplitz algebra }to be the $C^*$-algebra $C^*(\Phi(\KK))$ generated by a  Nica covariant representation $\Phi$ such that there is a compatible transcendental conditional expectation from $C^*(\Phi(\KK))$ onto $B_\KK$. This is equivalent to existence of a $*$-homomorphism $\Phi_*$ making the following diagram commute:
$$ %\begin{equation}\label{pseudo uniqueness diagram0}
\xymatrix{  \NT_{\LL}(\KK)  \ar@/_2pc/[rrrr]^{\,\, T\rtimes P}  \ar[rr]^{\Phi\rtimes P}  &
&C^*(\Phi(\KK))  \ar[rr]^{\Phi_* \,\,} & &   \NT_{\LL}^{r}(\KK)}
$$ %\end{equation}
%where $ T\rtimes P$ is the integrated form of the Fock representation, which we call the \emph{regular representation} of $\KK$.
The uniqueness theorems we aim at require studying two somewhat  independent problems, which are interesting in their own right.  The first problem is to identify ideals $\KK$ for which the regular representation $ T\rtimes P$ is injective. The second one is to find conditions on a Nica covariant representation $\Phi$ so that
 $C^*(\Phi(\KK))$ is an
exotic Nica-Toeplitz algebra. Having these two ingredients we infer that both $\Phi\rtimes P$ and $\Phi_*$ are isomorphisms.

Amenability as a characterization of injectivity of the regular representation of the universal $C^*$-algebra for Nica covariant representations is prominent in work of  Nica \cite{N}, Laca-Raeburn \cite{LR} and Fowler \cite{F99}. Motivated by this, we say that a well aligned ideal $\KK$ of $\LL$ is \emph{amenable} if the regular representation $T\rtimes P$  is an isomorphism from $\NT_\LL(\KK)$ onto $\NT_\LL^r(\KK)$. In Theorem~\ref{amenability main result} we prove a far-reaching generalization of \cite[Proposition 4.2]{LR} and \cite[Theorem 8.1]{F99}. %, which give sufficient conditions for amenability of semigroups (twisted by product systems).
 Our result establishes necessary and sufficient conditions for amenability of $\KK$ in terms of amenability of a Fell bundle that arises in a canonical way whenever there is a \emph{controlled semigroup homomorphism} from $P$ to another right LCM semigroup that sits inside a group $G$. In particular, if $G$ can be chosen to be amenable, then any well-aligned ideal in a right-tensor $C^*$-precategory over $P$ is amenable. As we show in Corollary~\ref{amenability of free products} this applies to a large class of semigroups
 obtained from free products of right LCM semigroups.

 A novel ingredient in our study is that we identify  an algebraic condition which characterizes when a representation
 $\Phi\rtimes P$ of 
$\NT_\LL(\KK)$ is injective on the core: we call this \emph{Toeplitz covariance}. Such a condition was previously considered only in the case $P=\N$, in the context of relative Cuntz-Pimsner and relative Doplicher-Roberts algebras, cf. \cite{kwa-doplicher}.
 In Corollary~\ref{the reduced core} we prove that $T$ is  Toeplitz covariant, which equivalently means that $T\rtimes P$ is injective on the core of $\NT_\LL(\KK)$. The main technical result needed to achieve these characterizations is Theorem~\ref{theorem for amenability and spectral subspaces}. It says that any controlled semigroup homomorphism  from $P$ to an arbitrary right LCM semigroup  induces a $C^*$-subalgebra of $\NT_\LL(\KK)$ containing the core, and gives conditions 
characterising when $\Phi\rtimes P$ is injective on the induced $C^*$-algebra. This result is inspired by \cite[Lemma 4.1]{LR} and the proof of \cite[Theorem 8.1]{F99}. Nevertheless, since we deal here with much more general situation our proof requires some new non-trivial steps.

Another new ingredient in our approach is related to a potential existence of invertible elements in a right LCM semigroup $P$. We show that for any well-aligned ideal $\KK$ in a right-tensor $C^*$-precategory $\LL$ the restriction of right tensoring to the group of invertible elements $P^*$ preserves $\KK$. In Definition \ref{definition of aperiodicity for right-tensor categories} we introduce the \emph{aperiodicity condition} for this action of $P^*$ on $\KK$. If $P^*=\{e\}$ is trivial, this condition becomes vacuous. If $P^*=P$, that is, when $P$ is a group, then  $\KK$ can be viewed as a Fell bundle $\B$ over $P$, and aperiodicity of $\KK$ is equivalent to aperiodicity of $\B$ (see Section \ref{groups section}). By  \cite{KM}, if $B_e$  is separable or of Type I, apperiodicity of $\B$ is equivalent to  \emph{topological freeness} of the dual partial action. Non-trivial examples where the aperiodicity condition holds come for instance from wreath products, see \cite{kwa-larII}.

 Our main results are the uniqueness theorems in Section~\ref{The main result section} and Section~\ref{relationship-K-L}. For a Nica covariant representation $\Phi$ of a well-aligned ideal $\KK$ in a right-tensor $C^*$-precategory $\LL$ we identify two sources from which to extract characterizations of injectivity of $\Phi\rtimes P$.  
 In Definition \ref{defn:Condition (C)} we introduce a geometric condition on $\Phi$,  which we call \emph{condition (C)}. It is closely related to the condition describing injectivity of representations of the Toeplitz algebra of a single $C^*$-correspondence, see \cite[Theorem 2.1]{Fow-Rae} or the condition for semigroup crossed products  twisted by a product system, see \cite[Equation (7.2)]{F99}.
 Theorem~\ref{preludium to Nica Toeplitz uniqueness2} is the main technical result on $\Phi$ that links condition (C),  injectivity of $\Phi\rtimes P$ on the core, Toeplitz covariance, and generation of an exotic Nica-Toeplitz algebra. Our first uniqueness result, Corollary~\ref{cor:uniquenessI}, says that  when $\KK$ is amenable and the action of $P^*$ on $\KK$ is aperiodic, then condition (C) on a representation $\Phi$ of $\KK$ implies injectivity of $\Phi\rtimes P$. The converse holds if $\KK$ is right-tensor invariant.

 The most satisfactory uniqueness result is contained in  Section~\ref{relationship-K-L}. Here we reveal the true nature of condition (C). Under natural assumptions we show that the $C^*$-algebra $\NT_\LL(\KK)$ associated to $\KK$ is a subalgebra of the $C^*$-algebra $\NT(\LL):=\NT_{\LL}(\LL)$ associated to $\LL$ (the latter can be thought of as Doplicher-Roberts version of the former). Moreover, every Nica  covariant representation $\Phi$ of $\KK$ extends uniquely to a Nica covariant representation $\overline{\Phi}$ of $\LL$. We prove, see Corollary~\ref{Uniqueness Theorem}, that condition (C) for $\Phi$ is equivalent to injectivity of  the representation $\overline{\Phi}\rtimes P$ of the (Doplicher-Roberts) $C^*$-algebra $\NT(\LL)$. This  also implies injectivity of  $\Phi\rtimes P$ on $\NT_\LL(\KK)$, as $\Phi\rtimes P$ is a restriction of  $\overline{\Phi}\rtimes P$. In addition  condition (C) is equivalent to  $\overline{\Phi}$ being Toeplitz covariant and injective, cf. Theorem \ref{Nica Toeplitz uniqueness}. 
It seems that in general, the geometric condition (C) is responsible for uniqueness of the $C^*$-algebra associated to $\LL$ while 
uniqueness of $C^*$-algebra associated to $\KK$ should be related with the algebraic condition of Toeplitz covariance, 
cf. also \cite{kwa-larII}.

\subsection{Acknowledgements}

The research leading to these
results has received funding from the European Union's
Seventh Framework Programme (FP7/2007-2013) under grant agreement number 621724. B.K. was  partially supported by the NCN (National Centre of Science) grant number 2014/14/E/ST1/00525.  This work was finished while he participated in the Simons Semester at IMPAN - Fundation grant 346300 and the Polish Government MNiSW 2015-2019 matching fund. Part of the work was done during the participation of both  authors in
the program "Classification of operator algebras: complexity, rigidity, and dynamics" at the Mittag-Leffler Institute (Sweden), and during the visit of N.L. at the University of Victoria (Canada). She thanks Marcelo Laca and the department at UVic for hospitality.

\section{Preliminaries}\label{section:preliminaries}

\subsection{LCM semigroups} We refer to \cite{BRRW} and \cite{bls} and the references therein for general facts about LCM semigroups. Throughout this paper $P$ is a \emph{left cancellative} semigroup with the \emph{identity} element $e$. We let $P^*$ be the group of \emph{units}, or invertible elements, in $P$, where $x\in P$ is invertible if there exists (a necessarily unique) $x^{-1}\in P$ such that $xx^{-1}=x^{-1}x=e$. A \emph{principal right ideal} in $P$ is a right ideal in $P$ of the form $pP=\{ps: s\in P\}$ for some $p\in P$. Occasionally, we will  write $\langle  p \rangle:=pP$.
The relation of inclusion on the principal right ideals induces a left invariant \emph{preorder} on $P$ given by
$$
p \leq q \quad \stackrel{def}{\Longleftrightarrow}\quad qP\subseteq  pP   \quad \Longleftrightarrow\quad \exists {r\in P}\,\, q=pr.
$$
This preorder  is a partial order if and only if $P^*=\{e\}$.
Left cancellation in $P$ implies that for fixed  $p,q\in P$, $pr=q$  determines $r\in P$ uniquely, motivating the notation:
$$
 p^{-1}q: =r \quad \textrm{ if }\,\,\,  q=pr .
 $$
The following property of semigroups is sometimes called {\em Clifford's condition}  \cite{Law2}, \cite{No0}.

\begin{defn}
 A semigroup $P$ is a \emph{right LCM semigroup} if it is left cancellative and the family $\{pP\}_{p\in P}$ of principal right ideals extended by the empty set  is closed under intersections, that is if for every pair of elements
$p, q\in P$ we have $pP\cap qP=\emptyset$ or $pP\cap qP=rP$ for some $r\in P$.
\end{defn}
In the case that $pP\cap qP=rP$, the element $r$ is a \emph{right least common multiple (LCM)}  of $p$ and $q$. Note that a right LCM is determined by $p$ and $q$ up to  multiplication from the right by an invertible element. Namely, if  $pP\cap qP=rP$, then  $pP\cap qP=tP$ if and only if there is $x\in P^*$ such that $t=rx$. If $P$ is a right LCM semigroup we will refer to $
J(P):=\{pP\}_{p\in P}\cup\{\emptyset\}$ as the \emph{semilattice of principal right ideals} of $P$, see \cite{BRRW} and \cite{Li}.

\begin{ex} One of the most known and studied examples of right LCM semigroups are positive cones in quasi-lattice ordered groups, introduced by Nica \cite{N}. More precisely, suppose that $P$ is a subsemigroup of a group $G$ such that $P\cap P^*=\{e\}$. Then the partial order we defined on $P$ extends 
to a left-invariant partial order on $G$  where $g \leq h$   
if and only if $g^{-1}h\in P$
 for all $g,h\in G$. The pair $(G,P)$ is a \emph{quasi-lattice ordered group} if every pair of elements $g,h\in G$ that has an upper bound in $P$ admits a least upper bound in $P$. Note that if the upper bound exists, it is unique since $P^*=\{e\}$. In \cite[Definition 32.1]{exel-book}, Exel calls the pair $(G,P)$  \emph{weakly quasi-lattice ordered} if for each pair of elements $g,h\in P$ with an upper bound in $P$ there exists a (necessarily unique)   least upper bound in $P$.

Every positive cone  $P$ in a weakly quasi-lattice ordered group $(G,P)$ is an LCM semigroup with $P^*=\{e\}$. Conversely, if $P$ is an  LCM subsemigroup of a group $G$ such that $P^*=\{e\}$ then $(G,P)$ is a weakly quasi-lattice ordered group.
\end{ex}

Let $P_i$, $i\in I$, be a family of right LCM semigroups. 
 The direct sum $\bigoplus_{i\in I} P_i$ is a right LCM semigroup with semilattice of principal right ideals 
 isomorphic to the direct sum of semilattices $J(P_i)$, $i\in I$. It is slightly less obvious that also the 
free product $\prod^*_{i\in I}P_i$ is a right LCM semigroup. This follows from the proof of the following proposition, which is a generalization of \cite[Proposition 4.3]{LR}. 
It links the two aforementioned constructions by a useful homomorphism.

\begin{prop}\label{form free products to direct sums}
Let $P_i$, $i\in I$, be a family of right LCM semigroups. Put $P:=\prod^\ast_{i\in I} P_i$  and $\P:=\bigoplus_{i\in I} P_i$, and let $\theta:P\to \P$ be the homomorphism which is the identity on each $P_i$, $i\in I$. Then  $\theta(P^*)=\P^*$ and for any $s,t,r\in P$ with $sP\cap tP = rP$ we have
\begin{equation}\label{lcb preserver2}
\theta(s)\P\cap \theta(t)\P = \theta(r)\P  \qquad \text{and}\qquad \theta(s)=\theta(t) \,\text{ implies }\,  s=t.
\end{equation}
\end{prop}
\begin{proof}
 Note that $P^*=\prod^*_{i\in I} P_i^*$ and $\P^*=\bigoplus_{i\in I} P_i^*$. Hence $\theta(P^*)=\P^*$.
Now, let $s=p_{i_1}\cdots p_{i_n}\in P$ be in reduced form, that is $p_{i_k}\in P_{i_k}$ and  $i_k\neq i_{k+1}$ for all $k=1,\dots, n-1$. Similarly, write $t=q_{j_1}\cdots q_{j_m}\in P$. Without loss of generality we may assume that $m\geq n$. Suppose that $sP\cap tP \neq \emptyset$. This implies that
$p_{i_1}\cdots p_{i_{n-1}}=q_{i_1}\cdots q_{i_{n-1}}$, $i_n=j_n$ and either
\begin{itemize}
\item[(1)] $p_{i_n}\leq q_{j_n}$, if $m >n$ , in which case $sP\cap tP=tP$, or
\item[(2)] $m =n$ and $p_{i_n}P_{i_n}\cap  q_{j_n}P_{i_n}=r_{i_n}P_{i_n}$ for some $r_{i_n}\in P_{i_n}$, in which case  $sP\cap tP=sr_{i_n}P$.
\end{itemize}
Clearly, in both cases, we have $\theta(s)\P\cap \theta(t)\P = \theta(r)\P$. Moreover, $\theta(s)=\theta(t)$ implies that $m=n$ and
$p_{i_n}= q_{j_n}$, that is, $s=t$.
\end{proof}

The property identified in \eqref{lcb preserver2} is important.  We will formalise it in a definition. To indicate the analogy with similar concepts introduced in \cite{LR} and \cite{CL}, we borrow Crisp and Laca's terminology of \emph{controlled map}, see \cite[Definition 4.1]{CL} and in particular conditions (C2), (C3) and (C5), but add the tag of right LCM semigroup.

\begin{defn}\label{def:controlled map}
A \emph{controlled map of right LCM semigroups} is an identity preserving homomorphism $\theta:P\to \P$ between right LCM semigroups $P,\P$ such that $\theta(P^*)=\P^*$ and for all $s,t\in P$ with $sP\cap tP \neq\emptyset$ we have
\begin{equation}\label{lcb preserver}
\theta(s)\P\cap \theta(t)\P = \theta(r)\P\,\text{ whenever }r\text{ is a right LCM for }s,t
\end{equation}
and
\begin{equation}\label{injectivity in the same direction}
\theta(s)=\theta(t) \,\, \Longrightarrow\,\,  s=t.
\end{equation}
\end{defn}

\begin{rem}\label{rem:P cancellative if controlled}
{(a)} If $\theta:P \to \P$  is a controlled map of right LCM semigroups and $\P$ is right cancellative, then so is $P$. Indeed, if $pr=qr$ in $P$ then $\theta(p)\theta(r)=\theta(q)\theta(r)$, so right cancellation in $\P$ implies $\theta(p)=\theta(q)$ giving $p=q$ by condition \eqref{injectivity in the same direction}. In our applications, $\P$ will be contained in a group so it will necessarily be cancellative.

(b) A controlled map of right LCM semigroups seems to be different from a homomorphism of right LCM semigroups as considered in \cite[Equation (3.1)]{bls2}. This is because a homomorphism of right LCM semigroups need not preserve the group of units and, besides, it keeps track of pairs of elements in $P$ and their images in $\P$ that do not have a right common upper bound as well as those who do.
\end{rem}

\subsection{$C^*$-precategories}

We recall some background on $C^*$-precategories from \cite{kwa-doplicher}. $C^*$-precategories should be viewed as  non-unital versions of $C^*$-categories, cf. \cite{glr}, \cite{dr}.

A \emph{precategory }$\LL$ consists of
a set  of   objects $\Ob(\LL)$ and a collection $\{\LL(\sigma,\rho)\}_{\sigma,\rho\in \Ob(\LL)}$ of sets  of morphisms   endowed with an associative composition. Explicitly,   $\LL(\sigma,\rho)$ stands for the space of morphisms from $\rho$ to $\sigma$, and  the
composition
$
\LL(\tau,\sigma)\times \LL(\sigma,\rho) \to  \LL(\tau,\rho),\,
(a, b)\to ab
$
must satisfy
$
(a b) c=a (b c)
$  whenever the compositions of morphisms $a$, $b$, $c$ are allowable. A morphism in $\LL(\sigma,\rho)$ may  be regarded as an arrow from $\rho$ to $\sigma$. One  can equip (if necessary)  $\LL(\sigma, \sigma)$ with identity morphisms  in such a way that a given  precategory $\LL$ becomes a category. In the sequel, we will identify $\LL$ with the collection of morphisms $\{\LL(\sigma,\rho)\}_{\sigma,\rho\in \Ob(\LL)}$.

 \begin{defn}(\cite[Definition 2.2]{kwa-doplicher})
 A \emph{$C^*$-precategory} is a precategory $\LL=\{\LL(\sigma,\rho)\}_{\sigma,\rho\in \Ob(\LL)}$  together with an operation
 $^*:\LL\to\LL$  such that the following hold:
\begin{enumerate}\renewcommand{\theenumi}{p\arabic{enumi}}
\item\label{it:p-1} each set of morphisms $\LL(\sigma,\rho)$, $\sigma, \rho\in \LL$ is a complex Banach space;
\item\label{it:p-2} composition gives a bilinear map
$$
\LL(\tau,\sigma)\times \LL(\sigma,\rho)\ni(a, b)\to ab\in \LL(\tau,\rho),
$$
which satisfies $\|ab\|\leq \|a\|\cdot \|b\|$;
\item\label{it:p-3} for each $\sigma$, $\rho\in \LL$,    $a \to a^*$  is an antilinear map from $\LL(\sigma,\rho)$ to $\LL(\rho,\sigma)$ such that $(a^*)^*=a$ and $(a b)^*=b^* a^*$ for all $b\in  \LL(\rho,\tau)$ and all $\tau \in \LL$;
\item\label{it:p-4} $\|a^* a\|=\|a\|^2$ for every  $a \in   \LL(\sigma,\rho)$; and
\item\label{it:p-5} for each $a \in   \LL(\sigma,\rho)$, we have $a^*a=b^*b$ for some $b\in \LL(\rho,\rho)$.
\end{enumerate}

We say that  a $C^*$-precategory $\LL$  is  \emph{$C^*$-category}  if  $\LL$ is a category.
\end{defn}
Note that (p3) says that the operation $^*$ is antilinear, involutive and  contravariant. Moreover, (p1)--(p4) imply that $\LL(\rho,\rho)$ is  a $C^*$-algebra, and (p5) says that $a^*a$ is positive as an element of  $\LL(\rho,\rho)$. Condition (p4)  implies that the operation $^*$ is isometric on every space $\LL(\sigma,\rho)$.
A $C^*$-precategory $\LL$ is a $C^*$-category if and only if every $C^*$-algebra $\LL(\rho,\rho)$, $\rho \in \Ob(\LL)$, is unital.
\begin{lem}\label{about approximate units}
Given  a $C^*$-precategory $\LL$ and objects $\sigma, \rho$ we have
$$
\LL(\sigma,\rho)= \LL(\sigma,\sigma) \LL(\sigma,\rho)= \LL(\sigma,\rho)\LL(\rho,\rho).
$$
\end{lem}
\begin{proof} In a natural way, $\LL(\sigma,\rho)$ is a left Banach $\LL(\sigma,\sigma)$-module and a right  Banach $\LL(\rho,\rho)$-module.
By  \cite[Lemma 2.5]{kwa-doplicher},  these module actions are non-degenerate. Hence the assertion follows by the Cohen-Hewitt factorization theorem.
\end{proof}

\begin{defn}[Definition 2.4 in \cite{kwa-doplicher}]
An \emph{ideal in a $C^*$-precategory} $\LL$ is a collection $\KK=\{\KK(\sigma,\rho)\}_{\sigma,\rho\in \Ob(\LL)}$  of closed linear subspaces $\KK(\sigma,\rho)$ of $\LL(\sigma,\rho)$,  $ \rho,\sigma \in \Ob(\LL)$,  such that
$$
  \LL(\tau,\sigma)\KK(\sigma,\rho) \subseteq \KK(\tau,\rho)\quad \textrm{ and } \quad \KK(\tau,\sigma)\LL(\sigma,\rho)\, \subseteq \KK(\tau,\rho),
  $$
  for all $\sigma,\rho,\tau\in \Ob(\LL)$.
\end{defn}

An ideal $\KK$ in a $C^*$-precategory $\LL$ is automatically selfadjoint in the sense that $\KK(\sigma,\rho)^*= \KK(\rho,\sigma)$,
for all $\sigma,\rho\in \Ob(\LL)$, cf. \cite[Proposition 1.7]{glr}. Hence $\KK$ is a $C^*$-precategory. Each space $\KK(\rho,\rho)$   is a closed two-sided  ideal in the  $C^*$-algebra $\LL(\rho,\rho)$.
A useful fact is that    $\KK$  is uniquely determined by these diagonal ideals.

\begin{prop}[Theorem 2.6 in \cite{kwa-doplicher}]\label{diagonal of ideals}
If $\KK$ is an ideal in a $C^*$-precategory $\LL$, then for all $\sigma$ and $\rho$ the space $\KK(\sigma,\rho)$ coincides with
\begin{equation}\label{ideal defining formula}
\{a\in \LL(\sigma,\rho):a^*a \in \KK(\rho,\rho)\}=\{a\in \LL(\sigma,\rho):aa^* \in \KK(\sigma,\sigma)\}.
\end{equation}
 Conversely,  a collection of ideals $\KK(\rho,\rho)$  in $\LL(\rho,\rho)$, for $\rho\in \Ob(\LL)$, satisfying \eqref{ideal defining formula} gives rise to an ideal $\KK$ in $\LL$.
\end{prop}
We generalize the notion of an essential ideal in a $C^*$-algebra as follows.
\begin{defn}\label{defn:essential ideal} An ideal  $\KK$ in a $C^*$-precategory $\LL$ is an \emph{essential ideal}  in $\LL$ if
 \begin{equation}\label{eq:essential ideal}
\KK(\rho,\rho) \text{ is an essential ideal in }\LL(\rho,\rho) \text{ for every } \rho \in \Ob(\LL).
\end{equation}
\end{defn}
Homomorphisms between $C^*$-precategories are defined in a natural way. We recall from
  \cite[Definition 2.8]{kwa-doplicher} that
a \emph{homomorphism} $\Phi$ from a $C^*$-precategory $\LL$ to a $C^*$-pre\-ca\-te\-gory  $\SS$ consists of a map $\Ob(\LL) \ni\sigma \mapsto \Phi(\sigma)\in \Ob(\SS)$ and  linear operators
$ \LL(\sigma,\rho)\ni a \mapsto \Phi(a) \in \SS(\Phi(\sigma),\Phi(\rho))$,  $\sigma, \rho\in \Ob(\LL)$, such that $
\Phi(a)\Phi(b)=\Phi(ab)$ and $\Phi(a^*)=\Phi(a)^*$
for all $a\in \LL(\tau,\sigma)$, $b\in \LL(\sigma,\rho)$ and all $\sigma, \rho, \tau\in \Ob(\LL)$.
 An \emph{endomorphism} of $\LL$ is a homomorphism from $\LL$ to $\LL$.

Note that composition of homomorphisms is again a homomorphism. If  $\Phi:\LL\to \SS$ is a homomorphism between $C^*$-precategories,
 then  $\Phi:\LL(\rho,\rho)\to \SS(\Phi(\rho),\Phi(\rho))$ are $^*$-homomorphisms of $C^*$-algebras. Using this observation one gets the following fact, cf.  \cite[Proposition 2.9]{kwa-doplicher}.
\begin{lem}\label{proposition 1.5}
For any homomorphism $\Phi:\LL\to \SS$  of  $C^*$-precategories   the  operators
\begin{equation}\label{restricted maps}
\Phi:\LL(\sigma,\rho)\to \SS(\Phi(\sigma),\Phi(\rho)), \qquad \sigma,\rho \in \Ob(\LL),
\end{equation}
 are  contractions. Moreover, all the maps in \eqref{restricted maps} are isometric if and only if all  the maps in \eqref{restricted maps} with $\sigma=\rho$ are injective.
 \end{lem}

 \begin{defn}\label{representations of categories definition}
 A \emph{representation of a $C^*$-precategory} $\LL$ in a $C^*$-algebra $B$ is a homomorphism $\Phi:\LL \to B$ where $B$ is considered as a $C^*$-precategory with a single object. Equivalently, $\Phi$ may be viewed as a collection $\{\Phi_{\sigma,\rho}\}_{\sigma,\rho\in \Ob(\LL)}$ of linear operators  $\Phi_{\sigma,\rho}:\LL(\sigma,\rho)\to B$ such that
$$
\Phi_{\sigma,\rho}(a)^*=\Phi_{\rho,\sigma}(a^*), \quad \textrm{ and } \quad  \Phi_{\tau,\rho}(ab)=\Phi_{\tau,\sigma}(a)\Phi_{\sigma,\rho}(b),
$$
for  all $a\in \LL(\tau,\sigma)$, $b\in \LL(\sigma,\rho)$.
 A \emph{representation} of $\LL$ \emph{on a Hilbert space} $H$  is a representation of  $\LL$ in the $C^*$-algebra $\B(H)$ of all bounded operators on $H$. We say that a representation $\{\Phi_{\sigma,\rho}\}_{\sigma,\rho\in \Ob(\LL)}$ is \emph{injective} if all $\Phi_{\rho,\rho}$ for $\rho\in \Ob(\LL)$ are injective (then all the maps $\Phi_{\sigma,\rho}$,  $\sigma\in \Ob(\LL)$, are isometric  by Lemma \ref{proposition 1.5}).
 \end{defn}

Let  $\{\Phi_{\sigma,\rho}\}_{\sigma,\rho\in \LL}$  be a representation of a $C^*$-precategory $\LL$.  If $\KK$ is an ideal in $\LL$ then $\{\Phi_{\sigma,\rho}|_{\KK(\sigma,\rho)}\}_{\sigma,\rho\in \LL}$ is a representation of $\KK$. In the converse direction we have the following result, cf. \cite[Proposition 2.13]{kwa-doplicher}.
\begin{prop}\label{extensions of representations on Hilbert spaces0} Suppose that $\KK$ is an ideal in a $C^*$-precategory $\LL$ and let $\Phi=\{\Phi_{\sigma,\rho}\}_{\sigma,\rho\in \LL}$ be a representation of $\KK$ on a Hilbert space $H$.
 There is  a unique extension $\overline{\Phi}=\{\overline{\Phi}_{\sigma,\rho}\}_{\sigma,\rho\in \LL}$ of  $\Phi$  to a representation  of  $\LL$   such that the essential subspace of $\overline{\Phi}_{\sigma,\rho}$ is  contained in the essential subspace of $\Phi_{\sigma,\rho}$, for every $\sigma,\rho\in\Ob(\LL)$.  Namely,
\begin{equation}\label{formula defining extensions of right tensor representations}
\overline{\Phi}_{\sigma,\rho}(a)(\Phi_{\rho,\rho}(\KK(\rho,\rho)) H)^\bot =0,\quad \text{ and }\quad  \overline{\Phi}_{\sigma,\rho}(a) \Phi_{\rho,\rho}(b)h = \Phi_{\sigma,\rho}(ab)h
\end{equation}
for all $ a\in \LL(\sigma,\rho)$, $b \in \KK(\rho,\rho)$, $h \in H$.
Moreover,
\begin{equation}\label{kernel of the extended}
(\ker\overline{\Phi})(\sigma,\rho)=\{a\in \LL(\sigma,\rho): a\KK(\rho,\rho)\subseteq \ker\Phi_{\sigma,\rho}\}.
\end{equation}
In particular, $\overline{\Phi}$ is injective if and only if $\Phi$ is injective and $\KK$ is an essential ideal in $\LL$.
\end{prop}
\begin{proof}
The existence and uniqueness of $\overline{\Phi}$ satisfying \eqref{formula defining extensions of right tensor representations} are guaranteed by \cite[Proposition 2.13]{kwa-doplicher}.  Let  $a\in \LL(\sigma,\rho)$, $\sigma, \rho\in \Ob(\LL)$. By \eqref{formula defining extensions of right tensor representations},  we have
$$
a\in\ker\overline{\Phi}_{\sigma,\rho}\,  \Longleftrightarrow\,  \Phi_{\sigma,\rho}(a \KK(\rho,\rho))=\{0\} \, \Longrightarrow\,   a\KK(\rho,\rho)\subseteq \ker\Phi_{\sigma,\rho},
$$
which proves the second part of  the assertion.
  \end{proof}

\begin{ex}\label{ex:primordial} The prototypical examples of $C^*$-precategories arise from adjointable maps between Hilbert modules. Specifically, let $X=\{ X_p\}_{p \in S}$ be a family
of right Hilbert modules over a $C^*$-algebra $A$, indexed by a set $S$.  Then the families
\begin{equation*}
\KK(p,q):= \KK(X_q,X_p),\qquad \LL(p,q) := \LL(X_q, X_p)
\end{equation*}
for $p,q \in S$ with  operations inherited from the corresponding spaces form   $C^*$-precategories.  In fact, $\LL$ is a $C^*$-category and $\KK$ is an essential ideal in $\LL$.
\end{ex}

\section{Nica covariant representations and Nica-Toeplitz algebra}\label{section:Nica covariance}

Right-tensor $C^*$-precategories over the semigroup $\N$ were introduced in \cite{kwa-doplicher}, where they were shown to provide a good framework for  studying Pimsner and Doplicher-Roberts type $C^*$-algebras. Here we notice that \cite[Definition 3.1]{kwa-doplicher} makes sense for an arbitrary semigroup $P$, and we set out to study associated $C^*$-algebras.

\begin{defn}  A \emph{right-tensor $C^*$-precategory} is a  $C^*$-precategory $\LL=\{\LL(p,q)\}_{p,q\in P}$ whose objects form a  semigroup $P$ with identity $e$ and which is equipped with a semigroup $\{\otimes 1_r\}_{r\in P}$ of endomorphisms of $\LL$ such that $\otimes 1_r$  acts on $P$ by sending  $p$ to $pr$, for all $p,r\in P$, and $\otimes 1_e=\id$.  For a morphism $a\in \LL(p,q)$ we denote the value of $\otimes 1_r$ on $a$ by $a\otimes 1_r$, and note that it belongs to $\LL(pr,qr)$.  We refer to $\{\otimes 1_r\}_{r\in P}$  as to a \emph{right tensoring} on $\LL=\{\LL(p,q)\}_{p,q\in P}$.
\end{defn}

Note in particular that for all $ a\in \LL(p,q)$,  $b\in \LL(q,s)$, and $p,q,r, s\in P$ we have
$$
 ((a\otimes 1_r)\otimes 1_s) = a\otimes 1_{rs},
 \qquad
(a\otimes 1_r)^*=a^*\otimes 1_r,\qquad   (a \otimes 1_r)  (b\otimes 1_r)= (ab)\otimes 1_r.
$$
The following definition is a semigroup generalization of \cite[Definition 3.6]{kwa-doplicher}.
\begin{defn}\label{tensor representation definition}
Let $\KK$ be an ideal in a right-tensor $C^*$-precategory $\TT$. We say that a representation $\Phi:\KK\to B$ of  $\KK$  in a $C^*$-algebra $B$ is a \emph{right-tensor representation} if for all
 $a\in \KK(p,q)$ and $b\in   \KK(s,t) $ such that $sP\subseteq qP$ we have
\begin{equation}\label{right tensor representation condition}
\Phi(a)\Phi(b)
=
\Phi \left((a \otimes 1_{q^{-1}s}) b\right).
\end{equation}
We let $C^*(\Phi(\KK))$ be the $C^*$-algebra  generated by the spaces $\Phi(\KK(p,q))$, $p,q\in P$. We call it the $C^*$-algebra generated by $\Phi$.
 \end{defn}
 \begin{rem}
 Since $\KK$ is an ideal  the right hand side  of \eqref{right tensor representation condition}  makes sense. Furthermore,  by  taking adjoints one gets the symmetrized  version of this equation:
  $$
 \Phi(a)\Phi(b)= \Phi( a (b\otimes 1_{s^{-1}q})),
$$
where $a\in \KK(p,q)$, $b\in  \KK(s,t)$, $qP\subseteq sP$, $p,q,s,t\in P$.
\end{rem}

Standard arguments coupled with our proof of existence of an injective Nica covariant representation, see Proposition \ref{Nica Toeplitz reduced representation} below, show that the following proposition holds.
\begin{prop}\label{Toeplitz description}  Let $\KK$ be an  ideal in a  right-tensor  $C^*$-precategory $\LL$. There are a $C^*$-algebra $\TT_{\LL}(\KK)$ and
an injective right-tensor representation $t_{\KK}: \KK \to \TT_{\LL}(\KK) $, such that
\begin{itemize}
\item[(a)] for every right-tensor  representation $\Phi$ of $\KK$ there is a homomorphism  $\Phi\times P$
of $\TT_{\LL}(\KK)$ such that $(\Phi\times P)\circ t_{\KK} =\Phi$; and
\item[(b)] $\TT_{\LL}(\KK)=C^*(t_{\KK}(\KK))$.
\end{itemize}
The $C^*$-algebra  $\TT_{\LL}(\KK)$ is unique up to canonical isomorphism.
\end{prop}
\begin{proof}
Since  every representation $\Phi:\KK\to B$ is automatically contractive, cf. Lemma \ref{proposition 1.5}, a direct sum of  right-tensor representations of $\KK$ is a right-tensor representation. Thus existence and uniqueness of $\TT_{\LL}(\KK)$ follow from \cite[Section 1]{blackadar}. %, see also \cite[Proposition 2.14]{kwa-exel}.
Injectivity of $t_{\KK}$ follows from Proposition \ref{Nica Toeplitz reduced representation} that we prove below.
\end{proof}
\begin{defn}
Given an  ideal $\KK$ in a right-tensor $C^*$-precategory $(\LL, \{\otimes 1_r\}_{r\in P})$, the $C^*$-algebra $\TT_{\LL}(\KK)$ described in Proposition \ref{Toeplitz description} is the  \emph{Toeplitz algebra} of $\KK$.
%We write $\TT(\LL)$ for the Toeplitz algebra $\TT_{\LL}(\LL)$ associated to $\LL$, viewed as an  ideal in itself.
\end{defn}
The Toeplitz algebra $\TT_{\LL}(\KK)$ in general is very large. It lacks a version of 'Wick ordering' and therefore its structure is hardly accessible. This is the main reason why in the present paper we will study  $C^*$-algebras generated by representations satisfying a  condition of Nica type, which is stronger than
\eqref{right tensor representation condition}. Since such conditions are (so far) established only for right LCM semigroups,  from now on (with the exception of Section \ref{Fock representation section}) we will \emph{always assume that $P$ is a  right LCM semigroup}.

\begin{defn}\label{well-alignment definition}
Let $(\LL, \{\otimes 1_r\}_{r\in P})$ be a right-tensor  $C^*$-precategory  over a right LCM semigroup $P$. An ideal $\KK$ in $\LL$ is \emph{well-aligned} in $(\LL, \{\otimes 1_r\}_{r\in P})$ if   for all $a\in \KK(p,p)$, $b\in \KK(q,q) $ we have
\begin{equation}\label{compact alignment relation}
(a\otimes 1_{p^{-1}r}) (b\otimes 1_{q^{-1}r}) \in \KK(r,r)\qquad \textrm{whenever}\quad pP\cap qP=rP.
 \end{equation}
An ideal $\KK$ in $\LL$ is $\otimes 1$-\emph{invariant}  if  $\KK(p,p)\otimes 1_r \subseteq \KK(pr,pr)$ for all $p,r\in P$. We denote this property of $\KK$ as $\KK\otimes 1\subseteq \KK$.
\end{defn}

Note that by Proposition~\ref{diagonal of ideals}, if $\KK\otimes 1\subseteq \KK$, then $\KK(p,q)\otimes 1_r \subseteq \KK(pr,qr)$ for all $p,q,r\in P$.
Plainly, if $\KK$ is a $\otimes 1$-invariant ideal in a right-tensor  $C^*$-precategory $\LL$, then $\KK$ is itself a right-tensor $C^*$-precategory, and $\KK$ is well-aligned both in $\LL$ and in $\KK$. The condition in \eqref{compact alignment relation} is a  generalization of the notion of compact alignment for product systems of $C^*$-correspondences from  \cite[Definition 5.7]{F99}, cf. \cite{bls2}, and see \cite{kwa-larII} for details. The next lemma shows that \eqref{compact alignment relation} captures more than just \emph{diagonal} fibres $\KK(p,p)$ for $p\in P$.

\begin{lem}\label{lemma on alignment}
Let $\KK$ be a well-aligned ideal in a  right-tensor  $C^*$-precategory $\LL$. For all
 $a\in \KK(p,q)$, $b\in   \KK(s,t) $ we have
\begin{equation}\label{compact alignment relation2}
(a\otimes 1_{q^{-1}r}) (b\otimes 1_{s^{-1}r}) \in \KK(pq^{-1}r,ts^{-1}r)\qquad \textrm{whenever}\quad qP\cap sP=rP.
 \end{equation}
\end{lem}
\begin{proof} By Lemma \ref{about approximate units}  we  have $a=a' a''$ and $b=b'b''$ where $a'\in \KK(p,q)$, $b'\in   \KK(s,t) $ and $a''\in \KK(q,q)$, $b''\in   \KK(s,s)$. If $qP\cap sP=rP$ then by \eqref{compact alignment relation} we have $(a''\otimes 1_{q^{-1}r}) (b''\otimes 1_{s^{-1}r}) \in \KK(r,r)$. Composing this from the left by $(a'\otimes 1_{q^{-1}r})\in \LL(pq^{-1}r,r)$ and from the right by $(b'\otimes 1_{s^{-1}r})\in \LL(r,ts^{-1}r)$ and using that $\KK$ is an ideal in $\LL$ gives  \eqref{compact alignment relation2}.
\end{proof}

The notion of Nica covariance for a representation of a compactly aligned product system of $C^*$-correspondences was introduced in \cite{F99} in the context of quasi-lattice ordered groups, and was extended to right LCM semigroups in \cite{bls2}. Lemma \ref{lemma on alignment} allows us to extend this concept to $C^*$-precategories over right LCM semigroups. In our generalization below, Nica covariance will be imposed in subspaces $\KK(p,q)$ for $p,q\in P$ that are not necessarily diagonal, in the sense that $p$ need not equal $q$.

\begin{defn} Let $\KK$ be a well-aligned ideal in a  right-tensor  $C^*$-precategory $\LL$.
A representation $\Phi:\KK\to B$ of  $\KK$  in a $C^*$-algebra $B$ is \emph{Nica covariant} if for all
 $a\in \KK(p,q)$, $b\in   \KK(s,t) $ we have
\begin{equation}\label{Nica covariance}
\Phi(a)\Phi(b)
=\begin{cases}
\Phi \left((a \otimes 1_{q^{-1}r}) (b\otimes 1_{s^{-1}r})\right)  & \textrm{ if } qP\cap sP=rP \textrm{ for some } r\in P,
\\
0 & \textrm{ otherwise}.
\end{cases}
 \end{equation}
\end{defn}
\begin{rem}\label{Wick is sick}
If  $\Phi$ is a Nica covariant representation of a well-aligned ideal $\KK$, then   $\Phi$ is a right-tensor representation and moreover the space
\begin{equation}\label{black star algebra}
C^*(\Phi(\KK))^0:=\spane\{\bigcup_{p,q \in P} \Phi(\KK(p,q))\}
\end{equation}
is a dense $*$-subalgebra of $C^*(\Phi(\KK))$; it is clearly
closed under taking adjoints and it 
is closed under multiplication
 by \eqref{Nica covariance}. Hence $C^*(\Phi(\KK))=\clsp\{\bigcup_{p,q \in P} \Phi(\KK(p,q))\}$.
\end{rem}
%\begin{rem}\label{ }

Since a  right-tensor $C^*$-precategory $\LL$ is well-aligned in itself we may 
always
talk about Nica covariant representations of $\LL$. For every Nica covariant representation $\Phi:\LL\to B$ of $\LL$ and every well-aligned ideal $\KK$ in $\LL$ the  restriction  $\Phi:\KK\to B$ is a Nica covariant representation of $\KK$. Moreover, \eqref{Nica covariance} readily implies that $C^*(\Phi(\KK))$ is a $C^*$-subalgebra of $C^*(\Phi(\LL))$, and if $\KK\otimes 1\subseteq \KK$, then   $C^*(\Phi(\KK))$ is in fact an ideal in $C^*(\Phi(\LL))$.
%\end{rem}

Since the element $r$ in the right hand side of \eqref{Nica covariance} is determined only up to invertible elements in $P^*$, Nica covariant representations behave in a special way with respect to $\otimes 1_x$, $x\in P^*$.
\begin{lem}\label{lemma on automorphic actions on ideals}  Let $\KK$ be a well-aligned ideal in a  right-tensor  $C^*$-precategory $\LL$.
For every  $x\in P^*$,  $\otimes 1_x$  is an automorphism of $\LL$ which maps $\KK(p,q)$ onto $\KK(px,qx)$ for every $p,q\in P$, thus it
restricts to an automorphism of   $\KK$.

 Moreover,
 if $x\in P^*$  and $\Phi:\KK\to B$ is a Nica covariant representation of $\KK$,
then  $\Phi(a)=\Phi(a\otimes 1_x)$ for all $a\in \K(p,q)$, $p,q\in P$.
\end{lem}
\begin{proof} Let $x\in P^*$.  Clearly, $\otimes 1_x$  is an automorphism of $\LL$ since $\otimes 1_{x^{-1}}$ acts as an inverse. Take $a\in \KK(p,q)$. As in the proof of Lemma \ref{lemma on alignment}, write  $a=a' a''$ where $a'\in \KK(p,q)$ and $a''\in \KK(q,q)$. Since $qP=qxP$, by \eqref{compact alignment relation2} we get $a\otimes 1_x= (a' a'')\otimes 1_x=(a'\otimes 1_x) (a''\otimes 1_x)\in \KK(px,qx)$. If $\Phi:\KK\to B$ is a Nica covariant representation of $\KK$,
using \eqref{Nica covariance} we get $\Phi(a)=\Phi(a')\Phi(a'')=\Phi((a'\otimes 1_x) (a''\otimes 1_x))= \Phi(a\otimes 1_x)$.
\end{proof}
Existence (and uniqueness) of the universal $C^*$-algebra described in the following proposition can be shown as in the proof of Proposition \ref{Toeplitz description}, or by considering a quotient of the Toeplitz algebra $\TT_\LL(\KK)$. Injectivity of the universal Nica covariant representation follows from  Proposition \ref{Nica Toeplitz reduced representation} below.
\begin{prop}\label{Nica Toeplitz description}  Let $\KK$ be a well-aligned ideal in a  right-tensor  $C^*$-precategory $\LL$. There are a $C^*$-algebra $\NT_{\LL}(\KK)$ and
an injective Nica covariant representation $i_{\KK}: \KK \to \NT_{\LL}(\KK) $, such that
\begin{itemize}
\item[(a)] for every Nica covariant  representation $\Phi$ of $\KK$ there is a homomorphism  $\Phi\rtimes P$
of $\NT_{\LL}(\KK)$ such that $(\Phi\rtimes P)\circ i_{\KK} =\Phi$; and
\item[(b)] $\NT_{\LL}(\KK)=C^*(i_{\KK}(\KK))$.
\end{itemize}
The $C^*$-algebra  $\NT_{\LL}(\KK)$ is unique up to canonical isomorphism.
\end{prop}
\begin{defn}
Given a well-aligned ideal  $\KK$ in a right-tensor $C^*$-precategory $\LL$, the $C^*$-algebra $\NT_{\LL}(\KK)$ described in Proposition \ref{Nica Toeplitz description} is called the \emph{Nica-Toeplitz algebra} of $\KK$. We write $\NT(\LL)$ for the Nica-Toeplitz algebra $\NT_{\LL}(\LL)
$ associated to $\LL$, viewed as a well-aligned ideal in itself.

\end{defn}\begin{rem}\label{remark about inclusions}
The universal property of $\NT_{\LL}(\KK)$ ensures existence of a  homomorphism
\begin{equation}\label{homomorphism to be embedding}
\iota: \NT_{\LL}(\KK) \longmapsto \NT(\LL),\qquad i_\KK(a)\longmapsto i_\LL\vert_\KK(a),
\end{equation}
for $a\in \KK(p,q)$ and $p,q\in P$.
We note that $\iota$ is injective  whenever the universal representation $i_\KK:\KK\to \NT_{\LL}(\KK)$  can be extended to a Nica covariant representation $\overline{i_\KK}:\LL\to B$ where $B$ is a $C^*$-algebra containing $\NT_{\LL}(\KK)$. Indeed, in this case we have $(\overline{i_\KK}\rtimes P)\circ i_\LL=\overline{i_\KK}$, from which it follows that $(\overline{i_\KK}\rtimes P)\circ \iota=\operatorname{id}_{\NT_{\LL}(\KK)}$, showing the claimed injectivity.
We will explore these issues in more detail in Section~\ref{relationship-K-L}.
\end{rem}
\begin{rem}\label{remark about abstract characterization NTalg} If $\KK\otimes 1\subseteq \KK$, then $\NT_{\LL}(\KK)$ does not depend on $\LL$. Indeed, in this case, the definition of Nica covariant representations involves only elements of $\KK$. Therefore, with $\NT(\KK)$ denoting the Nica-Toeplitz algebra of the right-tensor $C^*$-category $\KK$, we have $\NT(\KK)=\NT_{\LL}(\KK)$.   In general, $\NT_{\LL}(\KK)$ depends only on the $C^*$-precategory structure of $\KK$ equipped with a family of mappings $\{N_r\}_{r\in P}$ where  $N_r$, $r\in P$, is defined for  quadruples $p,q,s,t\in P$ such that $qP\cap sP=rP$
by the formula
$$
\KK(p,q)\times \KK(s,t) \ni (a,b)\longmapsto N_r(a,b):=(a\otimes 1_{q^{-1}r}) (b\otimes 1_{s^{-1}r}) \in \KK(pq^{-1}r,ts^{-1}r).
$$
Note that $N_r(a,b)^*=N_r(a^*,b^*)$ and  $N_e(a,b)=a b$   if $q=s$.
%$$ N_r(a,b)^*=N_r(a^*,b^*),\qquad \text{ and }\qquad N_e(a,b)=a b\,  \text{ if } \,q=s. $$
Moreover, for  any $c\in \KK(u,w)$ where $uP\cap ts^{-1}rP=zP$, for some $z\in P$, we have
$$
N_z(N_r(a,b),c)=N_{y^{-1}z}(a,N_y(b,c))
$$
for any $y\in P$ such that $tP\cap uP =yP$ (we then necessarily have  $qP\cap st^{-1}yP =y^{-1}zP$). Nevertheless, in this paper we will not pursue
this more general intrinsic description of $\NT_{\LL}(\KK)$ for three reasons. Firstly, such a  theory would be technically more  involved. Secondly,
we do not have good examples that require such an approach. Thirdly, the relationship between  $\NT_{\LL}(\KK)$ and $\NT(\LL)$ is interesting in its own right (in the context of Doplicher-Roberts algebras such a relationship was studied,
for instance, in \cite{dpz}, \cite{fmr}, \cite{kwa-doplicher}). The latter problem, in our setting, will be  addressed in Section~\ref{relationship-K-L}.
 \end{rem}

Obviously, the Nica-Toeplitz algebra $\NT_{\LL}(\KK)$  may be viewed as a quotient of the Toeplitz algebra $\TT_{\LL}(\KK)$.
If  every two elements in $P$ are comparable, then every ideal $\KK$ in a right-tensor $C^*$-precategory $\LL$ over $P$ is automatically well-aligned and  right-tensor representations of $\KK$ coincide with Nica covariant representations. Hence  $\NT_{\LL}(\KK)\cong\TT_{\LL}(\KK)$ in this case.

\begin{ex}[The case when $P=\N$]\label{The case when P=N}
	Let $\LL$ be a right-tensor $C^*$-precategory over $\N$ and $\KK$ an ideal in $\LL$. % Any right tensoring on $\LL$ is determined by an endomorphism $\otimes 1:\LL\to \LL$ sending $n$ to $n+1$, as iterating $\otimes 1$ we get the appropriate semigroup $\{\otimes 1_{k}\}_{k\in \N}$.
Due to above discussion $\KK$  is automatically well-aligned, and  a representation $\Phi$ of $\KK$ is Nica covariant if and only if $\Phi$ is a right-tensor representation.
%	Accordingly, for any ideal $\KK$ in $\LL$ we have $\TT_\LL(\KK)=\NT_\LL(\KK)$.
	 For any ideal $\JJ$ in $J(\KK):=\otimes 1^{-1}(\KK)\cap \KK$, $C^*$-algebras $\OO_\LL(\KK, \JJ)$  were introduced  in \cite{kwa-doplicher} as universal $C^*$-algebras with respect to  right-tensor representations $\Phi$ of $\KK$ satisfying $\Phi_{n,m}(a) = \Phi_{n+1,m+1}(a\otimes 1)$ for all $a\in \JJ(n,m)$ and $n,m,\in \N$.
	Therefore,
	$$\TT_\LL(\KK)\cong \NT_\LL(\KK)\cong \OO_\LL(\KK, \{0\})$$
	and every $C^*$-algebra $\OO_\LL(\KK, \JJ)$ is a quotient of $\NT_\LL(\KK)$.
\end{ex}

\begin{ex}[Product systems]\label{The case when LL(p,q)=C} Let $X= \bigsqcup_{p\in P}X_{p}$ be a product system as defined in \cite{F99}.
 In \cite[Section 2.1]{kwa-larII}, cf. the introduction, we associate to $X$ the right-tensor $C^*$-precategory  $\LL_X$. Then $\KK_X=\{\KK(X_p, X_q)\}_{p,q\in P}$ is an essential ideal in $\LL_X$. By \cite[Proposition 2.8]{kwa-larII} we have a natural isomorphism $\TT_{\LL_X}(\KK_X)\cong \TT(X)$ where $\TT(X)$ is the Toeplitz algebra of $X$ defined in \cite{F99}.
%The concept of compact alignment and a \emph{Nica-Toeplitz algebra} $\NT(X)$ associated to $X$ were introduced in  \cite{F99} in the context of quasi-latticed ordered pairs $(G,P)$. They were generalized to right LCM semigroups in \cite{bls2}.
Assume that $P$ is a right LCM semigroup. Then $\KK_X$ is well-aligned if and only if $X$ is compactly aligned. In this case, \cite[Proposition 2.10]{kwa-larII} gives
$$
\NT_{\LL_X}(\KK_X)\cong \NT(X),
$$
where $\NT(X)$ is the \emph{Nica-Toeplitz algebra}  associated to $X$, see \cite{bls2}, \cite{F99}.
In \cite{kwa-larII} we also   analyze a Doplicher-Roberts version $\mathcal{DR}(\NT(X))$  of $\NT(X)$, which by definition is $\NT(\LL_X)$.
\end{ex}

In view of the following lemma we may always assume that a well-aligned ideal $\KK$ in a right-tensor $C^*$-precategory $\LL$ generates $\LL$ as a right-tensor $C^*$-precategory.
\begin{lem}\label{the right tensor precatory generated by K}
For every well-aligned ideal $\KK$ in a right-tensor $C^*$-precategory $\LL$ the spaces
$$
\LL_\KK(p,q)=\clsp\{\KK(s,t)\otimes 1_r: sr=p, tr=q, \text{ for }s,t,r\in P\},\qquad p,q\in P,
$$
define the minimal right-tensor sub-$C^*$-precategory of $\LL$ containing $\KK$. In particular, we have
$
\NT_\LL(\KK)\cong\NT_{\LL_\KK}(\KK).
$
\end{lem}
\begin{proof} Plainly, the family  $\{\LL_\KK(p,q)\}_{p,q\in P}$ is closed under right-tensoring $\otimes 1$ and under taking adjoints. Suppose that $a\otimes{1_r}\in \LL_\KK(p,q)$ and $b\otimes 1_w\in \KK(q,z)$  where $a\in \KK(s,t)$, $b\in \KK(u,v)$, $sr=p$, $tr=q$, $uw=q$, $vw=z$. Then $tP\cap u P=yP$ for some $y\in P$, which implies that $(t^{-1}y)^{-1}r=y^{-1}q$ and $(u^{-1}y)^{-1}w=y^{-1}q$. Since $\KK$ is well-aligned we have $(a\otimes 1_{t^{-1}y})(b\otimes 1_{u^{-1}y})\in \KK(st^{-1}y, vu^{-1}y)$. Using this we obtain
\begin{align*}
(a\otimes{1_r}) (b\otimes 1_w)&=\Big((a\otimes 1_{t^{-1}y}) \otimes 1_{y^{-1}q}\Big) \Big((b\otimes 1_{u^{-1}y}) \otimes 1_{y^{-1}q}\Big)
\\
&=\Big((a\otimes 1_{t^{-1}y})  (b\otimes 1_{u^{-1}y})\Big)\otimes 1_{y^{-1}q}\in  \LL_\KK(p,z).
\end{align*}
Hence  $\{\LL_\KK(p,q)\}_{p,q\in P}$ is a right-tensor sub-$C^*$-precategory of $\LL$. Clearly, it is the smallest sub-$C^*$-precategory of $\LL$ containing $\KK$ and invariant under $\otimes 1$.
\end{proof}

An important role in the theory is played by the following core $C^*$-algebra.
\begin{defn}\label{Definition of cores}
Let $\KK$ be a well-aligned ideal in a  right-tensor  $C^*$-precategory $\LL$. For an arbitrary  Nica covariant representation $\Phi$ of $\KK$ the space
$$
B_e^\Phi:=\clsp\Bigl\{\bigcup_{p \in P} \Phi(\KK(p,p)) \Bigr\}
$$
is a $C^*$-algebra. We call $B_e^\Phi$ the \emph{core $C^*$-subalgebra} of $C^*(\Phi(\KK))$.
\end{defn}
\begin{rem}
 For any Nica covariant representation $\Phi:\KK\to B$ the  core $C^*$-algebra $B_e^\Phi$ is a non-degenerate subalgebra of $C^*(\Phi(\KK))$. Indeed, by Lemma \ref{about approximate units}, every $a\in \KK(p,q)$ can be written as $a=a_p a' a_q$ where $a_p\in \KK(p,p)$, $ a'\in \KK(p,q)$, $a_q\in \KK(q,q)$. Hence $\Phi(a)\in B_e^\Phi C^*(\Phi(\KK))B_e^\Phi$. In particular, every multiplier $m\in M(B_e^\Phi)$ of $B_e^\Phi$ extends via the formula $m\Phi(a):=(m\Phi(a_p)) \Phi (a' a_q)$ to a multiplier of $C^*(\Phi(\KK))$. We will use this embedding in the sequel to identify
$M(B_e^\Phi)$ as a subalgebra of $ M(C^*(\Phi(\KK))).$
\end{rem}

\section{Fock representation}\label{Fock representation section}

Only  for the purposes of this section we fix an \emph{arbitrary} left cancellative semigroup $P$, a right-tensor $C^*$-precategory  $(\LL, \{\otimes 1_r\}_{r\in P})$   and an  \emph{arbitrary} ideal $\KK$ in $\LL$. We will construct a  canonical  right-tensor representation  of $\LL$ associated to $\KK$, whose restriction to $\KK$ is injective.
This construction will proceed in two steps. First we associate a representation to each $t\in P$, regarded as a fixed source, and then we consider the direct sum of these representations as we vary $t$.

We fix $t\in P$. For each $s\in P$ the space $X_{s,t}:=\KK(s,t)$ is naturally equipped with a  structure of a right Hilbert module over $A_{t}:=\KK(t,t)$ given by
$$
 x \cdot a:=xa,\quad \langle x, y\rangle:=x^*y, \qquad  x,y \in X_{s,t},\,\, a\in A_t.
$$
Thus we may consider the following direct sum right Hilbert $A_t$-module:
$$
\FF_{\KK}^{t}:=\bigoplus_{s\in P} X_{s,t}.
$$
We will construct  representations of $\LL$ using the maps defined in the following lemma.
\begin{lem}\label{lemma about Fock operators}
 Let $a \in \LL(p,q)$ for $p,q\in P$. For each $s\in  qP$ there is a well defined operator $T_{p,q}^{s,t}(a)\in \LL(X_{s,t},X_{pq^{-1}s,t})$ given by
$$
T_{p,q}^{s,t}(a)x:=(a \otimes 1_{q^{-1}s})  x,\qquad x \in X_{s,t}.
$$
The adjoint is $T_{q,p}^{pq^{-1}s,t}(a^*)\in \LL(X_{pq^{-1}s,t},X_{s,t})$.
\end{lem}
\begin{proof}
Let $x\in X_{s,t}$ and  $y\in X_{pq^{-1}s,t}$. Clearly, $(a \otimes 1_{q^{-1}s})  x\in X_{pq^{-1}s,t}=\KK(pq^{-1}s,t)$, and
\begin{align*}
\langle T_{p,q}^{s,t}(a)x, y\rangle &= \langle (a \otimes 1_{q^{-1}s}) x, y\rangle= x^* (a^* \otimes 1_{q^{-1}s})y=\langle x, (a^* \otimes 1_{q^{-1}s})y\rangle
\\
&
=\langle x, T_{q,p}^{pq^{-1}s,t}(a^*)y\rangle.
\end{align*}
\end{proof}
Let  $a \in \LL(p,q)$ for $p,q\in P$. By Lemma \ref{lemma about Fock operators}, under the obvious  identifications of the Hilbert modules $X_{s,t}$, $s\in P$, with the corresponding submodules of $\FF_\KK^t$, we have
$
T_{p,q}^{s,t}(a)\in \LL(X_{s,t},X_{pq^{-1}s,t})\subseteq \LL(\FF_\KK^t).
$
Since $\|T_{p,q}^{s,t}(a)\|\leq \|a\|$ and the map $qP\ni s \to pq^{-1}s \in pP$ is a bijection, the direct sum
\begin{equation}\label{Toeplitz representation definition2}
\overline{T}_{p,q}^t(a):=\bigoplus_{s\in qP}T_{p,q}^{s,t}(a)\,\,
\end{equation}
 is an operator in $\LL(\FF_\KK^t)$ with norm bounded by $\|a\|$. In other words, we get a contractive mapping $\overline{T}_{p,q}^t: \LL(p,q) \to \LL(\FF_\KK^t)$, which  satisfies
\begin{equation}\label{Toeplitz representation definition for t}
\overline{T}_{p,q}^t(a)x =\begin{cases}
(a \otimes 1_{q^{-1}s})  x & \textrm{ if } s\in   qP ,
\\
0 & \textrm{ otherwise},
\end{cases}
\end{equation}
for every $a \in \LL(p,q)$, $s\in P$ and  $x\in X_{s,t}$.
\begin{lem}\label{representations which are summands0}
For each $t \in P$, the family of maps  $\overline{T}^t=\{\overline{T}_{p,q}^t\}_{p,q\in P}$ given by \eqref{Toeplitz representation definition for t}
is a right-tensor representation $\overline{T}^t:\LL\to \LL(\FF_{\KK}^t)$.
\end{lem}
\begin{proof}
By Lemma \ref{lemma about Fock operators} we have  $(T_{p,q}^{s,t}(a))^*=T_{q,p}^{pq^{-1}s,t}(a^*)$ for $a\in \LL(p,q )$, and therefore it follows from  \eqref{Toeplitz representation definition2}
 that $\overline{T}_{p,q}^t(a)^*=\overline{T}_{q,p}^t(a^*)$. Clearly, the maps $\overline{T}_{p,q}^t$, $p,q\in P$ are linear. Moreover, for any  $a\in \LL(r,p)$, $b \in \LL(p,q)$,  $x\in X_{s,t}$, where $r,p,q\in P$ with $s \in  qP$, we have
$$
T_{r,p}^{pq^{-1}s,t}(a)T_{p,q}^{s,t}(b)x= (a \otimes 1_{q^{-1}s})(b \otimes 1_{q^{-1}s})  x=(ab \otimes 1_{q^{-1}s})x =T_{r,q}^{s,t}(ab)x.
$$
Hence $
\overline{T}_{r,p}^t(a)\overline{T}_{p,q}^t(b)=\overline{T}_{r,q}^t(ab)
$ and thus $\overline{T}^t:\LL\to \LL(\FF_{\KK}^t)$ is a representation of $C^*$-precategories.

To see that $\overline{T}^t$ is a right-tensor representation let $a\in \LL(p,q)$ and $b\in   \LL(s,l) $ for $p,q,s,l\in P$ such that $sP\subseteq qP$. Note that if $w\notin lP$ then both $\overline{T}_{p,q}^t(a)\overline{T}_{s,l}^t(b)$ and $\overline{T}_{pq^{-1}s,l}^t((a\otimes 1_{q^{-1}s})b)$ act as zero on   $X_{w,t}$. Assume then that $w\in lP$ and let $x\in X_{w,t}$.   Then $(b \otimes 1_{l^{-1}w})x\in X_{sl^{-1}w,t}$ and $sl^{-1}w\in qP$. Thus we have
\begin{align*}
\overline{T}_{p,q}^t(a)\overline{T}_{s,l}^t(b)x&=\overline{T}_{p,q}^t(a) (b \otimes 1_{l^{-1}w})x= (a  \otimes 1_{q^{-1}sl^{-1}w})(b \otimes 1_{l^{-1}w})x
\\
&=\Big(\big((a \otimes 1_{q^{-1}s})b\big) \otimes 1_{l^{-1}w}\Big)x =\overline{T}_{pq^{-1}s,l}^t((a\otimes 1_{q^{-1}s})b)x.
\end{align*}
Accordingly,  $\overline{T}_{p,q}^t(a)\overline{T}_{s,l}^t(b)$ and $\overline{T}_{pq^{-1}s,l}^t((a\otimes 1_{q^{-1}s})b)$ coincide, and $\overline{T}^t:\LL \to \LL(\FF_{\KK})$ is a right-tensor representation.
\end{proof}

Note that for each  $t\in P$, we may view $\FF_{\KK}^{t}$ as a right Hilbert module over the $C^*$-algebra $A:=\bigoplus_{p\in P}A_p$, where multiplication on the right by an element of the summand $A_p$ for $p\neq t$ is defined to be zero.
We define the \emph{Fock module} of $\KK$ to be  the direct sum Hilbert $A$-module of $\FF_{\KK}^{t}$ as $t\in P$:
$$
\FF_{\KK}:=\bigoplus_{t\in P} \FF_{\KK}^{t}=\bigoplus_{s,t\in P} X_{s,t}.
$$
Accordingly,
$
\FF_{\KK}$ consists of elements $\bigoplus_{s,t\in P} x_{s,t}$ where $x_{s,t}\in \KK(s,t)$, $s,t\in P$, and the element
 $\bigoplus_{t\in P} \Big( \sum_{s\in P} x_{s,t}^* x_{s,t}\Big)$ belongs to the $C^*$-algebraic direct sum $\bigoplus_{t\in P} \KK(t,t)$.
We will  treat the $C^*$-algebraic direct product $\prod_{t\in P}\LL(\FF_{\KK}^{t})$ as a $C^*$-subalgebra of $\LL(\FF_\KK)$.
\begin{prop}\label{Toeplitz reduced representation}

The direct sum of representations  from Lemma \ref{representations which are summands0} as $t$ varies in $P$ yields a    right-tensor representation $\overline{T}:\LL\to \LL(\FF_{\KK})$  determined by the formula
\begin{equation}\label{Toeplitz representation definition}
\overline{T}_{p,q}(a)x =\begin{cases}
(a \otimes 1_{q^{-1}s})  x & \textrm{ if } s\in   qP ,
\\
0 & \textrm{ otherwise},
\end{cases}
\end{equation}
for $a \in \LL(p,q)$, $x\in X_{s,t}$ and $p,q,s,t\in P$. Furthermore,
the restriction of $\overline{T}$ to $\KK$ yields  an injective right-tensor representation $T:\KK\to \LL(\FF_{\KK})$.
%Representation $\overline{T}:\LL\to \LL(\FF_{\KK})$ is injective if and only if $\KK$ is an essential ideal in $\LL$.
\end{prop}
\begin{proof}
It is immediate that the direct sum of right-tensor representations:
$
\overline{T}_{p,q}:=\bigoplus_{t\in P}\overline{T}_{p,q}^{t}=\bigoplus_{s\in qP, t\in P}T_{p,q}^{s,t}$,
$p,q\in P,$
yields a right-tensor representation $\overline{T}:\LL\to \LL(\FF_{\KK})$  which  satisfies \eqref{Toeplitz representation definition}. Let $T:\KK\to \LL(\FF_{\KK})$ be its restriction to $\KK$. For every $p\in P$ the map $T_{p,p}^{p,p}:\KK(p,p)\to \LL(X_{p,p})$ is injective and hence  $T_{p,p}:\KK(p,p)\to \LL(\FF_{\KK})$ is injective. %Thus  $T:\KK\to \LL(\FF_{\KK})$ is injective by Lemma \ref{proposition 1.5}.
\end{proof}

\begin{defn} % Let $\KK$ be a  well-aligned ideal in a right tensor $C^*$-precategory $\LL$.
We call the right-tensor representation $T:\KK\to \LL(\FF_{\KK})$ from  Proposition \ref{Toeplitz reduced representation}  the \emph{Fock representation} of the  ideal $\KK$ in the right-tensor $C^*$-precategory $\LL$.
\end{defn}
\begin{rem}
For each $t\in P$, we may view the restriction $T^t:\KK\to \LL(\FF_{\KK}^t)$ of  $\overline{T}^t$ to $\KK$ as a Fock representation of $\KK$ with fixed source $t$. However, $T^t$ is injective if and only if  for every $a\in \KK(p,p)$ and $p\in P$ there is  $r\in P$ such that  $(a\otimes 1_r)\KK(pr,t)\neq 0$; and this may fail.
\end{rem}

\begin{rem}\label{remark on direct sums of Fock representations}
The Fock representation is the direct sum
$
T=\bigoplus_{t\in P}T^t
$
of \emph{$t$-th Fock representations}  $T^t:\KK\to \LL(\FF_{\KK}^t)$,  $t\in P$. So by projecting, for each $t\in P$,
 we get a surjective homomorphism $h_t:C^*(T(\KK)) \to C^*(T^t(\KK))$, where $h_t\circ T=T^t$. We will show  that $h_e$  is an isomorphism for Fell bundles  (cf. Proposition \ref{Fock t-th representations and reduced objects} below) and for right-tensor $C^*$-precategories arising from compactly aligned product systems, see  \cite{kwa-larII}. Thus our  Fock representation generalizes those for product systems and Fell bundles. 
\end{rem}

The grading of the Fock Hilbert module $\FF_\KK$  yields  natural  conditional expectations. 
To make this explicit we introduce some notation. For $w,t\in P$ we let $Q_w^t\in \LL(\FF_{\KK}^t)$
be  the projection onto $X_{w,t}$, and $Q_w:=\bigoplus_{s\in P}Q_w^s\in \LL(\FF_{\KK})$ be the projection onto $\bigoplus_{s\in P}X_{w,s}$. Note that $Q_w$ projects onto the subspace of $\FF_\KK$ of fixed range  $w$.
\begin{lem}\label{lemma on conditional expectations}
Let $T:\LL\to \LL(\FF_{\KK})$ be the Fock representation of $\KK$.

\textnormal{(a)}  For each $t\in P$, the space $
D^t:=\left\{S\in \LL(\FF_{\KK}^t): Q_w^t S= SQ_w^t \textrm{ for every }w\in P\right\}
$
is a $C^*$-subalgebra of $\LL(\FF_{\KK}^t)$ and the map  $E^t:\LL(\FF_{\KK}^t) \mapsto D^t$ given by
\begin{equation}\label{diagonal t operators algebra}
E^t(S)=\sum_{w\in P} Q_w^t SQ_w^t,\qquad  S\in \LL(\FF_{\KK}^t),
\end{equation}
is a  faithful conditional expectation.

\textnormal{(b)} The space $D:=\left\{S\in \LL(\FF_{\KK}): Q_w S= SQ_w \textrm{ for every }w\in P\right\}$ is a $C^*$-subalgebra of $\LL(\FF_{\KK})$ and the map $E:\LL(\FF_{\KK}) \mapsto D$ given by
\begin{equation}\label{diagonal operators algebra}
E(S)=\sum_{w\in P} Q_wSQ_w, \qquad S\in \LL(\FF_{\KK}),
\end{equation}
is a faithful conditional expectation.
Furthermore,  $E|_{\prod_{t\in P}\LL(\FF_{\KK}^t)}=\bigoplus_{t\in P}E^t$.
\end{lem}
\begin{proof}
 For part (a), clearly $D^t$ is a $C^*$-subalgebra of $\LL(\FF_{\KK}^t)$.  The map $a \mapsto Q_w^t a Q_w^t$ from $\LL(\FF_{\KK}^t)$ to $\LL(X_{w,t})\subseteq \LL(\FF_{\KK}^t)$ is a contractive completely positive map with range in $D^t$. The direct sum of these maps is a well defined contractive completely positive map $E^t:\LL(\FF_{\KK}^t)\to D^t\subseteq \LL(\FF_{\KK}^t)$ which is the identity on $D^t$. Hence  \eqref{diagonal t operators algebra} defines a conditional expectation as claimed. That $E^t$ is faithful follows because if $a\in \LL(\FF_{\KK}^t)$ is such that $E^t(a^*a)=0$ then for every $x\in X_{w,t}$, $w\in P$,
$$
\|ax\|^2=\|\langle ax, ax \rangle\|= \|\langle x, a^*ax \rangle\|=\|\langle x, E^t(a^*a)x \rangle\|=0.
$$
Since the elements  $x\in X_{w,t}$, $w\in P$, span $\FF_{\KK}^t$ it follows that $a=0$. The proof of part (b) is analogous to (a) and is left to the reader.
\end{proof}

\section{Reduced Nica-Toeplitz algebra and the transcendental core}\label{Reduced Nica-Toeplitz algebra section}

In this section, we come back to our standing assumption that $P$ is a right LCM semigroup.   We fix  a right-tensor $C^*$-precategory  $(\LL, \{\otimes 1_r\}_{r\in P})$    and a well-aligned ideal $\KK$ in $\LL$. Under these assumptions, the Fock representation is Nica covariant:

\begin{lem}\label{representations which are summands}
For each $t \in P$, the family of maps  $\overline{T}^t=\{\overline{T}_{p,q}^t\}_{p,q\in P}$ given by \eqref{Toeplitz representation definition for t}
is a Nica covariant representation $\overline{T}^t:\LL\to \LL(\FF_{\KK}^t)$.
\end{lem}
\begin{proof}
By Lemma \ref{representations which are summands0} we only need to show that $\overline{T}^t$ is Nica covariant. Let $a\in \LL(p,q)$ and $b\in   \LL(s,l) $ for $p,q,s,l\in P$. Note that if $w\in lP$ and $x\in X_{w,t}$, then $T_{s,l}^{w,t}(b)$ is in $X_{sl^{-1}w, t}$. Since $T_{p,q}^{u,t}(a)$ acts in $X_{u,t}$, we have
\begin{equation}\label{implication to be used}
T_{p,q}^{u,t}(a)T_{s,l}^{w,t}(b)\neq 0 \,\,\Longrightarrow \,\, u=sl^{-1}w  \text{ and } u\in qP.
\end{equation}
In particular, $T_{p,q}^{u,t}(a)T_{s,l}^{w,t}(b)\neq 0$ implies that  $qP\cap sP= rP$ for some $r\in P$ such that $u\in rP$. Hence,  if $qP\cap sP=\emptyset$, then  $\overline{T}_{p,q}^t(a)\overline{T}_{s,l}^t(b)=0$. Assume now $qP\cap sP=rP$. Using \eqref{implication to be used} and the fact that
$ls^{-1}rP\ni w \to sl^{-1}w\in rP$ is a bijection, we get
\begin{align*}
\overline{T}_{p,q}^t(a)\overline{T}_{s,l}^t(b)&=\bigoplus_{u\in qP}T_{p,q}^{u,t}(a)\,\, \bigoplus_{w\in lP}T_{s,l}^{w,t}(b)\,\,
=\bigoplus_{u\in rP}T_{p,q}^{u,t}(a) \bigoplus_{w\in ls^{-1}rP}T_{s,l}^{w,t}(b)
\\
&= \bigoplus_{w\in ls^{-1}rP}T_{p,q}^{sl^{-1}w,t}(a) T_{s,l}^{w,t}(b).
\end{align*}
Moreover, for every $w\in ls^{-1}rP$ and every $x\in X_{w,t}$ we have
\begin{align*}
T_{p,q}^{sl^{-1}w,t}(a) T_{s,l}^{w,t}(b)x&=(a \otimes 1_{q^{-1}sl^{-1}w})(b \otimes 1_{l^{-1}w})x
\\
&=(a \otimes 1_{q^{-1}r(ls^{-1}r)^{-1}w})(b \otimes 1_{s^{-1}r(ls^{-1}r)^{-1}w})x
\\
&=T_{pq^{-1}r, ls^{-1}r}^{w,t}\left((a \otimes 1_{q^{-1}r})(b \otimes 1_{s^{-1}r})\right)x.
\end{align*}
Accordingly,  $\overline{T}_{p,q}^t(a) \overline{T}_{s,l}^t(b)=T_{pq^{-1}r, ls^{-1}r}^t\left((a \otimes 1_{q^{-1}r})(b \otimes 1_{s^{-1}r})\right)$.
Thus $\overline{T}^t:\LL \to \LL(\FF_{\KK})$ is Nica covariant.
\end{proof}
\begin{prop}\label{Nica Toeplitz reduced representation}

The right-tensor representation  $\overline{T}:\LL\to \LL(\FF_{\KK})$ given by \eqref{Toeplitz representation definition} is Nica covariant.
In particular, the Fock representation  $T:\KK\to \LL(\FF_{\KK})$ is an injective Nica covariant representation.
%Representation $\overline{T}:\LL\to \LL(\FF_{\KK})$ is injective if and only if $\KK$ is an essential ideal in $\LL$.
\end{prop}
\begin{proof} Since a direct sum of Nica covariant representations is a Nica covariant representation, the first part follows  from Lemma \ref{representations which are summands}. The second part follows from Proposition \ref{Toeplitz representation definition} and the fact that restriction of a Nica covariant representation to a well-aligned ideal is Nica covariant.
\end{proof}

\begin{defn} % Let $\KK$ be a  well-aligned ideal in a right tensor $C^*$-precategory $\LL$.
We define the \emph{reduced Nica-Toeplitz algebra} of the well-aligned ideal $\KK$ in the right-tensor $C^*$-precategory $\LL$  to be  the $C^*$-algebra
$$
\NT^{r}_{\LL}(\KK):=C^*(T(\KK)),
$$
 and we call $T \rtimes P:\NT_\LL(\KK)\to \NT^{r}_{\LL}(\KK)$  \emph{the regular representation} of $\NT_{\LL}(\KK)$, cf. Proposition \ref{Nica Toeplitz description}. When $\KK=\LL$, we also write
$
\NT^{r}(\LL):=\NT^{r}_{\LL}(\LL).
$
\end{defn}

It turns out that if $P$ is right cancellative, in particular if it is a  group, then the conditional expectation \eqref{diagonal operators algebra} restricts to a conditional expectation of the $C^*$-algebra $\NT^{r}_{\LL}(\KK)$ onto the core $C^*$-subalgebra  $B_e^T$ of $\NT^{r}_{\LL}(\KK)$. If $P$ is not right cancellative, this is no longer true and in particular $E$ may not preserve the $C^*$-algebra $\NT^{r}_{\LL}(\KK)$.
\begin{prop}\label{Nica Toeplitz conditional expectation}
The conditional expectation  defined by \eqref{diagonal operators algebra} restricts  to a faithful contractive completely positive  map $E^T:\NT^{r}_{\LL}(\KK) \to\LL(\FF_{\KK})$ given by
 \begin{equation}\label{first form of E T}
E^T\Big(\sum_{p,q\in F } T(a_{p,q})\Big)=\sum_{p,q\in F } \bigoplus_{w\in pP\cap qP,t\in P \atop p^{-1}w=q^{-1}w}T_{p,q}^{w,t}(a_{p,q})
\end{equation}
for $F\subseteq P$ finite, $a_{p,q}\in \KK(p,q),\,\, p,q \in F$. The range of $E^T$ is the following self-adjoint operator space:
 \begin{equation}\label{the core space}
B_\KK:=\clsp\biggl\{\bigoplus_{w\in pP\cap qP, t\in P \atop p^{-1}w=q^{-1}w}T_{p,q}^{w,t}(a): a\in \KK(p,q),    p,q \in P\biggr\} \subseteq \LL(\FF_{\KK}).
\end{equation}
If $P$ is cancellative, then
$
B_\KK=\clsp\left\{\bigcup_{p \in P} T_{p,p}(\KK(p,p)) \right\}
$
equals the core $C^*$-subalgebra $B_e^T$ of $\NT^{r}_{\LL}(\KK)$ and  $E^T$ is a faithful conditional expectation onto $B_\KK$ given by the formula
\begin{equation}\label{another form of E T}
E^T\Big(\sum_{p,q\in F } T(a_{p,q})\Big)=\sum_{p\in F } T(a_{p,p}) ,
\end{equation}
for  all $a_{p,q}\in \KK(p,q)$,  $p,q \in F$ and $F\subseteq P$ finite.
\end{prop}

\begin{proof} The crucial observation is the following claim: for $a\in \LL(p,q)$, $w\in qP$ and $p,q,t\in P$ we have
\begin{equation}\label{crucial-observation-about-T}
  Q_w T_{p,q}^{w,t}(a)Q_w =\begin{cases}
T_{p,q}^{w,t}(a)  & \textrm{ if } w\in pP\cap qP \textrm{ and } p^{-1}w=q^{-1}w
\\
0 & \textrm{ otherwise}.
\end{cases}
\end{equation}
To see this, recall that $T_{p,q}^{w,t}(a)$ acts as an adjointable operator from $X_{w,t}$ to $X_{pq^{-1}w, t}$. Since $Q_w$ is the projection onto fixed range $w$, the only possibility to have $Q_w T_{p,q}^{w,t}(a)Q_w$ nonzero is that $w=pq^{-1}w$, in which case it equals $T_{p,q}^{w,t}(a)$. Thus to prove \eqref{crucial-observation-about-T} it remains to see that if $p,q\in P$ with $w\in qP$, then $w=pq^{-1}w$ is equivalent to $w\in pP\cap qP \textrm{ and } p^{-1}w=q^{-1}w$. However,  the non-trivial left to right implication follows since $w\in pP$ implies $w=ps$ for some $s\in P$ and so left cancellation gives $s=q^{-1}w=p^{-1}w$.

Let now $t\in P$, $F\subseteq P$ finite, $a_{p,q}\in \KK(p,q)$ for $p,q\in F$. By \eqref{crucial-observation-about-T}, $E(T_{p,q}(a_{p,q}))=\{0\}$ when $pP\cap qP=\emptyset$, and if $pP\cap qP\neq \emptyset$, then
$$
E(T_{p,q}(a_{p,q}))
=E\left(\bigoplus_{s\in qP, t\in P}T_{p,q}^{s,t}(a_{p,q})\,\, \right) =\bigoplus_{w\in pP\cap qP, t\in P \atop p^{-1}w=q^{-1}w}T_{p,q}^{w,t}(a_{p,q}).
$$
Thus $E^{T}=E|_{\NT_{\LL}^{r}(\KK)}$ maps $\NT_{\LL}^{r}(\KK)$ onto $B_\KK$ according to the formula \eqref{first form of E T}.

Now suppose that $P$ is right cancellative.  Note that
$w\in pP\cap qP$ and $p^{-1}w=q^{-1}w$ if and only if $p=q$ and  $w\in pP=qP$, where the non-trivial left to right implication follows  upon invoking right cancellation in $w=p (p^{-1}w)= q (q^{-1}w)$.   By using this observation one sees that \eqref{first form of E T} reduces to  \eqref{another form of E T} and  $B_\KK=B_e^T$. Clearly,  $E^T$ is an idempotent map and therefore   a conditional expectation onto $B_\KK$.
\end{proof}

\begin{defn}\label{Definition for transcendental cores}
Let $\KK$ be a well-aligned ideal in a right-tensor $C^*$-precategory $\LL$. We call the space $B_\KK$ given by   \eqref{the core space} the \emph{transcendental core} for $\KK$, and the map $E^T$ given by \eqref{first form of E T}
the \emph{transcendental conditional expectation} from $\NT^{r}_{\LL}(\KK)$ onto $B_\KK$.
\end{defn}
\begin{rem} The transcendental core $B_\KK$ always contains the core $C^*$-subalgebra $B_e^T$, see Definition \ref{Definition of cores}.
For semigroups $P$ that are not right cancellative, we  may have  $B_e^T\subsetneq B_\KK$, and then  it is not clear whether there is a conditional expectation from $\NT_{\LL}^{r}(\KK)$ onto $B_e^T$.
\end{rem}

\begin{rem}\label{remark about transcendentals for teeth} In view of Lemma \ref{lemma on conditional expectations} and Proposition \ref{Nica Toeplitz conditional expectation} one sees that for each $t\in P$, the conditional expectation defined by \eqref{diagonal t operators algebra} restricts to  a faithful contractive completely positive  map $E^{T,t}$ on $C^*(T^t(\KK))$ with range the subspace of $B_\KK$ where $t$ is fixed. In fact, we have
$
 E^{T}=\bigoplus_{t\in P}E^{T,t}.
$
\end{rem}

The semilattice of projections introduced in the following lemma is one of the key tools to analyze the structure of the reduced Nica-Toeplitz algebra.
%The Fock representation is equipped with a semilatice of projections isomorphic to $\JJ(P)$
\begin{lem}\label{lemma-proj-T-semillatice}
For each $p\in P$, let $Q^T_{\langle p\rangle} \in \LL(\FF_\KK)$  be the projection
$$
Q^{T}_{\langle p\rangle } \Big(\bigoplus_{s,t\in P} x_{s,t}\Big) :=\bigoplus_{s\in pP, t\in P} x_{s,t}, \qquad\quad  \bigoplus_{s,t\in P} x_{s,t}\in \FF_\KK.
$$
The  assignment $pP \mapsto Q^T_{\langle p\rangle}$ and $\emptyset \mapsto  0$ forms a semilattice homomorphism $J(P) \longmapsto \Proj(\LL(\FF_\KK))$, meaning that
\begin{equation}\label{eq:definition-semilattice-homom-T}
  Q^T_{\langle p\rangle}Q^T_{\langle q\rangle}=\begin{cases}Q^T_{\langle r\rangle},&\text{ if }pP\cap qP=rP\text{ for some }r\in P\\
  0,&\text{ if }pP\cap qP=\emptyset \end{cases}
\end{equation}
for all $p,q\in P$. In particular, we have $\JJ(P)\cong \{Q^T_{\langle p\rangle}: p\in P\}\cup\{0\}$.
\end{lem}
\begin{proof}
The proof is immediate from the definition of $Q^T_{\langle p\rangle}$.
\end{proof}
\begin{lem}  Let  $\overline{T}:\LL\to \LL(\FF_{\KK})$  be the  representation given  by \eqref{Toeplitz representation definition}.
Then the projections introduced in Lemma~\ref{lemma-proj-T-semillatice}  satisfy the relation:
\begin{equation}\label{relation for projections associated to regular rep}
\overline{T}(a)Q_{\langle p\rangle }^T=\begin{cases}
\overline{T}(a\otimes 1_{q^{-1}w})  & \textrm{ if } qP\cap pP=wP \textrm{ for some } w\in P,
\\
0 & \textrm{ otherwise},
\end{cases}
\end{equation}
for all $a\in\LL(r,q)$,  $r,p,q\in P$.
\end{lem}
\begin{proof}
Let  $a\in\LL(r,q)$ and $x_{u,v}\in X_{u,v}$ where $u,v,r,q\in P$. Let $p\in P$. If $\overline{T}(a)Q_{\langle p\rangle }^T x_{u,v}\neq 0$ then by the definition of $Q_{\langle p\rangle }^T$ and \eqref{Toeplitz representation definition} we necessarily have  that  $u\in  pP$ and $u\in qP$, which implies that
  $qP\cap pP=wP$  for some $w\in P$ where $u\in wP$. Assume then that  $qP\cap sP=wP$  for some $w\in P$. Note that $\overline{T}(a\otimes 1_{q^{-1}w})x_{u,v}\neq 0$ implies that $u\in wP$. Thus both sides of \eqref{relation for projections associated to regular rep} are zero when $u\notin wP$. It remains to verify that equality holds when $u\in wP$. This follows from two applications of \eqref{Toeplitz representation definition}:
\begin{align*}
\overline{T}(a)x_{u,v}&\stackrel{\eqref{Toeplitz representation definition}}{=}(a\otimes 1_{q^{-1}u})x_{u,v}=(a\otimes 1_{q^{-1}(w(w^{-1}u))})x_{u,v}
\\
&=
\Big((a\otimes 1_{q^{-1}w}) \otimes 1_{w^{-1}u}\Big) x_{u,v}  \stackrel{\eqref{Toeplitz representation definition}}{=}\overline{T}(a\otimes 1_{q^{-1}w})x_{u,v}.
\end{align*}
\end{proof}

\begin{lem}\label{properties of semilattice projections in the reduced}
The projections $\{Q^{T}_{\langle p\rangle}\}_{p\in P}$ may be treated as multipliers of both $C^*(\overline{T}(\LL))$ and $B_e^{\overline{T}}$. If $\KK\otimes 1 \subseteq \KK$ in the sense of Definition~\ref{well-alignment definition}, then  $\{Q^{T}_{\langle p\rangle }\}_{p\in P}$ may be treated as multipliers of both $\NT_{\LL}^{r}(\KK)$ and $B_e^T$.
\end{lem}
\begin{proof}
By  Lemma \ref{about approximate units} we have $T(\KK(p,p))X_{p,t}=X_{p,t}$ for all $p,t\in P$, which implies that $B_e^T$ and therefore also $\NT_{\LL}^{r}(\KK)$, $C^*(\overline{T}(\LL))$ and $B_e^{\overline{T}}$ act on  $\FF_\KK$ in a non-degenerate way. Using \eqref{relation for projections associated to regular rep} we see that $C^*(\overline{T}(\LL))Q_{\langle s\rangle }^{T}\subseteq C^*(\overline{T}(\LL))$ and $B_e^{\overline{T}}Q_{\langle p\rangle }^{T}\subseteq B_e^{\overline{T}}$. Thus, since $Q_{\langle p\rangle }^{T}$ is self-adjoint, we may treat $Q_{\langle p\rangle }^{T}$ as a multiplier of both $C^*(\overline{T}(\LL))$ and $B_e^{\overline{T}}$, cf. \cite[Proposition 2.3]{lance}. If $\KK\otimes 1 \subseteq \KK$, then \eqref{relation for projections associated to regular rep} implies that $\NT_{\LL}^{r}(\KK)Q_{\langle p\rangle }^{T}\subseteq \NT_{\LL}^{r}(\KK)$ and $B_e^{T}Q_{\langle p\rangle }^{T}\subseteq B_e^{T}$, which finishes the proof.
\end{proof}

We end this section by establishing certain norm formulas for elements in the transcendental core  $B_\KK$ which can be considered far reaching generalizations of similar formulas obtained in \cite[Lemma 7.4]{F99}.

We begin by introducing some notation, which is inspired by \cite[Remark 1.5]{LR}
and \cite[Remark 5.2]{FR}, where  quasi-ordered groups were considered, cf. also \cite{F99} and \cite{bls}. Suppose that $C$ is a finite subset of $P$. We put  $\sigma(C):=e$ if $C=\emptyset$. If $C$ is non-empty, then either $\bigcap_{c\in C}cP=\emptyset$ or $\bigcap_{c\in C}cP=c'P$ for an element $c'\in P$  (determined by $C$ up to multiplication from the right by elements of $P^*$). In  the latter case we write  $\sigma(C)=c'$.
 Let $F$ be a finite subset of $P$. A subset $C$ of $F$ is an
\emph{initial segment} of $F$ if $\sigma(C)$ exists in $P$ and   $C=\{t\in F: t\leq \sigma(C)\}$. We denote by $\In(F)$ the collection of all initial segments of $F$. For each $C\in \In(F)$,  the  set
$$
P_{F,C} := \{t\in P:  \sigma(C) \leq t \textrm{ and } f \not\leq t \textrm{ for all }f\in F\setminus C \}
$$
is non-empty. Note that neither definition of initial segment nor of the set $P_{F,C}$ depends on the choice of $\sigma(C)$.
Moreover, $\{P_{F,C}: C\in \In(F)\}$ form a
decomposition of the set $P$; in particular, $s\in P_{F,C}$ if and only if $C=\{t\in F: t\leq s\}$
(the latter set is in $\In(F)$ for every $s\in P$). Now, for every finite set $F\subseteq P$ and any $C\in \In(F)$,
$$
Q_{F,C}^T:=Q_{\langle {\sigma(C)}\rangle }^T \prod_{s\in F\setminus C}(1-Q_{\langle s\rangle }^T)
$$
are mutually orthogonal projections that sum up to the identity in $\LL(\FF_\KK)$ as $C$ varies. We use this partition of the identity to prove the following lemma.
\begin{lem}\label{norm of an element in the core}
Suppose that $\KK$ is a well-aligned ideal in a right-tensor $C^*$-precategory $\LL$ and consider an element
\begin{equation}\label{definition of Z}
Z=\sum_{p,q\in F }\bigoplus_{w\in pP\cap qP, t\in P \atop p^{-1}w=q^{-1}w }T_{p,q}^{w,t}(a_{p,q}) \in B_\KK,
\end{equation}
where  $F\subseteq P$ is a finite set and $a_{p,q}\in \KK(p,q)$ for $p,q\in F$. Then
\begin{equation}\label{norm of Z formula}
\|Z\|=\max_{C \in \In(F)} \sup_{w\in    P_{F,C}  } \Big\| T_{w,w}^{w,w}(\sum_{p,q\in C \atop p^{-1}w=q^{-1}w} a_{p,q}\otimes 1_{q^{-1}w})\Big\|.
\end{equation}
\end{lem}

\begin{proof} %[Proof of Lemma~\ref{norm of Z formula}]
 For any projection $Q_{\langle r\rangle }^T$, $r\in P$, and every $a\in \KK(p,q)$, $p,q\in P$, we have
$$
\bigoplus_{w\in pP\cap qP,t\in P \atop p^{-1}w=q^{-1}w }T_{p,q}^{w,t}(a)Q_{\langle r\rangle }^T = \bigoplus_{w\in pP\cap qP\cap rP ,t\in P\atop p^{-1}w=q^{-1}w }T_{p,q}^{w,t}(a)=Q_{\langle r\rangle }^T\bigoplus_{w\in pP\cap qP ,t\in P\atop p^{-1}w=q^{-1}w }T_{p,q}^{w,t}(a).
$$
Thus the projections  $Q_{F,C}^T$, where $C\in \In(F)$, commute with $Z$. Since they form a partition of identity and
\begin{equation}\label{positivity of a summand of Z}
ZQ_{F,C}^T =\sum_{p,q\in C}
\bigoplus_{w\in  P_{F,C},  t\in P \atop  p^{-1}w=q^{-1}w }T_{p,q}^{w,t}(a_{p,q})
=\bigoplus_{w\in    P_{F,C}, t\in P} \sum_{p,q\in C \atop p^{-1}w=q^{-1}w} T_{p,q}^{w,t}(a_{p,q})
\end{equation}
we get
\begin{align*}
\|Z\|&=\max \Big\{ \Big\|\bigoplus_{w\in   P_{F,C}, t\in P} \sum_{p,q\in C \atop p^{-1}w=q^{-1}w} T_{p,q}^{w,t}(a_{p,q})
\Big\|: C \in \In(F) \Big\} \nonumber
\\
&=\max_{C \in \In(F)} \sup_{w\in    P_{F,C} , t\in P} \Big\|\sum_{p,q\in C \atop p^{-1}w=q^{-1}w} T_{p,q}^{w,t}(a_{p,q})\Big\|.
\end{align*}
Now, in the summation above over $p,q\in C$ we have $a_{p,q}\otimes 1_{q^{-1}w}\in \LL(w,w)$, and therefore
\begin{equation}\label{line with supremum}
 \|Z\|= \max_{C \in \In(F)} \sup_{w\in    P_{F,C}, t\in P} \Big\|T_{w,w}^{w,t}\Big(\sum_{p,q\in C \atop p^{-1}w=q^{-1}w} a_{p,q}\otimes 1_{q^{-1}w}\Big)\Big\|.
\end{equation}
Next note that for any $w,t\in P$ and $a\in \LL(w,w)$ we have $\ker T_{w,w}^{w,w}\subseteq \ker T_{w,w}^{w,t}$ due to
\begin{align}
a\in \ker T_{w,w}^{w,w}\, & \Longleftrightarrow\,  a \KK(w,w)=\{0\} \, \Longrightarrow\,   a \KK(w,t)\KK(t,w)=\{0\} \label{line with kernels}
\\
&  \Longleftrightarrow\,  a \KK(w,t)=\{0\}     \, \Longleftrightarrow\,  a\in \ker T_{w,w}^{w,t}.\nonumber
\end{align}
Therefore, the supremum in \eqref{line with supremum} is attained for $t=w$.  This proves \eqref{norm of Z formula}.
\end{proof}
\begin{cor}\label{cor:norm of Z formula}
With the assumptions from Lemma~\ref{norm of an element in the core}, if
 $\KK$ is essential in the right-tensor sub-$C^*$-precategory $\LL_\KK$ of $\LL$ generated by $\KK$, cf. Lemma \ref{the right tensor precatory generated by K}, then
 \eqref{norm of Z formula} reduces to
\begin{equation}\label{norm of Z formula2}
\|Z\|=\max_{C \in \In(F)} \sup_{w\in    P_{F,C} } \Big\| \sum_{p,q\in C \atop p^{-1}w=q^{-1}w} a_{p,q}\otimes 1_{q^{-1}w}\Big\|.
\end{equation}
If, additionally, $Z=\sum_{p\in F }T_{p,p}(a_{p,p})\in B_e^T$, then
\begin{equation}\label{norm of Z formula3}
\|Z\|=\max \Big\{ \Big\|\sum_{p\in C} a_{p,p}\otimes 1_{p^{-1}\sigma(C)}
\Big\|: C \in \In(F) \Big\}.
\end{equation}
\end{cor}
\begin{proof}%[Proof of Corollary~\ref{cor:norm of Z formula}]
Note that elements of the form $\sum_{p,q\in C \atop p^{-1}w=q^{-1}w} a_{p,q}\otimes 1_{q^{-1}w}$ belong to  $\LL_\KK$. If $\KK$ is essential in $\LL_\KK$  then the maps $T_{w,w}^{w,w}$ are injective  on  $\LL_\KK(w,w)$, and therefore \eqref{norm of Z formula} reduces to \eqref{norm of Z formula2}. If, in addition, $Z=\sum_{p\in F }T_{p,p}(a_{p,p})\in B_e^T$, then \eqref{norm of Z formula2} reduces to \eqref{norm of Z formula3} because  for every $C\in \In(F)$ the supremum
$\sup_{w\in  \sigma(C)P} \|\sum_{p\in C} a_{p,p}\otimes 1_{p^{-1}w}\|$ is attained  at $w=\sigma(C)$.
\end{proof}

\section{Faithfulness of representations on core $C^*$-subalgebras}\label{section: faithful on the core}
The goal of this section is to prove Theorem \ref{theorem for amenability and spectral subspaces}, which  is inspired by certain results  used to obtain  amenability criteria, cf.  \cite[Lemma 4.1]{LR}, \cite[Theorem 6.1]{FR}, \cite[Theorem 8.1]{F99}.  In this section we use it to detect necessary and sufficient conditions for injectivity of representations $\Phi\rtimes P$ on the core  $C^*$-subalgebra $B_e^{i_{\KK}} \subseteq \NT_{\LL}(\KK)$. We will  apply Theorem \ref{theorem for amenability and spectral subspaces}, in its full force,  in  Section \ref{amenability section} where we discuss the problem of amenability.

\begin{thm}\label{theorem for amenability and spectral subspaces}
Let $\KK$ be a well-aligned ideal in a right-tensor $C^*$-precategory $\LL$ over a right LCM semigroup $P$. Suppose that $\theta:P\to \P$ is a controlled map of right LCM semigroups, cf. Definition \ref{def:controlled map}. %The following hold:

\textnormal{(a)} The subspace of $\NT_{\LL}(\KK)=C^*(i_\KK(\KK))$ defined as
\begin{equation}\label{core like algebra}
\clsp\Big\{ i_\KK( \KK(p,q)): p,q\in P,\theta(p)=\theta(q)\Big\}
\end{equation}
is a $C^*$-subalgebra of $\NT_{\LL}(\KK)$ on which  the regular representation $T\rtimes P$ is faithful.

\textnormal{(b)}
 If $\Phi$ is a Nica covariant representation of $\KK$, the representation $\Phi\rtimes P$ is faithful on \eqref{core like algebra} if and only if $\Phi$ is injective and satisfies
\begin{equation}\label{Toeplitz like condition}
\begin{array}{l}
 \clsp\{\Phi(\KK(p,q)): \theta(p)=\theta(q)=u\}\cap\clsp\{\Phi(\KK(s,t)): \theta(s)=\theta(t)\in F\} =\{0\}
\\[4pt]
\text{for all $u\in \P$  and all finite sets $F\subseteq \P$  such that  $u\not\geq v$ for every $v\in F$.}
\end{array}
\end{equation}

\end{thm}

Before we pass to the proof of  Theorem \ref{theorem for amenability and spectral subspaces} let us derive some consequences in the case when the  homomorphism $\theta$ is the identity map. To facilitate the discussion we introduce a name for the condition in \eqref{Toeplitz like condition} in the case $\theta$ is injective.
\begin{defn}
 A Nica covariant representation $\Phi:\KK\to B$ is  \emph{Toeplitz covariant} if
\begin{equation}\label{Toeplitz condition}
\begin{array}{l}
 \text{for every   $p\in P$  and $q_1,...,q_n\in P$  such that  $p\not\geq q_i$ for all $i=1,...,n$, $n\in \N$,}
\\[4pt]
\text{we have }\Phi(\KK(p,p))\cap \clsp\{\Phi(\KK(q_i,q_i)):i=1,...,n\} =\{0\}.
\end{array}
\end{equation}
In short we will call such representations  \emph{Nica-Toeplitz covariant}  representations of $\KK$.
\end{defn}
\begin{cor}\label{Nica-Toeplitz representation corollary}
Let $\Phi:\KK\to B$ be a Nica covariant representation. The representation  $\Phi\rtimes P$ of $\NT_{\LL}(\KK)$ restricts to an isomorphism
$$
(\Phi\rtimes P)\vert_{B_e^{i_\KK}}: B_e^{i_\KK} {\overset{\cong}{\longrightarrow}} B_e^\Phi
$$
if and only if $\Phi$ is  injective  and Toeplitz covariant.
\end{cor}
\begin{proof}
Applying Theorem \ref{theorem for amenability and spectral subspaces} with $\theta=\id$, we see that the $C^*$-algebra in \eqref{core like algebra} equals $B_e^{i_\KK}$, and condition \eqref{Toeplitz like condition} collapses to \eqref{Toeplitz condition}.
\end{proof}
%In the case $P=\N$, Corollary \ref{Nica-Toeplitz representation corollary} follows from \cite[Theorem 7.3]{kwa-doplicher}.
\begin{cor}\label{the reduced core}
The Fock representation $T:\KK\to \LL(\FF_\KK)$ is  an injective Nica-Toeplitz covariant representation of $\KK$.  Equivalently, the regular representation $T\rtimes P$ restricts to an isomorphism of core subalgebras $(T\rtimes P)\vert_{B_e^{i_\KK}}: B_e^{i_\KK} {\overset{\cong}{\longrightarrow}} B_e^T.$
\end{cor}
\begin{proof}
Apply Theorem \ref{theorem for amenability and spectral subspaces} with $\theta=\id$.
\end{proof}

\begin{cor}\label{core for cancellative semigroups}
If   $P$ is cancellative, then the formula
\begin{equation}\label{conditional expectation for the universal}
E\Big(\sum_{p,q\in F }i_{\KK}(a_{p,q})\Big)=\sum_{p\in F }i_{\KK}(a_{p,p}) , \qquad a_{p,q}\in \KK(p,q), p,q \in F\subseteq P,
\end{equation}
defines a conditional expectation $E:\NT_{\LL}(\KK)\to B_e^{i_\KK}\subseteq \NT_{\LL}(\KK)$.
\end{cor}
\begin{proof}
Since $P$ is cancellative,  $B_\KK=B_e^T$. By Corollary \ref{the reduced core},  $T\rtimes P$ restricts to an isomorphism of $B_e^{i_\KK}$ onto $B_e^T$. Denote by $(T\rtimes P)^{-1}$ the inverse to this isomorphism.  Taking into account \eqref{another form of E T} one sees that   $E=(T\rtimes P)^{-1} \circ E^T \circ (T\rtimes P)$ is a conditional expectation satisfying \eqref{conditional expectation for the universal}.
\end{proof}

The overall strategy of the proof of Theorem \ref{theorem for amenability and spectral subspaces} is comparable to that behind the proofs of the quoted  results in  \cite{LR}, \cite{FR}, \cite{F99}. Nevertheless, we deal here with a  much more general situation, which will require new insight. One of the new difficulties is the presence of  invertible elements in the semigroup $P$. To deal with them we consider the following equivalence relation on $P$:
$$
p\sim q   \Longleftrightarrow \quad  p=qx\text{ for some }x\in P^*.
$$
We denote by $[p]$ the equivalence class of $p\in P$ in $P/_\sim$. Note that $\sim$ might not be a congruence and therefore $P/_\sim$ might not inherit the semigroup structure from $P$, cf. \cite[Proposition 2.7]{bls}. However, $P/_\sim$ inherits the preorder. In fact,
$$
[p]\leq [q]  \Longleftrightarrow q\in pP
$$
yields a partial order on $P/_\sim$. Moreover, $[p]$ and $[q]$ have a  common upper bound  if and only if $pP\cap qP=rP$ for some $r\in P$, and if this holds then $[r]$ is the unique least common upper bound of $[p]$ and $[q]$. In the latter situation we write  $[p]\vee [q]:=[r]$.

The following two lemmas play a key role in the proof of Theorem \ref{theorem for amenability and spectral subspaces}.
\begin{lem}\label{auxiliary lemma0}
Retain the assumptions of Theorem \ref{theorem for amenability and spectral subspaces} (possibly without condition \eqref{injectivity in the same direction}).
 For every subset $F\subseteq \P/_\sim$ such that $[u]\vee [v] \in F$ whenever $[u], [v]\in F$, the space
\begin{equation}\label{form of K_F}
\KK_F:=\clsp\big\{ i_\KK( \KK(p,q)): p,q\in P,  \theta(p)=\theta(q), [\theta(p)]\in F\big\}
\end{equation}
is a $C^*$-subalgebra of $\NT_{\LL}(\KK)$ such that for any $F_0\subseteq \P$ with $F=\{[u]: u\in F_0\}$ we have
\begin{equation}\label{bla form of K_F}
\KK_F=\clsp\big\{ i_\KK( \KK(p,q)):p,q\in P, \theta(p)=\theta(q)\in F_0\big\}.
\end{equation}
\end{lem}
 \begin{proof}

To see that $\KK_F$ is a $C^*$-algebra, it suffices to show that $\KK_F$ is closed under multiplication. Let $u,v \in \P$ with $[u],[v]\in F$ and suppose $\theta(p)=\theta(q)=u$ and $\theta(s)=\theta(t)=v$. Let $a\in \KK(p,q)$ and $b\in \KK(s,t)$. Since $i_\KK$ is Nica covariant the product  $i_\KK(a)i_\KK(b)$ is either zero or $qP\cap sP=rP$, for some $r\in P$, and then
$$
i_\KK(a)i_\KK(b)=i_\KK\big((a\otimes 1_{q^{-1}r}) (b \otimes 1_{s^{-1}r})\big).
$$
In the latter case we have $(a\otimes 1_{q^{-1}r}) (b \otimes 1_{s^{-1}r}) \in \KK(pq^{-1}r,ts^{-1}r)$ and using \eqref{lcb preserver} we get $[\theta(r)]=[u]\vee [v] \in F$. Since $\theta(q)\theta(q^{-1}r)=\theta(r)$ we have $\theta(q^{-1}r)=\theta(q)^{-1}\theta(r)$ and therefore
$$
\theta(p q^{-1}r)=\theta(p) \theta(q^{-1}r)=\theta(p) \theta(q)^{-1}\theta(r)=\theta(r).
$$
Similarly,  $\theta(ts^{-1}r)=\theta(r)$. Thus $i_\KK(a)i_\KK(b)\in  \KK_F$. Hence $\KK_F$ is a $C^*$-algebra.

To prove \eqref{bla form of K_F}, let $F_0$ be a transversal for $F$ and suppose that  $\theta(p)=\theta(q)$ with  $[\theta(p)]\in F$ for $p,q\in P$. Since $\theta(P^*)=\P^*$ there is $x\in P^*$ such that $\theta(px)\in F_0$. By Lemma \ref{lemma on automorphic actions on ideals}  we have
$i_\KK( \KK(px,qx))=i_\KK( \KK(p,q))$. This finishes the proof.
\end{proof}
We will prove  injectivity in Theorem~\ref{theorem for amenability and spectral subspaces}  by induction over the size of finite subsets of $\P/_\sim$. In particular, it is crucial to establish the claim on sets with one element. As a matter of notation, if $F=\{[u]\}\subset \P/_\sim$, we write $\KK_{[u]}$ instead of $\KK_F$.
\begin{lem}\label{auxiliary lemma}
Retain the assumptions of Theorem \ref{theorem for amenability and spectral subspaces}. If $\Phi$ is an injective Nica covariant representation of $\KK$, then the homomorphism $\Phi\rtimes P$ of $\NT_\LL(\KK)$ is faithful on $\KK_{[u]}$ for every $u\in \P$.
\end{lem}
 \begin{proof}
 Let $u\in \P$. Then
 $\KK_{[u]}=\clsp\Big\{\bigcup_{p,q\in \theta^{-1}(u)} i_\KK( \KK(p,q))\Big\}$ by \eqref{bla form of K_F}.

Suppose that $\Phi$ is an injective Nica covariant representation of $\KK$. Note that whenever $s,t\in \theta^{-1}(u)$ with $s\neq t$,  \eqref{injectivity in the same direction} implies that $sP\cap tP =\emptyset$, and so by Nica covariance the $C^*$-subalgebras $i_\KK(\KK(s,s))$ and $i_\KK(\KK(t,t))$  of $\KK_{[u]}$ are orthogonal. Hence we have a direct sum $C^*$-subalgebra $D_u:=\bigoplus_{s\in \theta^{-1}(u) } i_\KK(\KK(s,s))$  in $\KK_{[u]}$. Since $\Phi\rtimes P$
is faithful  on each summand it is also faithful on $D_u$.

We claim that for every surjective $*$-homomorphism $\pi:\KK_{[u]}\to B$ for $B$ a $C^*$-algebra, the formula
\begin{equation}\label{conditional expectation fo KuKu}
E_\pi\Big(\pi\big(\sum_{s,t\in I}i_\KK(a_{s,t}\big)\Big):=\sum_{s\in I} \pi\big(i_\KK(a_{s,s})\big),
\end{equation}
where $a_{s,t}\in \KK(s,t)$ and $I\subseteq \theta^{-1}(u)$ is a finite set,  defines a faithful conditional expectation from $\pi(\KK_{[u]})$ onto $\pi(D_u)$. By choosing a faithful and non-degenerate  representation  of $B$ on a Hilbert space $H$, we may assume that $\pi:\KK_{[u]}\to \B(H)$ is a non-degenerate representation.

For each $r\in \theta^{-1}(u)$ let $Q_r\in \B(H)$ be the projection onto the essential space $\pi(i_\KK(\KK(r,r)))H$ for the $C^*$-algebra $\pi(i_\KK(\KK(r,r)))$. Note that the projections $Q_r$, $r\in \theta^{-1}(u)$, are pairwise orthogonal and their ranges span  $H$.
Thus, exactly as in the proof of  Lemma \ref{lemma on conditional expectations}, we conclude that the formula $E_\pi(a)= \sum_{r\in \theta^{-1}(u)} Q_r(a) Q_r$, for $a\in \B(H)$, defines a faithful conditional expectation from $\B(H)$ onto the commutant of $\{Q_r\}_{ r\in \theta^{-1}(u)}$. Thus it suffices to check that this map satisfies \eqref{conditional expectation fo KuKu}. But this is easy. Indeed, clearly we have
$Q_s \pi\bigl(i_\KK(a_{s,s})\bigr) Q_s=\pi\bigl(i_\KK(a_{s,s})\bigr)$ for every $s\in I$. On the other hand, if $r\in \theta^{-1}(u)$  is such that  either $r\neq s$ or $r\neq t$, then
$
Q_r\pi\bigl(i_\KK(a_{s,t} )\bigr)Q_r=0,
$
since, by Lemma~\ref{about approximate units}, we have $\pi\bigl(i_\KK(a_{s,t})\bigr)\in \pi\bigl(i_\KK(\KK(s,s))i_\KK (\KK(s,t))i_\KK(\KK(t,t))\bigr)\subseteq Q_{s}\B(H) Q_t$.

 The above claim applied separately to $\id$ and the $*$-homomorphism $\Phi\rtimes P$  gives a commutative diagram
$$
\begin{xy}
\xymatrix{
\KK_{[u]}  \ar[d]_{E_{\id}} \ar[rr]^{\Phi\rtimes P}&  & (\Phi\rtimes P)( \KK_{[u]} )
 \ar[d]^{E_{\Phi\rtimes P}}&
  \\
		D_u \ar[rr]^{\Phi\rtimes P }	&         & (\Phi\rtimes P)(D_u)
			}
  \end{xy}
$$
in which  the vertical arrows and the bottom horizontal arrow are faithful. Therefore, also  $(\Phi\rtimes P):\KK_{[u]}\to (\Phi\rtimes P)( \KK_{[u]} )$ is faithful.
\end{proof}

\begin{proof of}{Theorem \ref{theorem for amenability and spectral subspaces}}
We begin by proving that the Fock representation $T$ satisfies  \eqref{Toeplitz like condition}. To this end, let $u\in \P$  and  $F\subseteq \P$ be a finite set  such that  $u\not\geq v$ for every $v\in F$. Suppose that $p,q\in P$ satisfy $\theta(p)=\theta(q)=u$. By \eqref{crucial-observation-about-T}, the projection $Q_q$ in $\LL(\FF_\KK)$ onto $\bigoplus_{t\in P} X_{q,t}$ satisfies $Q_q\bigl(T(\KK(s,t))\bigr)Q_q=0$ whenever $s,t$ are in $P$ such that $q\notin sP\cap tP$. Thus, for all $s,t\in P$ with $\theta(s)=\theta(t)\in F$, we have $Q_q\bigl(T(\KK(s,t))\bigr)Q_q=0$ by the choice of $u$ and $F$.

An arbitrary element $a$ in $\clsp\{T(\KK(p,q)): \theta(p)=\theta(q)=u\}$ has the form
 $$
a=\sum_{q\in I}  a_q \,\,\text{ where }\,\,a_q\in \clsp\Big\{\bigcup_{p\in \theta^{-1}(u)} T( \KK(p,q))\Big\}
$$
where $I\subseteq \theta^{-1}(u)$ is finite. Suppose that $a\in \clsp\{T(\KK(s,t)): \theta(s)=\theta(t)\in F\}$. The considerations of the previous paragraph imply that $aQ_q=0$. On the other hand, $aQ_q=(\sum_{r\in I}a_r)Q_q$. For each $r\in I$ with $r\neq q$, the assumption \eqref{injectivity in the same direction} implies that $rP\cap qP=\emptyset$, and so the considerations above show that $a_rQ_q=0$. Therefore $aQ_q=a_qQ_q$ for every $q\in I$. This implies that  $a_q^*a_q Q_q=0$. However, since  $a_q^*a_q\in T(\KK(q,q))$ and $T(\KK(q,q))$ acts faithfully on $\oplus_{t\in P}X_{q,t}$ (because $\KK(q,q)$ acts faithfully on the subspace $X_{q,q}$) we get $a_q^*a_q=0$. Thus $a_q=0$ for every $q\in I$. Accordingly, $a=0$ and \eqref{Toeplitz like condition} is proved. Thus the injectivity claim in part (a) of the theorem will follow from part (b).

\emph{Sufficiency in part (b).} Let $\Phi:\KK\to B$ be an injective Nica covariant  representation  satisfying  \eqref{Toeplitz like condition}.

Let $\FF$ denote the collection of all finite subsets $F\subseteq \P/_\sim$ such that $[u]\vee [v] \in F$ whenever $[u], [v]\in F$.
For any finite  $F_0\subseteq \P/_\sim$  the set $F$ consisting of least upper bounds of all finite sub-collections of elements in $F_0$ is finite, contains $F_0$ and is closed under $\vee$. Thus $\FF$ is a directed set. Accordingly, the corresponding $C^*$-algebras \eqref{form of K_F} form  an inductive system  $\{\KK_F: F\in \FF\}$  with limit $\overline{\bigcup_{F\in \FF}\KK_F}=\clsp\Big\{\bigcup_{p,q\in P, \atop \theta(p)=\theta(q)} i_\KK( \KK(p,q))\Big\}$. Hence  \eqref{core like algebra} is a $C^*$-algebra, and to prove faithfulness of  $\Phi\rtimes P$ on $\overline{\bigcup_{F\in \FF}\KK_F}$ it suffices to show that $\Phi\rtimes P$ is faithful on $\KK_F$ for each $F\in \FF$, see e.g. \cite[Lemma 1.3]{alnr}.

By Lemma \ref{auxiliary lemma},    $\Phi\rtimes P$ is faithful on $\KK_{F}$ if $F= \{[u]\}$ for some $u\in \P$.
For the inductive step, let $F\in \FF$ and suppose that $\Phi\rtimes P$ is faithful on $\KK_{F'}$ whenever $F'\in \FF$ and $|F'|<|F|$; we aim to prove that $\Phi\rtimes P$ is faithful on $\KK_{F}$. Let $Z\in \KK_F$ be a finite sum of the form
$$
Z=\sum_{[u]\in F} Z_{[u]} \qquad \textrm{where } Z_{[u]}\in \KK_{[u]}\,\,\textrm{ for }\,\, [u]\in F.
$$
Suppose that $\Phi\rtimes P(Z)=0$. We will show that $Z=0$, giving the desired injectivity.

 Since $F$ is finite it has a minimal element, that is, there is $[u_0]\in F$ such that $[u] \not \leq [u_0]$ for every $[u]\in F\setminus \{[u_0]\}$.
It is immediate from considerations concerning products of elements in $\KK_F$, cf. the proof of Lemma \ref{auxiliary lemma0}, that the $C^*$-algebra $\KK_{F\setminus \{[u_0]\}}$ is an ideal in $\KK_F$.
Hence $(\Phi\rtimes P)(\KK_{F\setminus \{[u_0]\}})$ is an ideal in $(\Phi\rtimes P)(\KK_{F})$. Let
$$
\rho:(\Phi\rtimes P)(\KK_{F})\to (\Phi\rtimes P)(\KK_{F})/(\Phi\rtimes P)(\KK_{F\setminus \{[u_0]\}})
$$
be the quotient map. We claim that  $\rho$ is injective on $(\Phi\rtimes P)(\KK_{[u_0]})$. Indeed, by \eqref{bla form of K_F} we have
$$
(\Phi\rtimes P)(\KK_{[u_0]})= \clsp\{\Phi(\KK(p,q)): \theta(p)=\theta(q)=u_0\}
$$
and for any $F_0\subseteq \P$  finite set such that $F\setminus \{[u_0]\}=\{[u]: u \in F_0\}$, we have
$$
(\Phi\rtimes P)(\KK_{F\setminus \{[u_0]\}})=\clsp\{\Phi(\KK(s,t)): \theta(s)=\theta(t)\in F_0\}.
$$
Note that $v\not\leq u_0$ for every $v\in F_0$. Thus  condition \eqref{Toeplitz like condition} applied to $u_0$ and $F_0$ proves our claim. In particular, $\rho$ is isometric on the  $C^*$-algebra $\Phi\rtimes P(\KK_{[u_0]})$. Thus
$$
\|(\Phi\rtimes P) (Z_{[u_0]})\|=\|\rho\Big((\Phi\rtimes P) (Z_{[u_0]})\Big)\|=\|\rho\Big((\Phi\rtimes P) (Z)\Big)\|=0.
$$
This implies that $Z_{[u_0]}=0$,  as $\Phi\rtimes P$ is isometric on  $\KK_{[u_0]}$ by the first inductive step. Hence $Z=\sum_{[u]\in F\setminus\{u_0\}} Z_{[u]} \in \KK_{F\setminus \{[u_0]\}}$. Therefore, $Z=0$ by the inductive hypothesis.

\emph{Necessity in part (b).} Since  $T$ satisfies \eqref{Toeplitz like condition}, we conclude by sufficiency in part (b) that  the regular representation $T\rtimes P$ is faithful on the $C^*$-algebra in \eqref{core like algebra}. It is not difficult to see that this forces the universal representation $i_\KK$ to satisfy \eqref{Toeplitz like condition}: take $u$ and $F$ as specified for \eqref{Toeplitz like condition} and suppose that an element $C$ in $\NT_\LL(\KK)$ satisfies  $C=i_\KK(C_1)=i_\KK(C_2)$ where $C_1\in \KK(p,q)$ is in the closed span determined by $u$ and $C_2\in \KK(s,t)$ is in the closed span determined by $F$. Then $T(C_1)=(T\rtimes P)\circ i_\KK(C_1)=(T\rtimes P)\circ i_\KK(C_2)=T(C_2)$ and so, first, $T(C_1)=T(C_2)=0$ because $T$ satisfies  \eqref{Toeplitz like condition}, and second, $i_\KK(C_1)=i_\KK(C_2)=0$ because both are in the $C^*$-subalgebra
\eqref{core like algebra} on which $T\rtimes P$ is faithful. This gives $C=0$ in this particular case, and the general case follows from here. Now, since $i_\KK$ satisfies \eqref{Toeplitz like condition}, then every Nica covariant representation whose integrated form is faithful on  \eqref{core like algebra} has to satisfy \eqref{Toeplitz like condition}. This concludes the proof of the theorem.
\end{proof of}

 \section{Exotic Nica-Toeplitz $C^*$-algebras}\label{exotic embeddings section} We fix a well-aligned ideal $\KK$ in a right-tensor $C^*$-category $\LL$. In this section, we introduce Nica-Toeplitz $C^*$-algebras of $\KK$ that sit between $\NT_\LL(\KK)$ and $\NT_\LL^r(\KK)$. Inspired by the terminology introduced in the context of group $C^*$-algebras in \cite{brow-guen}, we shall call these $C^*$-algebras \emph{exotic}. %Our use of this terminology is not arbitrary: 
In view of Example \ref{Crossed products by group actions} below, our exotic Nica-Toeplitz $C^*$-algebras  are in fact generalizations of exotic crossed products for actions of discrete groups.
\begin{defn} We say that a Nica covariant representation $\Phi:\KK\to B$  \emph{generates an exotic Nica-Toeplitz $C^*$-algebra} if
$
\ker (\Phi\rtimes P) \subseteq \ker (T\rtimes P)
$. If this is the case, we refer to $C^{*}(\Phi(\KK))$ as an \emph{exotic Nica-Toeplitz $C^*$-algebra} of $\KK$.
\end{defn}
\begin{rem}
Clearly, $\Phi$ generates an exotic Nica-Toeplitz $C^*$-algebra if and only if there is a homomorphism $\Phi_*:C^{*}(\Phi(\KK))\to \NT_{\LL}^{r}(\KK)$ making the following diagram commute:
\begin{equation}\label{pseudo uniqueness diagram}
\xymatrix{  \NT_{\LL}(\KK)  \ar@/_2pc/[rrrr]^{\,\, T\rtimes P}  \ar[rr]^{\Phi\rtimes P}  &
&C^*(\Phi(\KK))  \ar[rr]^{\Phi_* \,\,} & &   \NT_{\LL}^{r}(\KK)}.
\end{equation}
%Hence exotic Nica-Toeplitz $C^*$-algebras are intermediate between the full and the reduced one.
\end{rem}
%Representations generating exotic Nica-Toeplitz algebras are necessarily Toeplitz covariant.
\begin{prop}\label{proposition to invoke}
If $\Phi:\KK\to B$ is a Nica covariant representation that  generates  an exotic Nica-Toeplitz $C^*$-algebra, then $\Phi$ is  injective and Toeplitz covariant.
\end{prop}
\begin{proof}
The diagram \eqref{pseudo uniqueness diagram} restricts to a commuting  diagram of core subalgebras
\begin{equation}\label{restricted diagram}
\xymatrix{  B_e^{i_\KK}  \ar@/_2pc/[rrrr]^{ T\rtimes P}  \ar[rr]^{\Phi\rtimes P}  &
&B_e^\Phi   \ar[rr]^{   \Phi_* } & &   B_e^T}.
\end{equation}
By  Corollary \ref{the reduced core}, the map $T\rtimes P$ restricted to $B_e^{i_\KK}$ is an isomorphism onto  $B_e^T$. Thus  $\Phi\rtimes P:B_e^{i_\KK} \to B_e^\Phi$ is an isomorphism and Corollary  \ref{Nica-Toeplitz representation corollary} implies the assertion.
\end{proof}

Nica covariant representations generating exotic Nica-Toeplitz algebras can be characterized as representations which  admit a \emph{transcendental conditional expectation}:
 \begin{prop}\label{preludium to Nica Toeplitz uniqueness}
Let $B_\KK$ be the self-adjoint operator subspace of $\LL(\FF_\KK)$ introduced in \eqref{first form of E T}. For a Nica covariant representation $\Phi:\KK\to B$  the following conditions are equivalent:
\begin{itemize}
\item[(i)] $\Phi:\KK\to B$  generates  an exotic Nica-Toeplitz $C^*$-algebra of $\KK$,

\item[(ii)] There is a bounded map $E^{\Phi}:C^{*}(\Phi(\KK))\to B_\KK$ satisfying
\begin{equation}\label{formula defining strange map}
E^{\Phi}\Big(\sum_{p,q\in F}\Phi(a_{p,q})\Big)=\sum_{p,q\in F }\bigoplus_{w\in pP\cap qP,t\in P \atop p^{-1}w=q^{-1}w }T_{p,q}^{w,t}(a_{p,q})
\end{equation}
for every $a_{p,q}\in \KK(p,q)$ and finite $F\subseteq P$.
\end{itemize}
If (i) and (ii) hold then   $E^{\Phi}$ is a contractive completely positive map making the diagram
\begin{equation}\label{diagram to commute}
\begin{xy}
\xymatrix{
C^{*}(\Phi(\KK)) \ar[dr]_{E^{\Phi}} \ar[rr]^{\Phi_*}&  & \NT_{\LL}^{r}(\KK) \ar[dl]^{E^T}
  \\
			&  B_\KK       &
			}
  \end{xy}
\end{equation}
commute. Moreover, $\Phi_*$ is an isomorphism if and only if $E^{\Phi}$ is faithful.
\end{prop}
\begin{proof} The implication (i)$\Rightarrow$(ii) follows by letting $E^\Phi:= E^{T}\circ \Phi_*$. In particular,  $E^\Phi$ is a contractive completely positive map and  the diagram \eqref{diagram to commute} commutes. Moreover,  since $E^T$ is faithful, $\Phi_*$ is faithful if and only if $E^{\Phi}$ is faithful.

We prove next the implication (ii)$\Rightarrow$(i). Let us  put $E^{i_\KK}:=E^T \circ (T\rtimes P)$. Then
$$
 \ker (T\rtimes P)=\{ a\in  \NT_{\LL}(\KK): E^{i_\KK}(a^*a)=0\}.
$$
Indeed, inclusion $\ker (T\rtimes P)\subseteq \{ a\in  \NT_{\LL}(\KK): E^{i_\KK}(a^*a)=0\}$ is trivial. The reverse inclusion follows, because if $a\in  \NT_{\LL}(\KK)$ is such that $E^{i_\KK}(a^*a)=0$, then  $E^{T}\big((T\rtimes P)(a)^*(T\rtimes P)(a)\big)=0$ and therefore $(T\rtimes P)(a)=0$ by faithfulness of $E^T$. The assumption in (ii) implies that we have a commutative diagram
$$
%\begin{equation}\label{diagram to commute}
\begin{xy}
\xymatrix{
\NT_{\LL}(\KK) \ar[dr]_{E^{i_\KK}} \ar[rr]^{\Phi\rtimes P}&  &  C^{*}(\Phi(\KK))  \ar[dl]^{E^\Phi}
  \\
			&  B_\KK       &
			}
  \end{xy}
%\end{equation}
$$
Indeed, this diagram commutes when restricted to the  dense $C^*$-subalgebra
  $\NT_{\LL}(\KK)^0=\spane\{\bigcup_{p,q \in P} i_\KK(\KK(p,q))\}$ of $\NT_{\LL}(\KK)$, cf.
\eqref{black star algebra}. Hence by continuity the diagram commutes on $\NT_{\LL}(\KK)$.
Now,
$$
a\in \ker (\Phi\rtimes P) \,\Longrightarrow\, E^\Phi\big((\Phi\rtimes P)(a^*a)\big)=0 \,\Longrightarrow\, E^{i_\KK}(a^*a)=0.
$$
% \,\Longrightarrow\, a\in \ker (T\rtimes P).
Thus $\ker (\Phi\rtimes P)\subseteq \{ a\in  \NT_{\LL}(\KK): E^{i_\KK}(a^*a)=0\}=\ker (T\rtimes P)$. \end{proof}

\begin{rem} Suppose that $P$ is cancellative. Then the right-hand side of \eqref{formula defining strange map} reduces to the right-hand side of \eqref{another form of E T}  and $B_\KK=B^T_e$, cf. Proposition \ref{Nica Toeplitz conditional expectation}. %$E^T$ is a conditional expectation from $\NT_{\LL}^r(\KK)$ onto
Accordingly, Proposition \ref{preludium to Nica Toeplitz uniqueness} reduces to the following statement: a Nica covariant representation  $\Phi$   of $\KK$ generates an exotic Nica-Toeplitz $C^*$-algebra if and only if  the formula
$$
E_\Phi\Big(\sum_{p,q\in F } \Phi(a_{p,q})\Big)=\sum_{p\in F } \Phi(a_{p,p}) , \qquad a_{p,q}\in \KK(p,q),\,\, p,q \in F\subseteq P, \, F \text{ finite},
$$
defines  a genuine conditional expectation from $C^*(\Phi(\KK))$  onto its core subalgebra $B^\Phi_e$. In fact, if  $\Phi$    generates an exotic Nica-Toeplitz $C^*$-algebra, we have the commutative diagram
\begin{equation}\label{diagram to commute666}
\begin{xy}
\xymatrix{
\NT_{\LL}(\KK)  \ar[d]_{E} \ar[rr]^{\Phi \rtimes P}&  & C^{*}(\Phi(\KK)) \ar[d]^{E_\Phi}   \ar[rr]^{\Phi_*}&  & \NT_{\LL}^r(\KK) \ar[d]_{E^T}
  \\
		B_e^{i_\KK} \ar[rr]^{\Phi \rtimes P}	&         & B_e^\Phi \ar[rr]^{\Phi_*} 	&          & B_e^T
			}
  \end{xy}
\end{equation}
where the bottom horizontal arrows are isomorphisms and $E_\Phi=(\Phi_* \vert_{B_e^\Phi})^{-1}\circ E^\Phi$. Moreover,   $\Phi_*:C^{*}(\Phi(\KK))\to \NT_{\LL}^{r}(\KK)$ is an isomorphism if and only if $E^{\Phi}$ is faithful.
\end{rem}

As a first application of Proposition~\ref{preludium to Nica Toeplitz uniqueness}, we show that in order to study $\NT_\LL^r(\KK)$, one may in some cases use the $t$-th Fock subrepresentation $T^t$  of $T$ (usually $T^{e}$ will work, see \cite{kwa-larII}).
\begin{prop} \label{Fock t-th representations and reduced objects}
Let $t\in P$ and suppose that the well-aligned ideal $\KK$ satisfies the condition:
\begin{equation}\label{condition for reducing Focks}
\forall_{p\in P}\,\, \exists_{x\in P^*} \text{  } \overline{\KK(px,t)\KK(t,px)}\text{ is an  essential  ideal in the $C^*$-algebra } \LL_\KK(px,px).
\end{equation}
Then the $t$-th Fock representation $T^t:\KK\to \LL(\FF_\KK^t)$ generates a copy of $\NT_\LL^r(\KK)$, in the sense that there is
an isomorphism  $h_t:\NT^{r}_{\LL}(\KK) \to C^*(T^t(\KK))$ such that $h_t\circ T=T^t$.
\end{prop}
\begin{proof}
In view of Remark \ref{remark about transcendentals for teeth} and the last part of Proposition  \ref{preludium to Nica Toeplitz uniqueness}, it suffices to show
 that we have an isometry
\begin{equation}\label{projection to be isometry}
 B_\KK\in  Z=\sum_{p,q\in F }\bigoplus_{w\in pP\cap qP, s\in P\atop p^{-1}w=q^{-1}w }T_{p,q}^{w,s}(a_{p,q})  \longmapsto Z^t:=\sum_{p,q\in F }\bigoplus_{w\in pP\cap qP\atop p^{-1}w=q^{-1}w }T_{p,q}^{w,t}(a_{p,q}) \in  B_\KK^t,
\end{equation}
This map is a well defined contraction, as it is the restriction of  the projection from $\bigoplus_{s\in P}\LL(\FF_{\KK}^{s})$ to $ \LL(\FF_{\KK}^{t})$. Now,
 the argument leading to \eqref{line with supremum} gives us that
$$
\|Z^t\|=\max_{C \in \In(F)} \sup_{w\in P_{F, C}}  \Big\| T_{w,w}^{w,t}(\sum_{p,q\in C \atop p^{-1}w=q^{-1}w} a_{p,q}\otimes 1_{q^{-1}w}) \Big\|.
$$
Let us fix $w\in  P_{F, C}$. Note that the sum under $T_{w,w}^{w,t}$ belongs to $ \LL_\KK(px,px)$. By \eqref{condition for reducing Focks}  there is $x\in P^*$ such that the homomorphism $T_{wx,wx}^{wx,t}: \LL_\KK(wx,wx) \to \LL(X_{wx,t})$ is injective (and hence isometric), cf. \eqref{line with kernels}. Since $\LL_\KK\otimes 1\subseteq \LL_\KK$ and the homomorphism $\otimes 1_x$ is isometric we get
\begin{align*}
 \Big\| \sum_{p,q\in C \atop p^{-1}w=q^{-1}w} a_{p,q}\otimes 1_{q^{-1}w} \Big\|&= \Big\| \sum_{p,q\in C \atop p^{-1}wx=q^{-1}wx} a_{p,q}\otimes 1_{q^{-1}wx} \Big\|
\\
&=   \Big\|T_{wx,wx}^{wx,t}\Big(\sum_{p,q\in C \atop p^{-1}wx=q^{-1}wx} a_{p,q}\otimes 1_{q^{-1}wx}\Big) \Big\|.
\end{align*}
Clearly, $wx\in P_{F, C}$ and therefore $\|Z^t\|$  is not greater than $\|Z\|$ by \eqref{norm of Z formula}. Hence \eqref{projection to be isometry} is an isometry.
\end{proof}

 \section{Amenability and Fell Bundles}\label{amenability section}
We fix  a well-aligned ideal $\KK$ in a  right-tensor $C^*$-precategory  $(\LL, \{\otimes 1_r\}_{r\in P})$ over a right LCM semigroup $P$. In this section, we study the properties under which all exotic Nica-Toeplitz algebras of $\KK$ are naturally isomorphic.

\begin{defn}
The well-aligned ideal $\KK$ in $(\LL, \{\otimes 1_r\}_{r\in P})$ is called \emph{amenable} if the regular representation  $T\rtimes P:\NT_{\LL}(\KK)\longrightarrow \NT_{\LL}^r(\KK)$ is an isomorphism. A right-tensor $C^*$-precategory $\LL$ is \emph{amenable} if it is amenable as an ideal in itself.
\end{defn}
We have the following simple characterization of amenability in terms of the natural map on the transcendental core $B_\KK$.
\begin{lem}\label{nonsense lemma}
A well-aligned ideal $\KK$  in a right-tensor $C^*$-precategory $\LL$ is amenable if and only if the map $E^{i_\KK}=E^T \circ (T\rtimes P): \NT_{\LL}(\KK)  \to B_\KK$ is faithful.

If $P$ is cancellative, then $\KK$ is amenable if and only if the conditional expectation $E$ from $\NT_{\LL}(\KK)$ onto $B_e^{i_\KK}$ given by \eqref{conditional expectation for the universal}  is faithful.
\end{lem}
\begin{proof}
Apply the last part of Proposition \ref{preludium to Nica Toeplitz uniqueness}.
\end{proof}
We obtain more efficient  amenability criteria using the theory of Fell bundles. We first recall the definition of a full coaction of $G$ on a $C^*$-algebra $A$.  An unadorned tensor product of
$C^*$-algebras will denote the minimal tensor product. We write $g \mapsto i_G(g)$ for the canonical inclusion of
$G$ as unitaries in the full group $C^*$-algebra $C^*(G)$.  There is a homomorphism $\delta_G:C^*(G)\to C^*(G)\otimes C^*(G)$ given by $\delta_G(g)=i_G(g)\otimes i_G(g)$. A \emph{full coaction} of $G$ on $A$ is an injective, non-degenerate homomorphism $\delta:A\to A\otimes C^*(G)$ that satisfies the coaction identity $(\delta \otimes id_{C^*(G)})\circ \delta=  (id_A\otimes \delta_G)\circ \delta$.

\begin{prop}\label{gauge co-action}
Suppose that $\theta:P \to G$ is a unital semigroup homomorphism from $P$ into a group $G$. There is a full coaction $\delta$ of $G$ on $\NT_{\LL}(\KK)$ such that
$$
\delta(i_\KK(a))=i_\KK(a)\otimes i_G(\theta(p)\theta(q)^{-1})\text{ for every }a\in \KK(p,q),\, p,q\in P.
$$
The Fell bundle $\B^\theta=\{B_g^\theta\}_{g\in G}$ associated to $\delta$ has fibers given by
$$
B_g^\theta=\clsp\big\{i_\KK\left(\KK(p,q)\right) : p,q \in P,  g=\theta(p)\theta(q)^{-1}\big\}\,\,\text{ if }\,\,g\in \theta(P)\theta(P)^{-1},
$$
and $B_g^\theta=\{0\}$ if $g\notin \theta(P)\theta(P)^{-1}$. In particular, $\NT_{\LL}(\KK)\cong C^*(\B^\theta)$ with the isomorphism which is identity on the spaces $B_g^\theta$, $g\in G$.
\end{prop}
\begin{proof}
We claim  that the maps $\KK(p,q) \ni a \longmapsto i_\KK(a)\otimes i_G(\theta(p)\theta(q)^{-1})$ yield a Nica covariant representation $\Phi:\KK \to \NT_{\LL}(\KK)\otimes C^*(G)$.  Indeed, let $a\in \KK(p,q)$ and $b\in K(s,t)$, $p,q,s,t\in P$.  If $qP\cap sP=\emptyset$, then $i_\KK(a)i_\KK(b) =0$
 and therefore $\Phi(a)\Phi(b)=0$. Assume  that $qP\cap sP=rP$ for some $r\in P$. Writing for example $qq'=ss'=r$ for some $s',q'\in P$ shows that
 $\theta(p)\theta(q)^{-1}\theta( s)\theta(t)^{-1}=\theta(pq')\theta(ts')^{-1}$. With $q'=q^{-1}r$ and $s'=s^{-1}r$, we
 therefore have
\begin{align*}
\Phi(a)\Phi(b)&= \Big(i_\KK(a)i_\KK(b)\Big)\otimes\Big( i_G(\theta(p)\theta(q)^{-1})i_G(\theta( s)\theta(t)^{-1})\Big)
\\
&= i_\KK\Big((a \otimes 1_{q^{-1}r}) (b\otimes 1_{s^{-1}r})\Big)\otimes\Big(\theta(pq^{-1}r)\theta(ts^{-1}r)^{-1})\Big)
\\
&=\Phi\Big((a \otimes 1_{q^{-1}r}) (b\otimes 1_{s^{-1}r})\Big).
\end{align*}
Thus $\Phi$ integrates to  a homomorphism $\delta=\Phi\rtimes P:\NT_{\LL}(\KK)\to \NT_{\LL}(\KK)\otimes C^*(G)$. By Lemma \ref{about approximate units} and the Hewitt-Cohen factorization theorem every element $a\in \KK(p,q)$, $p,q\in P$, can be written as $a=a'a''$ where $a'\in \KK(p,p)$ and $a''\in \KK(p,q)$. Thus
$$
i_\KK(a)\otimes i_G(\theta(p)\theta(q)^{-1})=\Big(i_\KK(a')\otimes i_G(e)\Big) \Big(i_\KK(a'')\otimes i_G(\theta(p)\theta(q)^{-1})\Big)\in \NT_{\LL}(\KK)\otimes C^*(G).
$$
 This implies that $\delta$ is non-degenerate.  For every $a\in \KK(p,q)$, $p,q\in P$ we have
\begin{align*}
 \Big((\delta \otimes id_{C^*(G)})\circ \delta\Big)(i_\KK(a)) &=  (\delta \otimes id_{C^*(G)}) \big(i_\KK(a)\otimes i_G(\theta(p)\theta(q)^{-1})\big)
 \\
 &=  i_\KK(a)\otimes i_G\big(\theta(p)\theta(q)^{-1})  \otimes i_G(\theta(p)\theta(q)^{-1}\big)
 \\
 &=\Big((id_A\otimes \delta_G)\circ \delta\Big) (i_\KK(a)).
 \end{align*}
 Hence $\delta$ is a full coaction. It is readily seen that the spectral subspaces $B_g^\theta$, $g\in G$, for the coaction $\delta$ are of the claimed form. To see that $\NT_{\LL}(\KK)\cong C^*(\B^\theta)$ it suffices   to note that $C^*(\B^\theta)$ is generated by a universal  Nica covariant representation. Now $C^*(\B^\theta)$ is generated by the spaces $i_\KK(\KK(p,q))$, $p,q\in G$, and  for a Nica covariant representation $\Phi$ of $\KK$ the map $ \Phi\rtimes G: \bigoplus_{g\in G} B_g^\theta\to  C^*(\Phi(\KK))$ given by $(\Phi\rtimes G)( \oplus_{g\in G} b_g):=\sum_{g\in G}(\Phi\rtimes P)(b_g)$ is  a $*$-homomorphism. Since $C^*(\B^\theta)$ is the completion of $\bigoplus_{g\in G} B_g$ in the maximal $C^*$-norm, the homomorphism $ \Phi\rtimes G$ extends to the epimorphism $ \Phi\rtimes G:C^*(\B^\theta) \to  C^*(\Phi(\KK))$.
\end{proof}

\begin{thm}\label{amenability main result}
Suppose that $\theta:P \to \P\subseteq G$ is a controlled map of right LCM semigroups such that $\P$ is a subsemigroup of a group $G$. Let
$\B^\theta$ be the Fell bundle described in Proposition \ref{gauge co-action}. We  have a commutative diagram
$$
\begin{xy}
\xymatrix{
C^*(\B^\theta)  \ar[d]_{\cong} \ar[rr]^{ \Lambda }&  & C^*_r(\B^\theta) \ar[d]^{ \cong}
  \\
	\NT_{\LL}(\KK) \ar[rr]^{ T\rtimes P}	&         & 	\NT_{\LL}^r(\KK)}
  \end{xy}
$$
where vertical arrows are isomorphisms and horizontal ones are regular representations.
 In particular, $\B^\theta$  is amenable if and only if $\KK$ is amenable as an ideal of $\LL$.
\end{thm}
\begin{proof} Let   $\Pi:\NT_{\LL}(\KK)\to C^*(\B^\theta)$ be the isomorphism given by Proposition \ref{gauge co-action}.
 It is easy to see that  $\Phi:=\Lambda\circ \Pi\circ i_\KK$ is a Nica covariant representation of $\KK$ such that $C^*(\Phi(\KK))$ is equal to $C_r^*(\B^\theta)$.  There are canonical conditional expectations $E^\delta:C^*(\B^\theta)\to B_e^{\theta}$ and $E^{\delta,r}: C^*_r(\B^\theta)\to B_e^{\theta}$. The existence of a controlled map into a semigroup $\P$ that is right cancellative (being a subsemigroup of $G$) guarantees that $P$ is cancellative, see Remark~\ref{rem:P cancellative if controlled} (a). By Proposition~\ref{Nica Toeplitz conditional expectation}, the transcendental core  $B_\KK$ is equal to $B_e^T$, hence it is a subspace of
$(T\rtimes P)(B_e^\theta)=\clsp\big\{ T\left(\KK(p,q)\right):p,q \in P, \theta(p)=\theta(q)  \big\}$. By Theorem \ref{theorem for amenability and spectral subspaces} the $*$-homomorphism $T\rtimes P:B_e^{\theta}\to (T\rtimes P)(B_e^\theta)$  is an isomorphism.
Hence $E^\Phi:=E^T\circ(T\rtimes P)\circ E^{\delta,r}$ is
 a faithful completely positive map from $C^*_r(\B^\theta)$ onto $B_\KK=B_e^T$.
We claim that  $E^\Phi$  satisfies equation \eqref{formula defining strange map}. Note first that
\begin{equation}\label{conditional expectation from co-action}
E^\delta\circ\Pi(i_\KK(a)):=\begin{cases}
\Pi(i_\KK(a)) &\textrm{if }\theta(p)=\theta(q),
\\
0& \textrm{if }\theta(p)\neq\theta(q),
\end{cases} \quad \textrm{ for every }a\in \KK(p,q),\,\,p,q\in P.
	\end{equation}
Then for every choice of finite family $a_{p,q}$ in $\KK(p,q)$, where $p,q\in F$ finite, we have
\begin{align*}
E^\Phi(\sum_{p,q\in F}\Phi(a_{p,q}))
&=E^T\circ (T\rtimes P)\circ E^{\delta,r}\bigl(\sum_{p,q\in F}\Lambda\circ \Pi\circ i_\KK(a_{p,q})\bigr)\\
&=E^T\circ (T\rtimes P)\circ E^\delta\bigl(\sum_{p,q\in F} \Pi\circ i_\KK(a_{p,q})\bigr)\\
&=E^T\bigl(\sum_{\{p,q\in F:\theta(p)=\theta(q)\}} T(a_{p,q})\bigr),\\
\end{align*}
which is the term in the right-hand side of \eqref{formula defining strange map} (which in this case  reduces to the right-hand side of \eqref{another form of E T}). Thus  Proposition~\ref{preludium to Nica Toeplitz uniqueness} implies that there is a $*$-homomorphism $\Phi_*$ from $C_r^*(\B^\theta)$ onto $\NT_{\LL}^r(\KK)$ such that $T\rtimes P=\Phi_*\circ (\Phi\rtimes P)$. Since $E^\Phi$ is faithful, the same proposition shows that in fact $\Phi_*$ is an isomorphism. This  implies the assertion.
\end{proof}
\begin{rem}\label{rem:on Fell bundles from controlled homomorphisms}
Under the assumptions of Theorem \ref{amenability main result}, the Fell bundle $\B^\theta$ can be constructed using any injective Nica covariant representation $\Phi$ satisfying \eqref{Toeplitz like condition}. More specifically, given such a representation $\Phi$ we define
$$
B_g^{\Phi,\theta}:=\clsp\big\{ \Phi\left(\KK(p,q)\right):p,q \in P, g=\theta(p)\theta(q)^{-1}  \big\}\,\,\text{ if }\,\,g\in \theta(P)\theta(P)^{-1}
$$
and $B_g^{\Phi,\theta}:=\{0\}$ for $g\notin \theta(P)\theta(P)^{-1}$. Then  $\B^{\Phi,\theta}:=\{B_g^{\Phi,\theta}\}_{g\in G}$ is a Fell bundle isomorphic to  $\B^\theta$. Indeed, $\Phi\rtimes P:B_e^{\theta}\to B_e^{\Phi,\theta}$  is an isomorphism by Theorem \ref{theorem for amenability and spectral subspaces}, and hence by the $C^*$-equality  all the mappings $\Phi\rtimes P:B_g^{\theta}\to B_g^{\Phi,\theta}$, $g\in G$, are isomorphisms.
\end{rem}

The condition of amenability of $\B^\theta$ in Theorem~\ref{amenability main result} is satisfied for instance when
$G$ is amenable or  $\B^\theta$ has the approximation property, \cite{Exel}. In particular, we get  the
following generalization of  \cite[Corollary 8.2]{F99}.

\begin{cor}\label{amenability of free products}
Suppose that $P:=\prod^*_{i\in I} P_i$ is the free product of a family of right LCM semigroups $P_i$, $i\in I$, where each $P_i$ is a subsemigroup of an amenable group $G_i$.
Any well-aligned ideal $\KK$  in a right-tensor $C^*$-precategory $\LL$ over  $P$ is amenable.
\end{cor}
\begin{proof}
The direct sum $G:=\bigoplus_{i\in I} G_i$ is an amenable group. By Proposition \ref{form free products to direct sums}, we have a homomorphism $\theta:P \to G$ which satisfies the assumptions of Theorem \ref{amenability main result}. Hence the assertion follows from Theorem \ref{amenability main result}.
\end{proof}

\section{Projections associated to Nica covariant representations}\label{Projections section}
In this section we fix a Nica covariant  representation  $\Phi:\KK\to\B(H)$ of a well-aligned ideal $\KK$ in a right-tensor $C^*$-precategory $\LL$. We let $\overline{\Phi}:\LL\to \B(H)$ be the extension of $\Phi$  from Proposition \ref{extensions of representations on Hilbert spaces0}. Our goal is to investigate two families of projections associated to $\Phi$: $\{Q^\Phi_{p}\}_{p\in P}$ and $\{Q^\Phi_{\langle p \rangle}\}_{p\in P}$. The former are projections onto the essential spaces we used to define $\overline{\Phi}$. The latter  can be considered   analogues of projections we associated to the Fock representation in Lemma \ref{lemma-proj-T-semillatice}.  In general,  the family $\{Q^\Phi_{p}\}_{p\in P}$ has some good properties that $\{Q^\Phi_{\langle p \rangle}\}_{p\in P}$  lacks  and vice versa. Thus the case when these families coincide is desirable and we give a natural condition implying that.
\begin{defn}\label{definition of family projections2}
For each $p\in P$, we denote by $Q^\Phi_p$ and $Q^\Phi_{\langle p\rangle}$   projections in $\B(H)$ defined by
\begin{equation}\label{definition of family projections}
Q^\Phi_{p}H=\begin{cases}
\Phi(\KK(p,p))H  & \textrm{ if } p\in P\setminus P^*,
\\
C^*(\Phi(\KK))H & \textrm{ if } p\in  P^*.
\end{cases}
\end{equation}
and $Q^\Phi_{\langle p\rangle}H=\clsp\{ \Phi(\KK(w,w))H :  w\in  pP\}$
(so we have $Q^\Phi_{\langle p\rangle}=\bigvee_{w \geq p }Q^\Phi_w$).
\end{defn}
\begin{lem}\label{properties of family projections}
There is a well defined mapping $J(P) \longmapsto \Proj(\B(H))$ given by the assignment
\begin{equation}\label{mapping to be semilattice homo}
 pP \longmapsto Q^\Phi_{p}\quad \text{ and }\quad \emptyset \longmapsto  0
\end{equation}
which sends comparable ideals (in the sense of inclusion) to commuting projections and has the property that $Q_{p}^{\Phi}Q_{q}^{\Phi}=0$ whenever $pP\cap qP=\emptyset$ for $p,q\in P$.
Moreover, for every $p, q, s \in P$ and $a\in \KK(p,q)$
we have
\begin{equation}\label{action on projections general}
\Phi(a)Q^\Phi_s=\begin{cases}
\overline{\Phi}(a\otimes 1_{q^{-1}r})Q^\Phi_s  & \textrm{ if } sP\cap qP=rP \textrm{ for some } r\in P,
\\
0 & \textrm{ if }sP\cap qP=\emptyset.
\end{cases}
\end{equation}
If additionally $d\in \LL(s,s)$ and $t \geq s$, then $\overline{\Phi}(d)Q^\Phi_t= Q^\Phi_t\overline{\Phi}(d)$ and
\begin{equation}\label{action on projections general2}
\Phi(a)\overline{\Phi}(d)=\begin{cases}
\overline{\Phi}(a\otimes 1_{q^{-1}r})\overline{\Phi}(d)  & \textrm{ if } sP\cap qP=rP \textrm{ for some } r\in P,
\\
0 & \textrm{ if }sP\cap qP=\emptyset.
\end{cases}
\end{equation}
\end{lem}
\begin{proof}
Note first that  for every $p\in P\setminus P^*$ the projection $Q_{p}^{\Phi}$ is the strong limit of the net $\{\Phi(\mu_\lambda^p)\}_{\lambda\in \Lambda}$, where $\{\mu_{\lambda}^p\}_{\lambda\in \Lambda}$ is an approximate unit in $\KK(p,p)$.

If $pP=qP$ for some $p,q\in P$ then $q=px$ for some $x\in P^{*}$ and therefore $\Phi(\KK(q,q))=\Phi(\KK(p,p)\otimes{1_x})=\Phi(\KK(p,p))$ by Lemma \ref{lemma on automorphic actions on ideals}. Thus $Q_{p}^{\Phi}=Q_{q}^{\Phi}$, so the map in \eqref{mapping to be semilattice homo} is well defined. If $pP\cap qP=\emptyset$, then $\Phi(\KK(p,p))\Phi(\KK(q,q))=0$ by Nica covariance. Thus $Q_{p}^{\Phi}Q_{q}^{\Phi}=0$, as claimed. Now suppose that $p\leq  q$ for some $p,q\in P$, and let $a\in\KK(p,p)$. Then the equality $\Phi_{p,p}(a)\Phi_{q,q}(b)=\Phi_{q,q}((a\otimes 1_{p^{-1}q}) b)$ implies that  $\Phi_{p,p}(a)Q_{q}^{\Phi}= Q_{q}^{\Phi} \Phi_{p,p}(a)Q_{q}^{\Phi}$.
By passing to adjoints we get $Q_{q}^{\Phi}\Phi_{p,p}(a)= Q_{q}^{\Phi} \Phi_{p,p}(a)Q_{q}^{\Phi}$. Thus $Q_{q}^{\Phi}\in \Phi_{p,p}(\KK(p,p))'$. It  follows from the definition of  $\overline{\Phi}_{p,p}$, cf.  \eqref{formula defining extensions of right tensor representations}, that
$\Phi_{p,p}(\KK(p,p))' \subseteq \overline{\Phi}_{p,p}(\LL(p,p))'$. Hence
\begin{equation}\label{don't know what and what for}
Q_{q}^{\Phi}\in  \overline{\Phi}_{p,p}(\LL(p,p))',
\end{equation}
and in particular, $Q_{p}^{\Phi}Q_{q}^{\Phi}=Q_{q}^{\Phi}Q_{p}^{\Phi}$.

Let now $p, q, s \in P$ and $a\in \KK(p,q)$.   If $sP\cap qP=\emptyset$, then $\Phi(a)Q^\Phi_s=\Phi(a)Q^\Phi_qQ^\Phi_s=0$ by Lemma~\ref{about approximate units}. Assume then that $sP\cap qP=rP$.
For any $b\in \KK(s,s)$ we have
\begin{align*}
%\overline{\Phi} (a) \Phi(b)& =
\Phi (a) \Phi(b)& =\Phi \big( (a \otimes 1_{q^{-1}r}) (b\otimes 1_{s^{-1}r})\big)=\overline{\Phi}(a\otimes 1_{q^{-1}r}) \overline{\Phi}( b\otimes 1_{s^{-1}r})
\\
&=\text{s-}\lim \overline{\Phi}(a\otimes 1_{q^{-1}r})\Phi(\mu_{\lambda}^r) \overline{\Phi}(b\otimes 1_{s^{-1}r})
\\
&
=
\text{s-}\lim \overline{\Phi}(a\otimes 1_{q^{-1}r})\Phi(\mu_{\lambda}^r (b\otimes 1_{s^{-1}r}))
\\
&=\text{s-}\lim \overline{\Phi}(a\otimes 1_{q^{-1}r})\Phi(\mu_{\lambda}^r)  \Phi(b)
\\
&=\overline{\Phi}(a\otimes 1_{q^{-1}r}) \Phi(b)
\end{align*}
 by construction of $\overline{\Phi}$ in \eqref{formula defining extensions of right tensor representations}, and Nica covariance of $\Phi$. Claim \eqref{action on projections general} follows.
It implies \eqref{action on projections general2}   because  $\overline{\Phi}(d)=Q^\Phi_s\overline{\Phi}(d)$. If $t \geq s$, then $\overline{\Phi}(d)Q^\Phi_t= Q^\Phi_t\overline{\Phi}(d)$ by \eqref{don't know what and what for}.
\end{proof}

\begin{lem}\label{properties of family projections3}
The assignment $pP \mapsto Q^\Phi_{\langle p\rangle}$ and $\emptyset \mapsto  0$ is a homomorphism $J(P) \longmapsto \Proj(\B(H))$ of semilattices, i.e. $Q^\Phi_{\langle p\rangle}  Q^\Phi_{\langle q\rangle}=Q^\Phi_{\langle r\rangle}$  when  $pP\cap qP=rP$ and $Q^\Phi_{\langle p\rangle}  Q^\Phi_{\langle q\rangle}=0$ for $pP\cap qP=\emptyset$.
In particular, $Q^\Phi_{\langle p\rangle}$ for $p\in P$ are mutually commuting projections. Moreover,
\begin{equation}\label{projections in the commutant}
Q^\Phi_{\langle q\rangle} \in \Phi(\KK(p,p))'\,\,\, \text{ for all }p,q\in P.
\end{equation}
If  $\overline{\Phi}:\LL\to \B(H)$ is Nica covariant (which happens for instance when $\KK=\LL$) then
\begin{equation}\label{action on projections special:two}
\overline{\Phi}(a)Q^\Phi_{\langle s\rangle}=\begin{cases}
\overline{\Phi}(a\otimes 1_{q^{-1}r})  & \textrm{ if } sP\cap qP=rP \textrm{ for some } r\in P,
\\
0 & \textrm{ if }sP\cap qP=\emptyset
\end{cases}
\end{equation}
for every $p, q, s \in P$ and $a\in \LL(p,q)$. In particular,
$
\{Q^\Phi_{\langle p\rangle}\}_{p\in P}\subseteq M(B_e^{\overline{\Phi}})\subseteq  M (C^*(\overline{\Phi}(\LL))).
$
% In particular, $\{Q^\Phi_{\langle p\rangle}\}_{p\in P}\subseteq M(B_e^{\overline{\Phi}})\subseteq  M (C^*(\overline{\Phi}(\LL))). $
\end{lem}
\begin{proof}
 By Nica covariance of $\Phi$,   $C_p^*(\Phi):=\clsp\{\Phi(\KK(w,w)):  w\in pP\}$
is a $C^*$-algebra, for each $p\in P$.  We claim that, for every $p,q\in P$ we have
\begin{equation}\label{semi-lattice of C*-algebras}
C_p^*(\Phi)  C_q^*(\Phi)=C_p^*(\Phi)\cap   C_q^*(\Phi)=\begin{cases}
C_r^*(\Phi) & \textrm{ if } pP\cap qP=rP \textrm{ for some } r\in P,
\\
0 & \textrm{ otherwise}.
\end{cases}
\end{equation}
Indeed, if $pP\cap qP=\emptyset$, then  for any $w\in pP$ and $v\in  qP$ we have $wP\cap vP=\emptyset$ and hence $\Phi(\KK(w,w))\Phi(\KK(v,v))=\{0\}$ by Nica covariance. Accordingly, in this case  $C_p^*(\Phi) C_q^*(\Phi)=\{0\}$.  Assume next that $pP\cap qP=rP$ for some $r\in P$. Then  $C_r^*(\Phi)\subseteq C_q^*(\Phi)\cap C_p^*(\Phi)\subseteq C_q^*(\Phi) C_p^*(\Phi)$. The reverse inclusion $C_q^*(\Phi) C_p^*(\Phi) \subseteq C_r^*(\Phi)$ readily follows from Nica covariance of $\Phi$.

Plainly, for every $p\in P$, the  essential space $C_p^*(\Phi)H$ for $C_p^*(\Phi)$ is equal to  $Q^\Phi_{\langle p \rangle}H$.
 In particular, $Q^\Phi_{\langle p\rangle}$ is a strong limit of any approximate unit in $C_p^*(\Phi)$, and \eqref{semi-lattice of C*-algebras} justifies the claim about the semilattice homomorphism. Morever, \eqref{semi-lattice of C*-algebras}  implies that for every $p,q\in P$ and $a\in \KK(p,p)$ we have  $\Phi(a)Q_{\langle q\rangle}^{\Phi}=Q_{\langle q\rangle}^{\Phi}\Phi(a)Q_{\langle q\rangle}^{\Phi}$. By passing to adjoints we get  $Q_{\langle q\rangle}^{\Phi}\Phi(a)=Q_{\langle q\rangle}^{\Phi}\Phi(a)Q_{\langle q\rangle}^{\Phi}$ and therefore $\Phi(a)Q_{\langle q\rangle}^{\Phi}=Q_{\langle q\rangle}^{\Phi}\Phi(a)$ which proves \eqref{projections in the commutant}.

 Suppose now that $\overline{\Phi}:\LL\to \B(H)$ is Nica covariant. Let $p, q, s \in P$ and $a\in \LL(p,q)$. If  $sP\cap qP=\emptyset$, then $\overline{\Phi}(a) C_s^*(\Phi)=\{0\}$ by Nica covariance, and therefore $\overline{\Phi}(a)Q^\Phi_{\langle s\rangle}=0$.
Assume that $sP\cap qP=rP$.  Let $w\in sP$ and $b\in \KK(w,w)$. Since $
rP\cap wP=(sP\cap qP) \cap wP=sP\cap (qP \cap wP)=sP\cap tP$ for some $t\in P $, using   Nica covariance of $\overline{\Phi}$ twice we get
$$
\overline{\Phi}(a) \Phi(b)= \overline{\Phi}\big((a\otimes 1_{q^{-1}t}) (b\otimes 1_{w^{-1}t})\big)=\overline{\Phi}(a\otimes 1_{q^{-1}r}) \Phi(b).
$$
This implies that $\overline{\Phi}(a)Q^\Phi_{\langle s\rangle}=\overline{\Phi}(a\otimes 1_{q^{-1}r})Q^\Phi_{\langle s\rangle}$. Since $\overline{\Phi}(a\otimes 1_{q^{-1}r})=\overline{\Phi}(a\otimes 1_{q^{-1}r})Q^\Phi_{\langle r\rangle}$ and $Q^\Phi_{\langle r\rangle} \leq Q^\Phi_{\langle s\rangle}$, we get $\overline{\Phi}(a)Q^\Phi_{\langle s\rangle}=
\overline{\Phi}(a\otimes 1_{q^{-1}r})$. This proves \eqref{action on projections special:two}.

Relation  \eqref{action on projections special:two} readily implies $ B_e^{\overline{\Phi}} Q^\Phi_{\langle p\rangle}  \subseteq  B_e^{\overline{\Phi}}$. By taking adjoints we obtain $Q^\Phi_{\langle p\rangle} B_e^{\overline{\Phi}}\subseteq B_e^{\overline{\Phi}}$.  Thus assuming the standard identification $
M(B_e^{\overline{\Phi}})=\{a\in Q^\Phi_{\langle e\rangle}\B(H)Q^\Phi_{\langle e\rangle}: aB_{e}^{\overline{\Phi}},\, B_{ e}^{\overline{\Phi}}a \subseteq B_e^{\overline{\Phi}} \}
$ we conclude that $Q^\Phi_{\langle p\rangle}\in M(B_e^{\overline{\Phi}})$.

%Using \eqref{action on projections special:two} one gets that $\{Q^\Phi_{\langle p\rangle}\}_{p\in P}\subseteq M(B_e^{\overline{\Phi}})$  as in the proof of Proposition \ref{properties of family projections2}.
\end{proof}
% For the purposes of the next proposition we say that the map \eqref{mapping to be semilattice homo} is pre-order preserving if $pP\subseteq  qP$ implies  that  $Q_{p}^{\Phi} \leq Q_{q}^{\Phi}$,  for all $p,q\in P$.

\begin{prop}\label{properties of family projections2} %Let $\overline{\Phi}:\LL\to \B(H)$ be the extension  of $\Phi$.
The following conditions are equivalent:

\textnormal{(i)} $Q^\Phi_{p}=Q^\Phi_{\langle p\rangle}$  for every $p\in P$;

 \textnormal{(ii)} The map \eqref{mapping to be semilattice homo}
is a semilattice homomorphism;

\textnormal{(iii)}
The map  \eqref{mapping to be semilattice homo}  is a pre-order homomorphism  ($pP\subseteq  qP$ implies  $Q_{p}^{\Phi} \leq Q_{q}^{\Phi}$);
% for all $p,q\in P$.

\textnormal{(iv)} The extension $\overline{\Phi}:\LL\to \B(H)$ of $\Phi$ is Nica covariant, and
\begin{equation}\label{action on projections special}
\overline{\Phi}(a)Q^\Phi_s=\begin{cases}
\overline{\Phi}(a\otimes 1_{q^{-1}r})  & \textrm{ if } sP\cap qP=rP \textrm{ for some } r\in P,
\\
0 & \textrm{ if }sP\cap qP=\emptyset,
\end{cases}
\end{equation}
for every $p, q, s \in P$ and $a\in \LL(p,q)$.

In particular, if the above equivalent conditions hold, then  $Q^\Phi_p=Q^\Phi_{\langle p\rangle}\in \overline{\Phi}(\LL(q,q))'$ for all $p,q\in P$, and we have $
\{Q^\Phi_{ p}\}_{p\in P}=\{Q^\Phi_{\langle p\rangle}\}_{p\in P}\subseteq M(B_e^{\overline{\Phi}})\subseteq  M (C^*(\overline{\Phi}(\LL))).
$
\end{prop}
\begin{proof}
Implication (i)$\Rightarrow$(ii) follows from Lemma \ref{properties of family projections3}. Implication  (ii)$\Rightarrow$(iii) is trivial.
 %In view of Lemma \ref{properties of family projections},  it suffices to show that $pP\cap qP=rP$ implies that  $Q^\Phi_p  Q^\Phi_q=Q^\Phi_r$. Let then $pP\cap qP=rP$. Since \eqref{mapping to be semilattice homo} is a pre-order homomorphism we have $Q^\Phi_r \leq  Q^\Phi_q , Q^\Phi_p$. Hence  $Q^\Phi_r H \subseteq  Q^\Phi_pH\cap Q^\Phi_q H$. On the other hand, if $h\in Q^\Phi_pH\cap Q^\Phi_q H$ then $h=\Phi(a)\Phi(b)h'$ where $a\in \K(p,p)$,  $b\in \K(q,q)$ and $h'\in H$. By Nica covariance, $h=\Phi(a)\Phi(b)h'=\Phi((a\otimes 1_{p^{-1}r}) (b\otimes 1_{q^{-1}r}))h'\in \Phi(\KK(r,r))H=Q^\Phi_r H$. Accordingly, $Q^\Phi_r H =  Q^\Phi_pH\cap  Q^\Phi_q H$ and thus $Q^\Phi_p  Q^\Phi_q=Q^\Phi_r$.

We prove that (iii)$\Rightarrow$(iv).
 Let  $p, q, s \in P$ and $a\in \LL(p,q)$.   If $sP\cap qP=\emptyset$, then $Q^\Phi_qQ^\Phi_s=0$ by Lemma \ref{properties of family projections}, and hence
$\overline{\Phi}(a)Q^\Phi_s=\overline{\Phi}(a)Q^\Phi_qQ^\Phi_s=0$. Assume then that $sP\cap qP=rP$.
Note that $\overline{\Phi}(a)$ is a strong limit of the net $\Phi (a \mu_{\lambda}^q)=\overline{\Phi} (a) \Phi(\mu_{\lambda}^q) $, where $\{\mu_{\lambda}^q\}_{\lambda\in \Lambda}$ is an approximate unit in $\KK(q,q)$. Thus \eqref{action on projections special} follows from the calculations
\begin{align*}
\overline{\Phi} (a)Q^\Phi_s & =\text{s-}\lim\Phi (a \mu_{\lambda}^q) Q^\Phi_s
\stackrel{\eqref{action on projections general} }{=}\text{s-}\lim \overline{\Phi}\big((a \mu_{\lambda}^q)\otimes 1_{q^{-1}r}\big)Q^\Phi_s
\\
& =\text{s-}\lim \overline{\Phi}(a\otimes 1_{q^{-1}r})\overline{\Phi}(\mu_{\lambda}^q\otimes 1_{q^{-1}r}) Q^\Phi_s
\stackrel{\eqref{action on projections general} }{=}\text{s-}\lim \overline{\Phi}(a\otimes 1_{q^{-1}r})\Phi(\mu_{\lambda}^q) Q^\Phi_s
\\
&= \overline{\Phi}(a\otimes 1_{q^{-1}r})Q^\Phi_r Q^\Phi_q Q^\Phi_s
\stackrel{(iii)}{=}\overline{\Phi}(a\otimes 1_{q^{-1}r})Q^\Phi_r
=\overline{\Phi}(a\otimes 1_{q^{-1}r}).
\end{align*}
To prove Nica covariance,  let $a\in \LL(p,q)$, $b\in   \LL(s,t) $. By (i), if $qP\cap sP=rP$, then
$$
\overline{\Phi}(a)\overline{\Phi}(b)
=\overline{\Phi}(a)Q^\Phi_qQ^\Phi_s\overline{\Phi}(b)
=\overline{\Phi}(a)Q^\Phi_r\overline{\Phi}(b),
$$
which is $\overline{\Phi} \left((a \otimes 1_{q^{-1}r}) (b\otimes 1_{s^{-1}r})\right)$
 by the previous paragraph. The calculations above also show  that $\overline{\Phi}(a)\overline{\Phi}(b)=0$ when $qP\cap sP=\emptyset$.

To see that (iv)$\Rightarrow$(i), note that  \eqref{action on projections special:two} and  \eqref{action on projections special} imply that for every $p,q\in P$ and $a\in \KK(p,p)$ we have  $  Q^\Phi_{\langle p\rangle}\Phi(a)=Q^\Phi_p \Phi(a)$. Since $Q^\Phi_{\langle p\rangle}$ and
$Q^\Phi_p$ are zero on the orthogonal complement of  $C^*(\Phi(\KK))H$, this implies that $  Q^\Phi_{\langle p\rangle}=Q^\Phi_p $.

This proves the equivalence of  (i)-(iv). Applying  \eqref{action on projections special} and its adjoint to  $a\in \LL(q,q)$ we get $Q^\Phi_p\in \overline{\Phi}(\LL(p,p))'$, for all $p,q\in P$. The remaining part follows from Lemma \ref{properties of family projections3}.
\end{proof}
%Next we associate to $\Phi:\KK\to\B(H)$ a second family of projections which always yield a homomorphic image of the semilattice $J(P)$, however, in general  they do not satisfy the analogue of the relation \eqref{action on projections general}.
It is possible to cook up an example where the above equivalent conditions fail:
\begin{ex}\label{degenerate example}
Let $\LL=\KK=\{\KK(n,m)\}_{ n,m\in \N}$ be a  $C^*$-precategory where $\KK(n,m)=\{0\}$ for all $n\neq m$ and $\KK(n,n)$ are arbitrary (non-zero) $C^*$-algebras. The multiplication in $\LL$ is zero. We equip  $\LL$ with a right tensoring which is also zero. Then
$\NT_{\LL}(\KK)=\NT(\LL)=\TT(\LL)$ is naturally isomorphic with the direct sum $\bigoplus_{n\in \N} \KK(n,n)$. In particular,
taking any faithful representations $\Phi_{n,n}:\KK(n,n)\to \B(H_n)$ and setting $H:=\bigoplus_n H_n$ and  $\Phi_{n,m}:= 0$ for $n\neq m$, $n,m\in \N$,  we get an injective Nica covariant representation $\Phi$ of $\KK$. For each $n\in \N$, $Q^\Phi_{n}$ is the projection onto $H_n$ and
 $Q^\Phi_{\langle n\rangle}$ is the projection onto $\bigoplus_{k \geq n } H_k$.  Note that $\overline{\Phi}=\Phi$ is Nica covariant and \eqref{action on projections special:two} is satisfied. However, \eqref{action on projections special} fails.
\end{ex}
In order to avoid situations as in Example \ref{degenerate example}, a nondegeneracy condition was introduced
 in \cite[Definition 3.8]{kwa-doplicher}\footnote{We note that there are  typos in \cite[Definition 3.8]{kwa-doplicher},
one should put there $m=n$.}, in the case $P=\N$.
We generalize  this notion to arbitrary LCM semigroups.
Virtually all examples considered in \cite{kwa-larII} will satisfy this condition.
\begin{defn}
An ideal $\KK$ in a right-tensor $C^*$-precategory $\LL$ is $\otimes 1$-\emph{nondegenerate} if
$$
(\KK(p,p)\otimes 1_{r})\KK(pr,pr)=\KK(pr,pr)\textrm{ for every } p\in P\setminus{P^*} \textrm{ and } r\in P.
$$
\end{defn}
\begin{prop}\label{properties of a semi lattice of projections}
If $\KK$ is an  $\otimes 1$-nondegenerate ideal in $\LL$ then for every   Nica covariant representation $\Phi$ of $\KK$ the equivalent conditions  in Proposition \ref{properties of family projections2} hold true.
\end{prop}
\begin{proof} We show condition (i) in Proposition \ref{properties of family projections2}.
By definitions,  $Q^\Phi_{p}\leq Q^\Phi_{\langle p\rangle}$ for all $p\in P\setminus P^*$ and $Q^\Phi_{p}= Q^\Phi_{\langle p\rangle}$ for all $p\in P^*$.
By nondegeneracy of $\KK$, for any $p\in P\setminus P^*$ and any $w\in pP$ we have $\KK(w,w)=(\KK(p,p) \otimes 1_{p^{-1}w})\KK(w,w)$. Thus by Nica covariance we get $\Phi(\KK(w,w))=\Phi(\KK(p,p) \otimes 1_{p^{-1}w})\KK(w,w))=\Phi(\KK(p,p))\Phi(\KK(w,w))$,
which implies that $Q^\Phi_{p}\geq Q^\Phi_{\langle p\rangle}$.
\end{proof}

\section{Representations generating exotic $C^*$-algebras - uniqueness theorem}\label{The main result section}

We fix a well-aligned ideal $\KK$ in a right-tensor $C^*$-precategory $(\LL, \{\otimes 1_r\}_{r\in P})$. In this section we study conditions implying that a Nica covariant representation of $\KK$ generates an exotic Nica-Toeplitz $C^*$-algebra  of  $\KK$. In the presence of amenability this will lead  to isomorphism theorems for the universal Nica-Toeplitz algebra $\NT_{\LL}(\KK)$ as well.

We start by introducing the key condition that we call $(C)$. Here the letter C  stands for both compression and Coburn.
\begin{defn}\label{defn:Condition (C)} Let $\Phi:\KK\to \B(H)$ be a Nica covariant representation of $\KK$ on a Hilbert space, and let $\{Q^\Phi_{\langle p\rangle }\}_{p\in P}\subseteq \B(H)$ be the projections introduced in Definition \eqref{definition of family projections2}. We say that  $\Phi$ satisfies  \emph{condition $(C)$}  if
\begin{equation}\label{Coburn condition}
%\eqno{(C)\qquad \qquad\qquad \qquad }
\begin{array}{l}
\text{for every   $p\in P$  and $q_1,...,q_n\in P$  such that  $p\not\geq q_i$, for $i=1,...,n$, $n\in \N$, }
\\[4pt]
\text{the representation $\KK(p,p)\ni a \longmapsto  \Phi(a) \prod_{i=1}^{n}(1-Q^\Phi_{\langle q_i\rangle})$ is faithful. }
\end{array}\qquad
\end{equation}
\end{defn}
\begin{rem}\label{remark to be reformulated0} Since projections $\{Q^\Phi_{\langle p\rangle }\}_{p\in P}$ mutually commute, we have $(1-\bigvee_{i=1}^nQ^\Phi_{\langle q_i\rangle})=  \prod_{i=1}^{n}(1-Q^\Phi_{\langle q_i\rangle})$. By \eqref{projections in the commutant},  $(1-Q^\Phi_{\langle q_i\rangle}) \in \Phi(\KK(p,p))'$ and hence $\KK(p,p)\ni a \longmapsto  \Phi(a) \prod_{i=1}^{n}(1-Q^\Phi_{\langle q_i\rangle})$ is indeed a representation of the $C^*$-algebra $\KK(p,p)$.
\end{rem}
Every well-aligned ideal $\KK$ admits a representation satisfying condition $(C)$.
\begin{prop}\label{induced representations satisfying condition (C)} Let  $\pi$ be a representation of $\NT_{\LL}^r(\KK)$ induced by the Fock Hilbert module $\FF_\KK$ from a faithful representation  $\pi_0:\bigoplus_{t\in P}\KK(t,t)\to \B(H)$. Then $\pi\circ T: \KK \to \B(\FF_\KK\otimes_{\pi_0} H)$ is a Nica covariant representation satisfying condition $(C)$.
\end{prop}
\begin{proof} Recall that $\pi:\LL(\FF_\KK)\to \B(\FF_\KK\otimes_{\pi_0} H)$ is a  faithful representation  given by the formula
$a(x\otimes_{\pi_0} h):=(ax)\otimes_{\pi_0} h$, $a\in \LL(\FF_\KK)$, $x\in \FF_\KK$, $h\in H$.

Let   $p\in P$  and $q_1,...,q_n\in P$ be such that  $p\not\geq q_i$, for $i=1,...,n$. Consider projections  $\{Q^{T}_{{\langle q_i\rangle}}\}_{i=1}^n$ introduced in Lemma~\ref{lemma-proj-T-semillatice}. Since the representation $\KK(p,p)\ni a \longmapsto  T^{p,p}_{p,p}(a)\in \LL(X_{p,p})$ is faithful,  so is $\KK(p,p)\ni a \longmapsto  T(a) \prod_{i=1}^{n}(1-Q^{T}_{{\langle q_i\rangle}})$. With the definition of projections from Lemma~\ref{properties of family projections3}, it readily follows that  $Q^{\pi\circ T}_{\langle q_i\rangle}\leq \pi( Q^{T}_{\langle q_i\rangle})$, for $i=1,...,n$. Hence $\pi(\prod_{i=1}^{n}(1-Q^{T}_{\langle q_i\rangle}))\leq  \prod_{i=1}^{n}(1-Q^{\pi\circ T}_{\langle q_i\rangle})$ and therefore
$$
T(a) \prod_{i=1}^{n}(1-Q^{T}_{\langle q_i\rangle})\neq 0 \,\, \Longrightarrow\,\, \pi\Big(T(a) \prod_{i=1}^{n}(1-Q^{T}_{\langle q_i\rangle})\Big)\neq 0 \,\, \Longrightarrow\,\, \pi(T(a)) \prod_{i=1}^{n}(1-Q^{\pi\circ T}_{\langle q_i\rangle})\neq 0.
$$
Thus $\KK(p,p)\ni a \longmapsto  \pi(T(a)) \prod_{i=1}^{n}(1-Q^{\pi\circ T}_{\langle q_i\rangle})$ is faithful.
\end{proof}
Plainly, every  Nica covariant representation  satisfying  condition $(C)$ is injective and Toeplitz covariant in the sense of \eqref{Toeplitz condition}. In the converse direction we have the following:

\begin{prop}\label{Proposition 7.8} If $\KK\otimes 1\subseteq \KK$, then a Nica covariant representation $\Phi:\KK\to B(H)$  satisfies  condition $(C)$ if and only if $\Phi$ is injective and Toeplitz covariant.
\end{prop}
\begin{proof} The `only if' part is clear.
Suppose than that $\Phi$ is injective and Toeplitz covariant.
By Corollaries \ref{Nica-Toeplitz representation corollary} and \ref{the reduced core} there is an  isomorphism $\Phi_*:B_e^\Phi\to B_e^T $ making  the diagram \eqref{restricted diagram} commute.  Denote by $\overline{\Phi}_*:M(B_e^\Phi)\to M(B_e^T) $  the strictly continuous extension of $\Phi_*$. Let  $p\in P$  and $q_1,...,q_n\in P$ be such that  $p\not\geq q_i$.
We noticed in the proof of Proposition \ref{induced representations satisfying condition (C)} that the representation $\KK(p,p)\ni a \longmapsto T_{p,p}(a) \prod_{k=1}^{n} (1-Q^{T}_{\langle q_i\rangle})$ is faithful. By Lemma \ref{properties of semilattice projections in the reduced} we may view $Q^{T}_{\langle q_i\rangle}$  as elements of  $M(B^T_e)$. Similarly, by Lemma \ref{properties of family projections3}  applied to $\LL=\KK$, one concludes that we may treat  projections $Q^\Phi_{\langle q_i\rangle}$ as elements of $M(B_e^\Phi)$, and  then  $\overline{\Phi}_*(Q^\Phi_{\langle q_i\rangle})=Q^{T}_{\langle q_i\rangle}$. Thus
$$
T_{p,p}(a) \prod_{k=1}^{n} (1-Q^{T}_{\langle q_i\rangle}) =\overline{\Phi}_*\Bigl(\Phi_{p,p}(a) \prod_{k=1}^{n} (1-Q^\Phi_{\langle q_i\rangle })\Bigr)\quad\textrm{ for all } a \in \KK(p,p).
$$
Therefore $\KK(p,p)\ni a \longmapsto \Phi_{p,p}(a) \prod_{k=1}^{n} (1-Q^\Phi_{\langle q_i\rangle})$ is faithful, and  $\Phi$ satisfies $(C)$.
\end{proof}
\begin{cor}\label{Corollary 7.99} Let  $\Phi:\KK\to B(H)$ be a Nica covariant representation. Suppose that $\KK$ is essential in $\LL$ and that the extended representation $\overline{\Phi}: \LL\to B(H)$ is Nica covariant (which is automatic when  $\KK$ is  $\otimes 1$-non-degenerate). The following conditions are equivalent:

\textnormal{(i)} $\Phi$  satisfies  condition $(C)$;

\textnormal{(ii)} $\overline{\Phi}$ satisfies  condition $(C)$;

\textnormal{(iii)} $\overline{\Phi}$ is injective and Toeplitz covariant.
\end{cor}
\begin{proof} Equivalence (ii)$\Leftrightarrow$(iii) follows
from Proposition \ref{Proposition 7.8} applied to $\LL$ and $\overline{\Phi}$.
To see that (i)$\Leftrightarrow$(ii)  let    $p\in P$  and $q_1,...,q_n\in P$  such that  $p\not\geq q_i$, for $i=1,...,n$.
Note that $Q^\Phi_{\langle q\rangle}=Q^{\overline{\Phi}}_{\langle q\rangle}$, for $q\in P$. Hence since $\KK(p,p)$ is an essential ideal in the $C^*$-algebra $\LL(p,p)$, the  representation $\KK(p,p)\ni a \longmapsto  \Phi(a) \prod_{i=1}^{n}(1-Q^\Phi_{\langle q_i\rangle})$ is faithful if and only if the representation $\LL(p,p)\ni a \longmapsto  \overline{\Phi}(a) \prod_{i=1}^{n}(1-Q^{\overline{\Phi}}_{\langle q_i\rangle})$ is faithful.
\end{proof}
%A possible interpretation of condition (C) is given in the following:

Next we introduce an auxiliary condition which describes properties of the (not necessarily Nica covariant) representation of $\LL$ that extends a Nica covariant representation of $\KK$.

\begin{defn}\label{defn:Condition (C')}  Let $\Phi:\KK\to \B(H)$ be a Nica covariant representation on a Hilbert space, and let $\overline{\Phi}:\LL\to \B(H)$ be the extension from Proposition \ref{extensions of representations on Hilbert spaces0}. Let $\{Q^\Phi_p\}_{p\in P}\subseteq \B(H)$  be the projections given by \eqref{definition of family projections}.
We say that  $\overline{\Phi}$ satisfies  \emph{condition $(C')$}  if
\begin{equation}\label{Coburn condition1}
\begin{array}{l}
\text{for every   $p\in P$  and $q_1,...,q_n\in P$  such that  $p\not\geq q_i$, for $i=1,...,n$,  }
\\[4pt]
\text{ $\|(1-\bigvee_{i=1}^n Q^\Phi_{q_i})\overline{\Phi}(a) (1-\bigvee_{i=1}^nQ^\Phi_{q_i})\|=\|\overline{\Phi}(a)\|$ for all  $a \in \LL(p,p)$. }
\end{array}\qquad
\end{equation}
\end{defn}

\begin{prop}\label{condition C versus condition C'}
Let $\Phi:\KK\to \B(H)$ be a Nica covariant representation of $\KK$ and let $\overline{\Phi}:\LL\to \B(H)$ be the extended representation of $\LL$. Consider the following assertions:
\begin{itemize}
\item[(i)] $\Phi$ satisfies  condition $(C)$;
\item[(ii)]  $\Phi$ is injective and $\overline{\Phi}$ satisfies condition  $(C')$;
\end{itemize}
Then (i)$\Rightarrow$(ii). If $\KK$ is $\otimes 1$-non-degenerate or  $\KK\otimes 1\subseteq \KK$, then (i)$\Leftrightarrow$(ii).
\end{prop}
\begin{proof}

(i)$\Rightarrow$(ii).  Plainly,  condition $(C)$ implies that $\Phi$ is injective. Let  $p\in P$  and $q_1,...,q_n\in P$ be such that  $p\not\geq q_i$. Let $T^{p,p}_{p,p}:\LL(p,p)\to \LL(X_{p,p})=\LL(\KK(p,p))$ be the homomorphism  given by multiplication from left, cf. Lemma \ref{lemma about Fock operators}. By \eqref{kernel of the extended} and the construction of $T$, we have  that $\ker T^{p,p}_{p,p}= \ker \overline{\Phi}_{p,p}$. Similarly, using \eqref{Coburn condition},  we see that the kernel of the representation $\LL(p,p)\ni a \longmapsto  \overline{\Phi}(a) \prod_{i=1}^{n}(1-Q^\Phi_{\langle q_i\rangle})$ coincides with the kernel of $T^{p,p}_{p,p}$. Thus
$\|\overline{\Phi}(a)\prod_{i=1}^{n}(1-Q^\Phi_{\langle q_i\rangle})\|=\|\overline{\Phi}(a)\|$ for all $a\in \LL(p,p)$.
 Since  $\prod_{i=1}^{n}(1-Q^\Phi_{\langle q_i\rangle})\leq (1-\bigvee_{i=1}^n Q^\Phi_{q_i})$, this implies that $\|(1-\bigvee_{i=1}^n Q^\Phi_{q_i})\overline{\Phi}(a) (1-\bigvee_{i=1}^nQ^\Phi_{q_i})\| \geq \|\overline{\Phi}(a)\|$ for all $a\in \LL(p,p)$. The reverse inequality is clear.

(ii)$\Rightarrow$(i).  Let  $p\in P$  and $q_1,...,q_n\in P$ be such that  $p\not\geq q_i$.  For every $a\in \KK(q_i, q_i)$, $j=i,\dots,n$, we have  $\|(1-\bigvee_{i=1}^n Q^\Phi_{q_i})\overline{\Phi}(a) (1-\bigvee_{i=1}^nQ^\Phi_{q_i})\|=0$. Hence   condition $(C')$ implies that $\Phi$ is Toeplitz covariant, and therefore if $\KK\otimes 1\subseteq \KK$ then   $\Phi$ satisfies condition $(C)$ by Proposition \ref{Proposition 7.8}.
Suppose then that $\KK$ is  $\otimes 1$-non-degenerate.
By Proposition \ref{properties of a semi lattice of projections} we have $(1-\bigvee_{i=1}^n Q^\Phi_{q_i})=\prod_{i=1}^{n}(1-Q^\Phi_{\langle q_i\rangle})$. Hence for any $a\in \KK(p,p)$ we get
$$
\|\Phi(a)\prod_{i=1}^{n}(1-Q^\Phi_{\langle q_i\rangle})\|=\|(1-\bigvee_{i=1}^n Q^\Phi_{q_i})\overline{\Phi}(a) (1-\bigvee_{i=1}^nQ^\Phi_{q_i})\|=\|\Phi(a)\|=\|a\|.
$$
Thus $\Phi$ satisfies condition $(C)$.
\end{proof}

In order to get a uniqueness result in the case when the group  $P^*\subseteq P$ is non-trivial, we need to impose certain additional conditions on $\KK$ and $\LL$. It seems that there is no single  candidate for such a condition available on the scene. However, there are several natural conditions that  can be applied in different situations. For instance, if $P$ can be embedded into a group $G$ via a monomorphism  $\theta:P \to G$, and $\B^\theta$ is the Fell bundle described in Proposition \ref{gauge co-action}, then Theorem \ref{amenability main result} and results of \cite{KM}, \cite{KS} indicate that a natural condition is aperiodicity of
 the Fell bundle $\B^\theta$.

We recall from \cite[Definition 4.1]{KS} that a
 Fell bundle $\B=\{B_g\}_{g\in G}$ is \emph{aperiodic} if for  each $t\in G\setminus\{e\}$, each $b_t\in B_t$ and every non-zero hereditary
subalgebra $D$ of $B_e$,
$$
\inf \{\|db_td\| : d\in D^+,\,\, \|d\|=1\}=0.
$$
For general semigroups $P$ we may consider the following modification (and in fact generalization) of this notion expressed in terms of $\LL$ and $\KK$. We recall, see Lemma \ref{lemma on automorphic actions on ideals}, that  the group $P^*$ of invertible elements acts on $\KK$ by automorphisms.
\begin{defn}\label{definition of aperiodicity for right-tensor categories}
We say that the group $\{\otimes 1_{x}\}_{x\in P^*}$ of automorphisms of  $\LL$ is  \emph{aperiodic on $\KK$} if  for every $p\in P$, every non-zero hereditary $C^{*}$-subalgebra $D\subseteq \K(p,p)$ and every $b\in \KK(px,p)$ where $x\in P^*\setminus\{e\}$ we have
\begin{equation}\label{aperiodicity condition}
\inf \{\|(a\otimes{1_x})b a\| : a\in D^+,\,\, \|a\|=1\}=0.
\end{equation}
\end{defn}
\begin{rem}\label{rem:to be used for aperiodicity} For a fixed $p\in P$ the spaces $B^{(p)}_x:=\KK(px,p)$, $x\in P^*$, equipped with multiplication and involution defined by
$ b_x \cdot b_y := b_x\otimes 1_y b_y$  and $(b_x)^*:=(b_x^*)\otimes 1_{x^{-1}}$,
 for $b_x \in  B^{(p)}_x, b_y\in B^{(p)}_y$,  form a Fell bundle $\mathcal{B}^{(p)}:=\{B^{(p)}_x\}_{x\in P^*}$ over $P^*$. In particular, the group $\{\otimes 1_{x}\}_{x\in P^*}$  is  aperiodic on $\KK$ if and only if for each $p\in P$ the Fell bundle $\mathcal{B}^{(p)}$ is aperiodic. \end{rem}
The  following lemma can be proved using a standard argument, cf. the proof of \cite[Lemma 5.2]{MS}. In view of Remark \ref{rem:to be used for aperiodicity}, it is a corollary of  \cite[Lemma 4.2]{KS}.
 \begin{lem}\label{Muhly Solel lemma}
Suppose that the group $\{\otimes 1_{x}\}_{x\in P^*}$ of automorphisms of  $\KK$ is  aperiodic. Take any $p\in P$ and let $F\subseteq P^*$ be a finite set containing $e$.  For every family of elements $a_x\in \KK(px,p)$ for $x\in F$ with $a_e=a_e^*$, and every $\varepsilon >0$ there  is an element $d\in \overline{a_e\KK(p,p)a_e}$, $\|d\|=1$, such that
$
 \|(d\otimes 1_{x}) a_x d\| <\varepsilon\,\, \textrm{ for  every }\,\, x\in F\setminus \{e\}\,\, \textrm{ and }\,\,\| d  a_e d\| > \|a_e\| -\varepsilon.
$
 \end{lem}
\begin{proof}
 By  passing to $-a_e$ if necessary, we may assume that $\|a_e\|$ is a spectral value of $a_e$. Then  applying to $a_e$ a non-decreasing continuous function that vanishes on $(-\infty,0]$ and is identity on a neighborhood of $\|a_e\|$ we may assume $a_e$ is positive.
Then the assertion follows from \cite[Lemma 4.2]{KS} applied to the Fell bundle $\mathcal{B}^{(p)}$ described in  Remark \ref{rem:to be used for aperiodicity}.
\end{proof}
We will also need the following fact.
 \begin{lem}\label{a periodicity for essentials}
Suppose that $\KK$ is an essential ideal in  $\LL$. Then the group $\{\otimes 1_{x}\}_{x\in P^*}$ of automorphisms    is  aperiodic on $\KK$ if and only if it is aperiodic on $\LL$.
 \end{lem}
\begin{proof}
In view of  Remark \ref{rem:to be used for aperiodicity} the assertion follows from \cite[Corollary 6.9]{KM}.
 \end{proof}
Our main results  relate the properties of a Nica covariant representation  of $\KK$ on a Hilbert space that have been introduced so far.  Recall that Toeplitz covariance was defined in \eqref{Toeplitz condition}; we think of it as being an algebraic condition. By contrast, condition $(C)$ can be viewed as a geometric condition.
	
 \begin{thm}\label{preludium to Nica Toeplitz uniqueness2}
Let  $\KK$  be a well-aligned ideal in a right-tensor $C^*$-precategory $(\LL, \{\otimes 1_r\}_{r\in P})$.
Consider the following conditions on
a Nica covariant representation $\Phi:\KK\to B(H)$:
\begin{itemize}
\item[(i)] $\Phi$ satisfies  condition $(C)$;
\item[(ii)] $\Phi$ is injective and $\overline{\Phi}:\LL\to \B(H)$ satisfies condition  $(C')$;
\item[(iii)] $\Phi$ generates an exotic  Nica-Toeplitz $C^*$-algebra $C^*(\Phi(\KK))$ of $\KK$;
\item[(iv)] $\Phi\rtimes P$ is injective on the core $C^*$-subalgebra $B_e^{i_\KK}$ of $\NT_{\LL}(\KK)$;
 \item[(v)]   $\Phi$ is injective and  Toeplitz covariant.
\end{itemize}
Then  $(i)\Rightarrow (ii)$ and $(iii)\Rightarrow (iv)\Leftrightarrow (v)$.
If  one of the following conditions holds:
\begin{itemize}
\item[(1)] $P^*=\{e\}$,
\item[(2)] the group $\{\otimes 1_{x}\}_{x\in P^*}$   is aperiodic on $\KK$ and $\KK$ is essential in $\LL_{\KK}$, cf. Lemma \ref{the right tensor precatory generated by K},
\item[(3)] there is a  unital monomorphism  $\theta:P \to G$ into a group $G$  and  the Fell bundle $\B^{\theta}$  from Proposition \ref{gauge co-action} is aperiodic,
\end{itemize}
then  $(ii)\Rightarrow (iii)$. Further, if $\KK\otimes 1\subseteq \KK$, then all five conditions are equivalent.
\end{thm}

\begin{proof}
Implications (i)$\Rightarrow$(ii) and (iii)$\Rightarrow$(v)$\Leftrightarrow$(iv) follow  respectively from Propositions \ref{condition C versus condition C'} and \ref{proposition to invoke} and Corollary \ref{Nica-Toeplitz representation corollary}. If $\KK\otimes 1\subseteq \KK$ we also have  (v)$\Rightarrow$(i) by Proposition \ref{Proposition 7.8}. Thus it remains to show that (ii)$\Rightarrow$(iii), provided one of the conditions (1)-(3) holds.

Suppose therefore that $\Phi$ is injective and $\overline{\Phi}:\LL\to \B(H)$ satisfies condition  $(C')$. In view of Proposition \ref{preludium to Nica Toeplitz uniqueness}, it suffices to show that
 \eqref{formula defining strange map} defines a bounded map $E^{\Phi}$ on  the dense $*$-subalgebra $C^*(\Phi(\KK))^0$ of $C^*(\Phi(\KK))$, cf. \eqref{black star algebra}. Moreover, it  suffices to define a bounded map $E^\Phi$ on the $\R$-linear subspace $C^*(\Phi(\KK))^0_{sa}$ of $C^*(\Phi(\KK))^0$
consisting of elements of the form
\begin{equation}\label{bla form of an element}
a=\sum_{p,q\in F}\Phi(a_{p,q}), \,\, \textrm{ where }a_{p,q}\in \KK(p,q), \,\,  a_{p,q}^*=a_{q,p},
\end{equation}
for $F\subseteq P$ finite and $p,q \in F$. Indeed,  such map on $C^*(\Phi(\KK))^0_{sa}$ extends via the formula $E^{\Phi}(a)=E^{\Phi}(\frac{a+a^*}{2}) + i E^{\Phi}(\frac{a-a^*}{2i})$ to a  bounded map $E^{\Phi}$ on $C^*(\Phi(\KK))^0$ satisfying \eqref{formula defining strange map}.

Let us then fix a self-adjoint element $a$ given by \eqref{bla form of an element}.  Denote by $Z$ the right hand side of \eqref{formula defining strange map}, cf. also \eqref{definition of Z}. Note that  $Z^*=Z$.
Our strategy is to show (by consecutive  compressions of  $a$ and its compressions) that for every $\varepsilon > 0$ we have
\begin{equation}\label{equation in the proof4}
\|Z\|-\varepsilon \leq\|\, a \|.
\end{equation}
To this end, we fix $\varepsilon > 0$, and note that by Lemma \ref{norm of an element in the core} we may find  an initial segment $C$ of $F$ and $w\in    \sigma(C) P$ with $w\in P_{F, C} $ such that
\begin{equation*}
\|Z\| -\varepsilon/2 \leq  \Big\|T^{w,w}_{w,w}\Big(\sum_{p,q\in C  \atop pq^{-1}w=w} a_{p,q}\otimes 1_{q^{-1}w}\Big)\Big\|.
\end{equation*}
  As noticed in the proof of Proposition~\ref{condition C versus condition C'}, faithfulness of $\Phi$ implies that $\ker T^{w,w}_{w,w}= \ker \overline{\Phi}_{w,w}$. Thus  we obtain
\begin{equation}\label{equation in the proof5}
\|Z\| -\varepsilon/2 \leq  \Big\| \overline{\Phi}\Big(\sum_{p,q\in C  \atop pq^{-1}w=w} a_{p,q}\otimes 1_{q^{-1}w}\Big)\Big\|.
\end{equation}
Clearly, \eqref{equation in the proof4} will follow from \eqref{equation in the proof5} if we  prove that
\begin{equation}\label{equation in the proof3}
\Big\| \overline{\Phi}\Big(\sum_{p,q\in C  \atop pq^{-1}w=w} a_{p,q}\otimes 1_{q^{-1}w}\Big)\Big\|-\frac \varepsilon4\leq \|a\|.
\end{equation}
 We fix the above $C$ and $w$. Put
\begin{align*}
F_{rest}&:=\{r\in P :tP\cap wP=rP, t\in F\setminus C \},
\\
F_{>w}&:=\{pq^{-1}w:pq^{-1}w=wx, \, x \in P\setminus P^*, p, q\in C\},
\\
F_{w}&:=\{pq^{-1}w:pq^{-1}w=wx, \, x \in  P^*, p, q\in C\}.
\end{align*}
Note that for  $s\in F_{w}$ we have $s\geq w \geq s$ and for $s\in F_{>w}$ we have $s \geq w$ and  $w\not \geq s$. Since $\omega\notin  \bigcup_{t\in F\setminus C} tP$,  for $r\in F_{rest}$ we  have $r \geq w$ and $w\not \geq r$.
Now suppose that $d\in \LL(w,w)$.
Using \eqref{action on projections general2} (and its adjoint) and that $\overline{\Phi}$ is a representation we  get
\begin{align*}
\overline{\Phi}(d) a \overline{\Phi}(d)
= &  \, \,  \sum_{p ,q\in C  \atop  pq^{-1}w\in F_{w}}  \overline{\Phi}(d)\overline{\Phi}( a_{p,q}\otimes 1_{q^{-1}w}) \overline{\Phi}(d)
\\
\, \, \, \,+&\, \,  \sum_{p ,q\in C  \atop  pq^{-1}w\in F_{>w}} \overline{\Phi}(d)\overline{\Phi}(a_{p,q}\otimes 1_{q^{-1}w} ) \overline{\Phi}(d)
\\
 \, \,\, \,+&\, \, \sum_{p\in F ,q\in F\setminus C\atop qP\cap wP=rP, r\in F_{rest}} \overline{\Phi}(d)\overline{\Phi}(a_{p,q}\otimes 1_{q^{-1}r})\overline{\Phi}(d)
\\
\, \, \, \,+& \, \, \sum_{p\in F\setminus C,q\in C\atop pP\cap wP=rP, r\in F_{rest}}\overline{\Phi}(d)\overline{\Phi}(a_{p,q}\otimes 1_{p^{-1}r})\overline{\Phi}(d).
\end{align*}
We put $F_0:=F_{rest}\cup F_{>w}$ and consider the projection
$$
Q^\Phi_{F_0}:=\bigvee_{s\in F_0}Q_s^\Phi.
$$
Note that $(Q_{F_0}^\Phi)^\bot:=I-Q_{F_0}^\Phi$ is a nonzero projection. For any $s\in F_0$ we  have
 $s\geq w$ and therefore   $\overline{\Phi}(d)$ and  $Q_{s}^\Phi$ commute,  by the last part of Lemma \ref{properties of family projections}.  This implies that
$$
 (Q_{F_0}^\Phi)^\bot \overline{\Phi}(d) Q_{s}^\Phi=0\quad  \text{ and }\quad Q_{s}^\Phi \overline{\Phi}(d)  (Q_{F_0}^\Phi)^\bot=0\quad \text{ for every }s\in F_0.
$$
Applying this observation to $Q_{w}^\Phi a Q_{w}^\Phi$ we get
\begin{equation}\label{compression equality one}
(Q_{F_0}^\Phi)^\bot \overline{\Phi}(d) a \overline{\Phi}(d)(Q_{F_0}^\Phi)^\bot= (Q_{F_0}^\Phi)^\bot \sum_{p ,q\in C  \atop  pq^{-1}w\in F_{w}} \overline{\Phi}(d)\overline{\Phi}(  a_{p,q}\otimes 1_{q^{-1}w} )\overline{\Phi}(d) (Q_{F_0}^\Phi)^\bot.
\end{equation}
%We consider the relevant cases.
To deduce \eqref{equation in the proof3} from here we consider the cases (1)-(3).

\noindent{\emph{Case} (1)}. Assume that  $P^*=\{e\}$.
Then $F_w=\{w\}$. In particular, allowing $d$ in \eqref{compression equality one} to run through an approximate unit in $\KK(w,w)$ and taking strong limit we get
$$
(Q_{F_0}^\Phi)^\bot Q_{w}^\Phi a Q_{w}^\Phi(Q_{F_0}^\Phi)^\bot
=(Q_{F_0}^\Phi)^\bot \overline{\Phi} \Big(\sum_{p ,q\in C  \atop  pq^{-1}w=w}a_{p,q}\otimes 1_{q^{-1}w}\Big)(Q_{F_0}^\Phi)^\bot.
$$
Employing this equality and condition $(C')$  we have
\begin{align*}
\|a\|&\,\,\geq \,\,\|(Q_{F_0}^\Phi)^\bot Q_{w}^\Phi a Q_{w}^\Phi(Q_{F_0}^\Phi)^\bot\|
=\|(Q_{F_0}^\Phi)^\bot \overline{\Phi} \Big(\sum_{p ,q\in C  \atop  pq^{-1}w=w}a_{p,q}\otimes 1_{q^{-1}w}\Big)(Q_{F_0}^\Phi)^\bot\|
\\
&\stackrel{\eqref{Coburn condition1} }{=} \|\overline{\Phi} \Big(\sum_{p ,q\in C  \atop  pq^{-1}w=w}a_{p,q}\otimes 1_{q^{-1}w}\Big)\|.
\end{align*}
This proves \eqref{equation in the proof3} and finishes the proof under hypothesis (1).

\noindent{\emph{Case} (2)}. Suppose that the group $\{\otimes 1_{x}\}_{x\in P^*}$ is aperiodic on $\KK$ and that $\KK$ is essential in $\LL_{\KK}$. Then  $\{\otimes 1_{x}\}_{x\in P^*}$ is also aperiodic on $\LL_{\KK}$, by Lemma \ref{a periodicity for essentials}.  Since $Z$ is self-adjoint, it follows from  \eqref{positivity of a summand of Z} that $\sum_{\{p ,q\in C:  pq^{-1}w=w\}}a_{p,q}\otimes 1_{q^{-1}w}\in \LL_{\KK}(w,w)$
is also self-adjoint. Hence by Lemma \ref{Muhly Solel lemma} there is $d\in \LL_{\KK}(w,w)$, $\|d\|=1$, such that
\begin{equation}\label{inequality one}
\|(d\otimes 1_{x}) (a_{p,q} \otimes 1_{q^{-1}w})d\| <\frac{\varepsilon}{8|C|^2}
\end{equation}
for  every $p,q\in C$ and  $x\in P^{*}\setminus\{e\}$ where  $pq^{-1}w=wx$, and
\begin{equation}\label{inequality two}
\| d  \Big(\sum_{p ,q\in C  \atop  pq^{-1}w=w}a_{p,q}\otimes 1_{q^{-1}w}\Big) d\| > \| \sum_{p ,q\in C  \atop  pq^{-1}w=w}a_{p,q}\otimes 1_{q^{-1}w}\| -\frac{\varepsilon}{8}.
\end{equation}
We now return to the computation \eqref{compression equality one} with $d$ chosen above, and note that
\begin{align*}
(Q_{F_0}^\Phi)^\bot \overline{\Phi}(d)a \overline{\Phi}(d)(Q_{F_0}^\Phi)^\bot&
= (Q_{F_0}^\Phi)^\bot \overline{\Phi}\Big(d  \big(\sum_{p ,q\in C  \atop  pq^{-1}w=w}a_{p,q}\otimes 1_{q^{-1}w}\big) d\Big)(Q_{F_0}^\Phi)^\bot
\\
&+ (Q_{F_0}^\Phi)^\bot \sum_{p ,q\in C, x\in P^{*}\setminus\{e\}  \atop  pq^{-1}w=wx} \overline{\Phi}\Big((d\otimes 1_{x}) (a_{p,q}\otimes 1_{q^{-1}w}) d\Big)(Q_{F_0}^\Phi)^\bot.
\end{align*}
Since $\KK$ is essential in $\LL_{\KK}$ and $\Phi$ is injective, $\overline{\Phi}$ is isometric on $\LL_{\KK}$, by the last part of Proposition \ref{extensions of representations on Hilbert spaces0}. Note that $d  \big(\sum_{p ,q\in C  \atop  pq^{-1}w=w}a_{p,q}\otimes 1_{q^{-1}w}\big) d\in \LL_{\KK}(w,w)$. Thus we obtain the estimates
\begin{align*}
\|a\|& \,\,\geq \,\,\|(Q_{F_0}^\Phi)^\bot \overline{\Phi}(d)a\overline{\Phi}(d) Q_{F_0}^\Phi)^\bot\|
\\
& \stackrel{\eqref{inequality one}}{>} \|(Q_{F_0}^\Phi)^\bot\overline{\Phi} \Big(d  \big(\sum_{p ,q\in C  \atop  pq^{-1}w=w}a_{p,q}\otimes 1_{q^{-1}w}\big) d\Big)(Q_{F_0}^\Phi)^\bot\| -\frac{\varepsilon}{8}
\\
&
\stackrel{\eqref{Coburn condition1} }{=}\|\overline{\Phi} \Big(d  \big(\sum_{p ,q\in C  \atop  pq^{-1}w=w}a_{p,q}\otimes 1_{q^{-1}w}\big) d\Big)\| -\frac{\varepsilon}{8}\\
& =\| d  \big(\sum_{p ,q\in C  \atop  pq^{-1}w=w}a_{p,q}\otimes 1_{q^{-1}w}\big) d\| -\frac{\varepsilon}{8}
\\
&\stackrel{\eqref{inequality two}}{\geq}\|\overline{\Phi} \Big(\sum_{p ,q\in C  \atop  pq^{-1}w=w}a_{p,q}\otimes 1_{q^{-1}w}\Big)\| -\frac{\varepsilon}{4}.
% \stackrel{\eqref{equation in the proof5}}{\geq }\|Z\| -\frac{3\varepsilon}{4}.
\end{align*}
This proves \eqref{equation in the proof3} and finishes the proof under hypothesis (2).

\noindent{\emph{Case} (3)}.  Suppose that  $\theta:P \to G$ is a unital monomorphism and  the Fell bundle $\B^{\theta}$  is aperiodic. Then  $\theta:P \to \theta(P)\subseteq G$  is a controlled map of right LCM semigroups,   $P$ is cancellative, and $B_e^\theta=B_e^{i_\KK}$.
Injectivity of $\Phi$ and  condition  $(C')$ imply that $\Phi$ is a Toeplitz covariant representation, see the proof of Proposition~\ref{condition C versus condition C'}. Hence by Corollary \ref{Nica-Toeplitz representation corollary} (and the $C^*$-equality)  the maps
$\Phi\rtimes P:B^\theta_g\to C^*(\Phi(\KK))$, $g\in G$, are injective. Putting $F_G:=\{\theta(p)\theta(q)^{-1}: p,q\in F\}$ we have
$$
a=(\Phi\rtimes P)(\oplus_{g\in F_G} b_g)\quad  \text{ where }\quad b_g=\sum_{p,q\in F, \atop \theta(p)\theta(q)^{-1}=g}i_\KK(a_{p,q}), \,\, g\in F_G.
$$
By \eqref{bla form of an element}, $b_e$ is self-adjoint. Hence by  \cite[Lemma 4.2]{KS} there is $d\in B_e$ such that $\|d\|=1$, $\|db_ed-dad\|<\varepsilon /2$ and $\|db_ed\|>\|b_e\|-\varepsilon/2$. Thus we get
\begin{equation}\label{auxiliary inequality}
\|b_e\| -\varepsilon/2<  \|db_ed\|= \|\Phi\rtimes P (db_ed)\|\leq \|\Phi\rtimes P (dad)\|+ \varepsilon/2 \leq \|a\|+ \varepsilon/2.
\end{equation}
But since $P$ is cancellative and $\theta$ injective we also have
$$Z=\sum_{p\in F }T_{p,p}(a_{p,p})=\sum_{p,q\in F \atop \theta(p)\theta(q)^{-1}=e }T_{p,q}(a_{p,q})=T\rtimes P(b_e).
$$
  Hence $\|Z\|=\|b_e\|$ and   \eqref{auxiliary inequality} implies \eqref{equation in the proof4}.
\end{proof}
\begin{rem}\label{remark added in the proofs} %It will follow from the application of our theory to Fell bundles in section~\ref{groups section} that  
Under condition (3) in Theorem \ref{preludium to Nica Toeplitz uniqueness2}, also the implication  $(v)\Rightarrow (iii)$ holds. Indeed, given  Nica covariant $\Phi$ that is injective and Toeplitz covariant, assume $\theta:P \to G$ is an injective controlled homomorphism such that the associated Fell bundle $B^\theta$ is aperiodic.  Corollary~\ref{Nica-Toeplitz representation corollary} implies that $\Phi\rtimes P$ is injective on $B^{i_\KK}_e\cong B_e^\theta$ and, therefore  $\ker (\Phi\rtimes P) \subseteq \ker \Lambda$ by Proposition~\ref{uniqueness for cross-sectional algebras} below. Thus $\Phi$   generates  an exotic Nica-Toeplitz $C^*$-algebra by Theorem \ref{amenability main result}.
\end{rem}

\begin{cor}[Uniqueness Theorem I]\label{cor:uniquenessI}
Suppose that  $\KK$ is an amenable well-aligned ideal in a $C^*$-precategory $\LL$ and  that one of conditions $(1)-(3)$ in Theorem \ref{preludium to Nica Toeplitz uniqueness2} holds. Consider the following properties of
a Nica covariant representation $\Phi:\KK\to B(H)$:
\begin{itemize}
\item[(i)] $\Phi$ satisfies condition $(C)$;
\item[(ii)] The map $\Phi\rtimes P$ is an isomorphism $\NT_{\LL}(\KK) \cong C^*(\Phi(\KK))$.
\end{itemize}
Then  $(i)\Rightarrow (ii)$. If $\KK\otimes 1\subseteq \KK$, then $(i)\Leftrightarrow (ii)$.
\end{cor}
\begin{proof}
Since $\KK$ is amenable, the regular representation $T\rtimes P$ is injective. If any of the conditions $(1)-(3)$ in  Theorem \ref{preludium to Nica Toeplitz uniqueness2} is satisfied, we infer from that result that $\Phi$ generates an exotic Nica-Toeplitz $C^*$-algebra $C^*(\Phi(\KK))$. Now the very definition of an exotic Nica-Toeplitz algebra means that $\Phi\rtimes P$ is injective.
\end{proof}

By Proposition \ref{properties of a semi lattice of projections}, if the ideal $\KK$ is  $\otimes 1$-nondegenerate, then for every Nica covariant representation $\Phi$ of $\KK$ the extended representation $\overline{\Phi}$ of $\LL$ is Nica covariant. Moreover, if $\KK$ is essential in $\LL$ and $\Phi$ is injective, then  $\overline{\Phi}$ is injective, see  Proposition \ref{extensions of representations on Hilbert spaces0}.  This observation leads us to the following reformulation of Theorem~\ref{preludium to Nica Toeplitz uniqueness2}:

 \begin{thm}\label{Nica Toeplitz uniqueness}
Suppose that the well-aligned ideal $\KK$ in $\LL$ is essential and $\otimes 1$-nondegenerate. Assume also that  either $P^*=\{e\}$ or  the group $\{\otimes 1_{x}\}_{x\in P^*}$   is aperiodic on $\KK$. For every  Nica covariant representation $\Phi:\KK\to B(H)$ the extended representation $\overline{\Phi}$ of $\LL$  is  Nica covariant  and the following statements are equivalent:
\begin{itemize}
\item[(i)] $\Phi$ satisfies condition $(C)$;
\item[(ii)]  $\Phi$ generates an exotic  Nica-Toeplitz $C^*$-algebra $C^*(\Phi(\KK))$ of $\KK$,
 and $\overline{\Phi}$ generates an exotic  Nica-Toeplitz $C^*$-algebra $C^*(\overline{\Phi}(\LL))$ of $\LL$;
\item[(iii)] $\overline{\Phi}$ is injective and Toeplitz covariant;
\item[(iv)] The map $\overline{\Phi}\rtimes P$ is injective on the core $C^*$-subalgebra $B_e^{i_\LL}$ of $\NT(\LL)$.
\end{itemize}
\end{thm}
\begin{proof} By Theorem~\ref{preludium to Nica Toeplitz uniqueness2} condition (i) implies that $\Phi$ generates an exotic  Nica-Toeplitz $C^*$-algebra $C^*(\Phi(\KK))$ of $\KK$. By Lemma \ref{a periodicity for essentials}, if the group $\{\otimes 1_{x}\}_{x\in P^*}$   is aperiodic on $\KK$, then it is also aperiodic on $\LL$. By Corollary \ref{Corollary 7.99}, $\Phi$ satisfies condition $(C)$ if and only if $\overline{\Phi}$ satisfies this condition. Hence we may apply Theorem~\ref{preludium to Nica Toeplitz uniqueness2} to  the extended representation $\overline{\Phi}:\LL\to B(H)$. Then  we get that each of conditions (i), (iii) and (iv) is equivalent to  that $\overline{\Phi}$ generates an exotic  Nica-Toeplitz $C^*$-algebra $C^*(\overline{\Phi}(\LL))$ of $\LL$.
\end{proof}
The above theorem explains  why in general condition $(C)$  is stronger then the "uniqueness" for Nica-Toeplitz algebras: $(C)$ implies uniqueness not only for a representation of $\KK$ but also for the extended  representation of $\LL$. We make this comment more formal in next section, see Corollary~\ref{Uniqueness Theorem}.

\section{On the relationship between Nica-Toeplitz algebras of $\KK$ and $\LL$}\label{relationship-K-L}
Let $\KK$ be a well-aligned ideal in a right-tensor $C^*$-precategory $\LL$. In this section we collect  results which reflect relationship between the full and reduced Nica-Toeplitz algebras associated to $\KK$ and $\LL$. Recall from \eqref{homomorphism to be embedding} the existence of a homomorphism $\iota$ from $\NT_{\LL}(\KK) $ to $\NT(\LL)$. We write $\NT_{\LL}(\KK) \hookrightarrow \NT(\LL)$
whenever $\iota$ is injective. %We start with two results which exemplify such a situation.

\begin{lem}\label{lemma to be actually used}
If $\KK$ is  $\otimes 1$-nondegenerate, then $\NT_{\LL}(\KK)\hookrightarrow  \NT(\LL)$.
\end{lem}
\begin{proof}
Let $i_\KK:\KK\to \NT_{\LL}(\KK)$ be the universal covariant representation of $\KK$. Suppose that $\NT_{\LL}(\KK) \subseteq B(H)$ (for example in the usual universal representation as a $C^*$-algebra). Then
Proposition \ref{properties of a semi lattice of projections} and Proposition \ref{properties of family projections2}(ii) imply that the extension $\overline{i_\KK}:\LL\to B(H)$ of $i_\KK$ is Nica covariant. Then injectivity of $\iota$ follows as explained in Remark~\ref{remark about inclusions}.
\end{proof}

\begin{prop}
If $\KK$ is amenable, then
$
\NT_{\LL}(\KK)\hookrightarrow  \NT(\LL).
$
\end{prop}
\begin{proof}
The Fock representation $T$ of $\KK$ is the restriction of a Nica covariant representation  $\overline{T}:\LL\to \LL(\FF_\KK)$, cf. Proposition \ref{Toeplitz representation definition}. The hypothesis on $\KK$ means that  $T\rtimes P$ is an isomorphism from $\NT_{\LL}(\KK)$ onto $\NT_{\LL}^r(\KK)$. With $\eta$ denoting the composition of the inclusion map from $\NT_{\LL}^r(\KK)$ into $\LL(\FF_\KK)$ with $T\rtimes P$, it follows that  $i_\KK:\KK\to \NT_{\LL}(\KK)$ admits the extension $\overline{T}:\LL\to \eta(\NT_{\LL}(\KK))$, which is a Nica covariant representation of $\LL$. Then injectivity follows by Remark \ref{remark about inclusions}.
\end{proof}

\begin{thm}\label{embedding main result}
Let $\KK$ be a well-aligned ideal in a right-tensor $C^*$-precategory $(\LL, \{\otimes 1_r\}_{r\in P})$. Suppose that
either $P$ is cancellative or $\KK$ is essential in the right-tensor $C^*$-precategory $\LL_\KK$ generated by $\KK$.
There is an embedding of $C^*$-algebras:
$$
\iota^r:\NT_{\LL}^r(\KK)\hookrightarrow  \NT^r(\LL)
$$
determined by
$
\NT_{\LL}^r(\KK)\ni T(a)\to S(a)\in \NT^r(\LL)$,  $a\in \KK(p,q),\, p,q\in P$, where $T$ and $S$ are  Fock representations of $\KK$ and $\LL$, respectively. Moreover, $\iota^r$ intertwines the conditional expectations $E^T: \NT_{\LL}^r(\KK)\to B_\KK$ and $E^S: \NT^r(\LL)\to B_\LL$, in the sense that $E^S\circ \iota^r=\iota^r\vert_{B_\KK}\circ E^T$.
\end{thm}
\begin{proof}
Suppose first that $P$ is cancellative. Then $E^T$ and $E^S$ are  faithful conditional expectations onto $B_\KK=B_e^T$ and  $B_\LL=B_e^S$ respectively, and are given by \eqref{another form of E T}.  Since $S$ is an injective Nica-Toeplitz representation of $\LL$, it restricts to an injective Nica-Toeplitz representation $\Phi:\KK\to \LL(\FF_\LL)$, and $E^S$ restricts to a faithful conditional expectation $E_{\Phi}:C^*(\Phi(\KK))\to B_e^{\Phi} \subseteq C^*(\Phi(\KK))$. Applications of  Corollaries \ref{Nica-Toeplitz representation corollary} and \ref{the reduced core}  yield an  isomorphism $h:B_\KK\to B_e^{\Phi}\subseteq B_{\LL}$ such that $\Phi=h\circ T$ on each space $\KK(p,p)$, $p\in P$. The composition $E^\Phi:=  h^{-1} \circ E_{\Phi}$ is a faithful completely positive map from $C^*(\Phi(\KK))$ onto $B_\KK$ satisfying  equation \eqref{formula defining strange map}.
 Hence  Proposition~\ref{preludium to Nica Toeplitz uniqueness} gives an isomorphism $\Phi_*$ from $C^*(\Phi(\KK))$ onto $\NT_\KK^r(\LL)$. Composing $(\Phi_*)^{-1}$ with the embedding of $C^*(\Phi(\KK))=C^*(S(\KK))$ into $\NT^r(\LL)$  yields an embedding $\iota^r$. Since $E^S\circ \iota^r=E_\Phi\circ (\Phi_*)^{-1}$ and $\iota^r \vert_{B_\KK}\circ E^T=h\circ E^T$, the claim about intertwining conditional expectations follows.

Suppose now that $\KK$ is essential in $\LL_\KK$. Let $Z\in B_\KK$ be as in \eqref{definition of Z}  and put
$$
Z_S:=\sum_{p,q\in F }\bigoplus_{w\in pP\cap qP, t\in P\atop p^{-1}w=q^{-1}w }S_{p,q}^{w,t}(a_{p,q}) \in B_\LL,
$$
where the adjointable maps $S_{p,q}^{w,t}$ are as defined in Lemma~\ref{lemma about Fock operators} (for $\KK=\LL$). By Corollary~\ref{cor:norm of Z formula},  $\|Z\|=\|Z_S\|$, so the map $Z\to Z_S$ extends to an isometry  $h:B_\KK\to h(B_\KK)\subseteq \B_\LL$.  Let $\Phi:\KK\to \LL(\FF_\LL)$ be the restriction of $S$. Then  $E^S$ restricts to a faithful map $E_{\Phi}:C^*(\Phi(\KK))\to h(B_\KK)\subseteq C^*(\Phi(\KK))$. The composition $E^\Phi:=  h^{-1} \circ E_{\Phi}$ is a faithful completely positive map satisfying  \eqref{formula defining strange map} and the proof if completed as in the case of $P$ cancellative.
\end{proof}

\begin{cor}\label{proposition actually to be used :0} Suppose that $\LL$ is amenable and that either $P$ is cancellative or $\KK$ is essential in the right-tensor $C^*$-precategory $\LL_\KK$ generated by $\KK$. If
$\NT_{\LL}(\KK)\hookrightarrow  \NT(\LL)$, then $\KK$ is amenable.
\end{cor}
\begin{proof}
By our assumptions and Theorem \ref{embedding main result} we have a commutative diagram
$$
\begin{xy}
\xymatrix{
\NT_{\LL}(\KK) \ar[d]_{T\rtimes P}\ar[rr]^{\iota}&  & \NT(\LL)   \ar[d]^{S\rtimes P}
  \\
		\NT_{\LL}^r(\KK)  \ar[rr]^{\iota^r}	&         & \NT_{\LL}^r(\KK)
			}
  \end{xy}
$$
where $\iota$, $\iota^r$ and $S\rtimes P$ are injective. Hence $T\rtimes P$ is injective.
\end{proof}

\begin{cor}[Uniqueness Theorem II]\label{Uniqueness Theorem}
Suppose that the well-aligned ideal $\KK$ in $\LL$ is   essential and $\otimes 1$-nondegenerate. Assume also that either
$P^*=\{e\}$ or that the group $\{\otimes 1_{x}\}_{x\in P^*}$   is aperiodic on $\KK$. If $\LL$ is amenable, then $\KK$ is also amenable and for any Nica covariant representation $\Phi:\KK\to B(H)$ the following  statements are equivalent:
\begin{itemize}
\item[(i)] $\Phi$ satisfies condition $(C)$;
\item[(ii)]   $\overline{\Phi}$  generates the universal Nica-Toeplitz algebra $\NT(\LL)$, i.e.  $C^*(\overline{\Phi}(\LL))\cong\NT(\LL).$
\end{itemize}
Under the embedding $\NT_{\LL}(\KK) \hookrightarrow \NT(\LL)$,   the isomorphism in $(ii)$  restricts to an isomorphism $C^*(\Phi(\KK))\cong \NT_{\LL}(\KK)$.
\end{cor}
\begin{proof}
Lemma \ref{lemma to be actually used} implies that
$\NT_{\LL}(\KK)\hookrightarrow  \NT(\LL)$, hence  $\KK$ is amenable by Corollary~\ref{proposition actually to be used :0}.
The claim follows from Theorem \ref{Nica Toeplitz uniqueness}.
  \end{proof}

\section{Fell bundles and right-tensor $C^*$-precategories over groups}\label{groups section}
In this section we show that the theory of right-tensor $C^*$-precategories over groups  is equivalent  to the theory of Fell bundles over (discrete) groups. On one hand, we explain how results about $C^*$-algebras associated to Fell bundles  follow from our results. On the other hand, we will see the origin of some of our assumptions. 
Throughout this section we assume that $P$ equals a (discrete) group $G$.
\begin{prop}\label{right tensors over groups}
If $(\LL, \{\otimes 1_r\}_{r\in G})$ is a right-tensor $C^*$-precategory over $G$,  then the Banach spaces
%\begin{equation}\label{Fell bundle from a tensor precategory over a discerete group}
$B_g:=\LL(g,e)$ for   $g\in G$ equipped with the operations
\begin{align}
\circ : B_g\times B_h \ni (b_g,b_h) &\longrightarrow b_g \circ b_h :=(b_g\otimes 1_h)b_h\in B_{gh}\label{multiplication definition}
\\
^\star: B_g \ni b_g &\longrightarrow b_g^{\star} :=b_g^*\otimes 1_{g^{-1}} \in B_{g^{-1}},\label{star operation}
\end{align}
for $g,h\in G$, give rise to a Fell bundle $\B^\LL:=\{B_g\}_{g\in G}$ isomorphic to the Fell bundle $\B^{\id}$ associated to $\LL$ and $\id$ as in Proposition \ref{gauge co-action}, and so
\begin{equation}\label{natural isomorphisms with cross sectional algebras}
C^*_r(\B^\LL)\cong \NT^r(\LL), \qquad C^*(\B^\LL)\cong \NT(\LL).
\end{equation}
 Conversely, for any Fell bundle  $\B=\{B_g\}_{g\in G}$ over a discrete group $G$ the spaces
\begin{equation}\label{category from Fell bundles}
\LL(g,h):=B_{gh^{-1}}, \qquad g,h \in G,
\end{equation}
with composition and involution inherited from $\B$ and  right-tensoring given by
\begin{equation}\label{right tenor from Fell bundles}
\otimes 1_x : \LL(g,h)=B_{gh^{-1}} \ni b \longrightarrow b\otimes 1_x:= b\in  B_{gx (hx)^{-1}}=\LL(gx,hx)
\end{equation}
for $g,h,x\in G$, yield a right-tensor $C^*$-precategory $\LL:=\{\LL(g,h)\}_{g,h\in G}$ such that $\B=\B^\LL$.
%and hence \eqref{natural isomorphisms with cross sectional algebras} holds with $\B$ in place of $\B^\LL$.
%In particular, $C^*_r(\B)\cong \NT^r(\LL)$ and $C^*(\B)\cong \NT(\LL)$.
\end{prop}
\begin{proof} Let $(\LL, \{\otimes 1_r\}_{r\in G})$ be a right-tensor $C^*$-precategory over  $G$. By Lemma \ref{lemma on automorphic actions on ideals}, for every $g,p,q \in G$ the map $\otimes 1_g:\LL(p,q)\to \LL(pg,qg)$ is an isometric isomorphism and for every $a\in \LL(p,q)$ we have
$
i_\LL(a\otimes 1_g)=i_\LL(a).
$ This readily implies that the maps $B_g=\LL(g,e)\ni a \mapsto i_{\LL}(a)\in B_g^{\id}$, $g\in G$, are isometric isomorphisms. Under these maps, the  operations in the Fell bundle $\B^{\id}=\{\B_g^{\id}\}_{g\in G}$ translate to \eqref{multiplication definition}, \eqref{star operation}. Hence, $\B$ is a Fell bundle isomorphic to $\B^{\id}$. The isomorphisms  \eqref{natural isomorphisms with cross sectional algebras} follow  from  Theorem \ref{amenability main result}. The remaining claims of the proposition are straightforward and left to the reader.
\end{proof}
\begin{rem}\label{remark on isomorphism of right tensors}
To clarify relationships between the constructions in Proposition \ref{right tensors over groups}, let $\B^\LL=\{\LL(g,e)\}_{g\in G}$ denote  the Fell bundle associated to a right-tensor $C^*$-precategory $\LL$
and let  $\LL_\B=\{B_{gh^{-1}}\}_{g,h \in G}$ be the right-tensor $C^*$-precategory associated to a  Fell bundle $\B$. Then
$
\B^{\LL_\B}=\B$  and $\LL_{\B^\LL}\cong \LL$ as right-tensor $C^*$-precategories: the maps $\otimes 1_h: \LL(gh^{-1},e)\to\LL(g,h) $ for $g,h\in G$
  implement an
isomorphism $\LL_{\B^\LL}\cong \LL$ of $C^*$-precategories which intertwines right-tensorings.
\end{rem}
\begin{rem}\label{rem:about Fock representation of Fell bundles} When $P=G$ is a group and $\LL$ is any right-tensor $C^*$-precategory   over  $G$,
 then for every $t\in G$ condition \eqref{condition for reducing Focks} is trivially satisfied  (one may take $x=p^{-1}t$).
Thus, by Proposition \ref{Fock t-th representations and reduced objects}, $
\NT_\LL^r(\KK)\cong C^*(T^t(\KK)), \text{ for all }t\in G,$
with $T^t:\LL\to \LL(\FF_\LL^t)$ denoting the $t$-th Fock representation on $\LL$. Therefore,   if $\B=\{B_g\}_{g\in G}$ is a Fell bundle and $\LL=\{B_{gh^{-1}}\}_{g,h \in G}$ is the associated $C^*$-precategory from Proposition~\ref{right tensors over groups}, then the Fock representation $T^e$ of $\LL$, viewed as a representation of $\B$, coincides with the usual Fock representation of $\B$ introduced in \cite{Exel}.
\end{rem}
For right-tensor $C^*$-precategories over groups   well-aligned ideals have  nice structure.
\begin{lem}\label{alignment for groups} Let $\KK$ be a well-aligned ideal  in a right-tensor $C^*$-precategory $(\LL, \{\otimes 1_r\}_{r\in G})$ over a group $G$.
Then $\KK$ is automatically $\otimes 1$-invariant and $\otimes 1$-nondegenerate. Moreover,
$\KK$ is essential in $\LL$ if and only if $\KK(e,e)$ is essential in $\LL(e,e)$.
\end{lem}
\begin{proof}
Lemma \ref{lemma on automorphic actions on ideals} gives directly that $\KK$ is  $\otimes 1$-invariant and $\otimes 1$-nondegenerate. If $\KK(e,e)$ is essential in $\LL(e,e)$, then $\KK(p,p)=\KK(e,e)\otimes 1_p$ is essential in $\LL(p,p)=\LL(e,e)\otimes 1_p$, which proves the second part of the assertion.
\end{proof}
\begin{lem}\label{Ideals in right tensors over groups} Let $(\LL, \{\otimes 1_r\}_{r\in G})$ be a right-tensor $C^*$-precategory and $\B^\LL=\{\LL(g,e)\}_{g\in G}$ the associated Fell bundle. Then there is a bijective correspondence between ideals $\I=\{I_g\}_{g\in G}$ in $\B$ and well-aligned ideals $\KK$ in $\LL$, given  by $I_g=\KK(g,e)$ for all $g\in G.$
\end{lem}
\begin{proof}  In view of Remark \ref{remark on isomorphism of right tensors}, we may assume that  $\LL=\{B_{gh^{-1}}\}_{g,h \in G}$ and the right tensoring is given by identity maps \eqref{right tenor from Fell bundles}. Now it is clear that for any $\otimes 1$-invariant ideal $\KK$ in $\LL$, $I_g=\KK(g,e)$ defines an ideal in $\B^\LL$. Conversely, let $\I$ be an ideal in $\B^\LL$ and put $\KK(g,h):=I_{gh^{-1}}$, for all $g,h\in G$. Clearly, $\KK$ is an ideal in $\LL$. It is also the smallest $\otimes 1$-invariant ideal in $\B$ satisfying $I_g=\KK(g,e)$ for all $g$. However, by Lemma \ref{alignment for groups} well-alignment and $\otimes 1$-invariance coincide. Hence this proves the assertion.
\end{proof}

As a first application of our results we obtain embedding  of cross-sectional algebras of ideals in Fell bundles.  The latter problem for sub-bundles is studied in \cite[Section 21]{exel-book}.

\begin{prop}\label{embeddings of cross-sectional algebras}
Given a right-tensor $C^*$-precategory $(\LL, \{\otimes 1_r\}_{r\in G})$ over a group $G$ and a well-aligned ideal $\KK$  in $\LL$, there are natural embeddings
$$
\NT^r(\KK)\hookrightarrow  \NT^r(\LL)\qquad \text{ and }\qquad  \NT(\KK)\hookrightarrow    \NT(\LL).
$$
Equivalently, for any ideal $\I$ in a Fell bundle  $\B$ over a discrete group $G$ we have
$$
C^*_r(\I)\hookrightarrow C^*_r(\B)\qquad \text{ and }\qquad   C^*(\I)\hookrightarrow C^*(\B).
$$
\end{prop}
\begin{proof}
 Theorem \ref{embedding main result} gives  $\NT^r(\KK)\hookrightarrow  \NT^r(\LL)$.  By Lemmas \ref{alignment for groups} and \ref{lemma to be actually used} we obtain $\NT(\KK)\hookrightarrow   \NT(\LL)$. In view of Proposition \ref{right tensors over groups} and Lemma \ref{Ideals in right tensors over groups}, the second part of the assertion is equivalent to the first one.
\end{proof}
\begin{rem}
It is shown in \cite[Theorem 21.13]{exel-book} that the embedding $C^*(\I)\hookrightarrow C^*(\B)$ is valid, not only for ideals but also for hereditary  sub-bundles $\I$ of  $C^*(\B)$. In fact, Proposition \ref{embeddings of cross-sectional algebras} could be deduced from \cite[Theorem 21.13 and Proposition 21.3]{exel-book}
\end{rem}
As a next application we get a correct version of \cite[Proposition 3.15]{KS}, see Remark \ref{stupid me} below.
\begin{prop}\label{intersection property for cross-sectional algebras}
Let $\B=\{B_g\}_{g\in G}$ be  a Fell bundle and  let $\Psi:C^*(\B)\to \B(H)$ be a representation injective on $B_e$. The following conditions are equivalent
 \begin{itemize}
\item[(i)]  $
\ker \Psi \subseteq \ker \Lambda \text{ where }\Lambda: C^*(\B)\to C^*_r(\B)  \text{ is the regular representation};
$
\item[(ii)] $\overline{\Psi(\bigoplus_{g\in G} B_g)}$ is  topologically graded, that is we have
$\|b_e\|\leq \|\Psi(\bigoplus_{g\in G} b_g)\|$   for any $\bigoplus_{g\in G} b_g\in \bigoplus_{g\in G} B_g$;

\item[(iii)] condition (ii) is satisfied for  any positive $\bigoplus_{g\in G} b_g\in \bigoplus_{g\in G} B_g$.
\end{itemize}
\end{prop}
\begin{proof} Let $\LL=\{B_{gh^{-1}}\}_{g,h \in G}$ be the right-tensor $C^*$-precategory associated to $\B$, as in Proposition \ref{right tensors over groups}. It is straightforward  that  the maps $\Phi_{g,h}:=\Psi|_{B_{gh^{-1}}}$ for $g,h\in G$ form a representation $\Phi$ of $\LL$ such that  $\Phi_{gr,hr}(a\otimes 1_r)=\Phi_{g,h}(a)$ for $a\in \LL(g,h)$ and all $g,h,r \in G$. Accordingly,  $\Phi$ is  Nica covariant. Identifying $\NT(\LL)$ with $C^*(\B)$, via \eqref{natural isomorphisms with cross sectional algebras}, the core of $\NT(\LL)$ is $\LL(e,e)=B_e$ and $\Psi=\Phi\rtimes P$. Thus Proposition \ref{preludium to Nica Toeplitz uniqueness} applied to $\Phi$ gives the  equivalence of (i) and (ii).
Equivalence between (ii) and (iii) is explained in the beginning of the proof of Theorem \ref{preludium to Nica Toeplitz uniqueness2}.
\end{proof}
\begin{rem}\label{stupid me} Retain the notation of Proposition \ref{intersection property for cross-sectional algebras}.
Unless $\B$ is amenable, condition (i)  is stronger than condition (i'): $
\Lambda(\ker \Psi)\cap B_e=\{0\}$. Unfortunately, \cite[Proposition 3.15]{KS} says that counterparts of conditions (i'), (ii), (iii) are equivalent. In particular, the proof of implication  (i)$\Rightarrow$(ii) in \cite[Proposition 3.15]{KS} is incorrect. This mistake does not affect the remaining results of \cite{KS}.
\end{rem}
Let $\B$ be a Fell bundle over $G$ and   $\LL$  the associated right-tensor $C^*$-precategory, given by \eqref{category from Fell bundles} and \eqref{right tenor from Fell bundles}. It is immediate that
the action of $\{\otimes 1_g\}_{g\in G}$ on $\LL$ is aperiodic if and only if  $\B$ is aperiodic.  By \eqref{natural isomorphisms with cross sectional algebras}, amenability of $\B$ is equivalent to amenability of $\LL$.
Moreover, since  $g\leq h$ holds for any $g$, $h$ in  $G$, conditions (C) and (C') are void.
Thus Theorem \ref{preludium to Nica Toeplitz uniqueness2} reduces to the following  uniqueness-type result, which in view of Remark \ref{stupid me} is stronger than \cite[Corollary 4.3]{KS}.
\begin{prop}\label{uniqueness for cross-sectional algebras}
Let $\B=\{B_g\}_{g\in G}$ be  an aperiodic Fell bundle.  For any   representation $\Psi:C^*(\B)\to \B(H)$  injective on $B_e$ we have
$\ker \Psi \subseteq \ker \Lambda$.
\end{prop}

\begin{rem} By  \cite[Theorem 9.11]{KM}, when $G=\Z$ or $G=\Z_n$ for a square free number $n>0$, aperiodicity of $\B$ is not only sufficient but also  necessary  for Proposition \ref{uniqueness for cross-sectional algebras} to hold,  at least when $B_e$ contains an essential ideal which  is separable or of Type I.
\end{rem}
% We illustrate our uniqueness result  by the example of a group action.
\begin{ex}[Crossed products by group actions]\label{Crossed products by group actions}
Let $\alpha:G\to \Aut(A)$ be an action of a discrete group $G$ on a $C^*$-algebra $A$. Consider the $C^*$-precategory $\LL=\{\LL(g,h)\}_{g,h\in G}$ where for each $g,h\in G$,  $\LL(g,h)$ equals $A$ as a Banach space, and composition and involution in $\LL$ are given by multiplication and involution in $A$. Then $a \longrightarrow a \otimes 1_r:=\alpha_r(a)$ for $a\in A=\LL(g,h)$ and $g,h,r\in G$ defines a
right-tensoring on $\LL$. The corresponding Fell bundle $\B^\LL$  coincides with the semi-direct product bundle associated to $\alpha$, cf. \cite{exel-book}. Hence
$$
\NT^r(\LL)\cong A\rtimes_\alpha^r G\quad  \text{ and }\quad \NT(\LL)\cong A\rtimes_\alpha G.
$$
We have a natural bijective correspondence between $\alpha$-invariant ideals $I$ in $A$ and well-aligned ($\otimes 1$-invariant) ideals $\KK$ in $\LL$.
Moreover, assuming that $A$  contains an essential ideal which is separable or of Type I,  by \cite[Theorems 2.13, 8.1 and Corollary 9.10]{KM}, the following conditions are equivalent:
\begin{itemize}
\item[(i)] the action of $\{\otimes 1_g\}_{g\in G}$ on $\LL$ is aperiodic;
\item[(ii)] $\alpha$ is pointwise properly outer, i.e. each $\alpha_g$ for $g\in G\setminus \{e\}$ is properly outer;
\item[(iii)] the dual action $\widehat{\alpha}:G\to  \text{Homeo}(\widehat{A})$ is topologically free, i.e. for any $g_1,\dots ,g_n\in G\setminus \{e\}$ the set $\bigcap_{i=1}^n\{\pi\in \widehat{A}:\widehat{\alpha}_{g_i}(\pi)=\pi\}$ has empty interior.
\end{itemize}
Using this, one  recovers \cite[Theorem 1]{Arch_Spiel} from Proposition \ref{uniqueness for cross-sectional algebras}.
\end{ex}


\begin{thebibliography}{00}

\bibitem{alnr} S. Adji, M. Laca, M. Nilsen and I. Raeburn, \emph{Crossed
products by semigroups of endomorphisms and the Toeplitz algebras of ordered groups}, { Proc.  Amer. Math. Soc.} {\bf 122} (1994), 1133-1141.

\bibitem{Arch_Spiel}  R. J. Archbold, J. S. Spielberg,
\emph{Topologically free actions and ideals in discrete $C^*$-dynamical systems,}
Proc. Edinb. Math. Soc. (2) \textbf{37} (1993), 119--124.

\bibitem{blackadar} B. Blackadar,
\emph{Shape theory for $C^*$-algebras},
Math. Scand. {\bf 56}  (1985),  249--275.

\bibitem{brow-guen} N. P. Brown and E. P. Guentner, \emph{New $C^*$-completions of discrete groups and related spaces}, Bull. London. Math. Soc. {\bf 45} (2013), 1181--1193.


\bibitem{bls} N. Brownlowe, N. S. Larsen  and N. Stammeier,
\emph{On $C^*$-algebras associated to right LCM semigroups}, Trans. Amer. Math. Soc. {\bf 369} (2017), no.1, 31--68.

\bibitem{bls2} N. Brownlowe, N. S. Larsen  and N. Stammeier,
\emph{$C^*$-algebras of algebraic dynamical systems and right LCM semigroups}, Indiana Univ. Math. J., preprint arXiv:1503.01599v1.

\bibitem{BRRW} N. Brownlowe, J. Ramagge, D. Robertson and M. F. Whittaker,
{\em Zappa-Sz\'{e}p products of semigroups and their $C^*$-algebras}, J.
Funct. Anal. {\bf 266} (2014), 3937--3967.

\bibitem{BEM} A. Buss, R. Exel, R. Meyer,
\emph{Reduced C*-algebras of Fell bundles over inverse semigroups}, Isr. J. Math. \textbf{220} (2017),  225--274.

\bibitem{CL} J. Crisp and M. Laca,
\emph{Boundary quotients and ideals of Toeplitz  $C^*$-algebras of Artin groups}, J. Funct. Anal. {\bf 242} (2007), 127--156.


\bibitem{dr} S. Doplicher,  J. E. Roberts, \emph{A new duality theory for compact groups}, Invent. Math. \textbf{98} (1989), 157--218.

\bibitem{dpz} S. Doplicher. C. Pinzari, R. Zuccante, \emph{The $C^*$-algebra of a Hilbert bimodule}.\, Boll Unione Mat. Ital. Sez. B Artc. Ric. Mat. \textbf{1} (1998), 263--281.


\bibitem{Exel} R. Exel,
\emph{Amenability for Fell bundles},
J. reine angew. Math. \textbf{492} (1997), 41--73.


\bibitem{exel-book} R. Exel, Partial dynamical systems, Fell bundles and applications, book available at: mtm.ufsc.br/~exel/papers/pdynsysfellbun.pdf

\bibitem{F99} N.~J.~Fowler,
\emph{Discrete product systems of Hilbert bimodules},
Pacific J. Math.  {\bf 204} (2002), 335--375.

\bibitem{fmr}   N. J. Fowler,  P. S. Muhly, I. Raeburn, \emph{Representations of Cuntz-Pimsner algebras},
Indiana Univ. Math. J. {\bf 52} (2003), 569--605.

\bibitem{FR}  N.~J.~Fowler and I.~Raeburn, \emph{Discrete product systems and twisted crossed products by
semigroups}, J. Funct. Anal., {\bf 155} (1998), 171--204.

\bibitem{Fow-Rae} N.~J.~Fowler and I.~Raeburn, \emph{The Toeplitz algebra of a Hilbert bimodule}, Indiana Univ. Math. J. {\bf 48} (1999), 155--181.


\bibitem{glr} P. Ghez,  R. Lima, J.E. Roberts, \emph{$W^*$-categories}, Pacific J. Math. \textbf{120}  (1985), 79--109.


\bibitem{kwa-doplicher} B. K. Kwa\'sniewski,
\emph{$C^*$-algebras generalizing both relative Cuntz-Pimsner and Doplicher-Roberts algebras},
Trans. Amer. Math. Soc. {\bf 365} (2013), 1809--1873.


\bibitem{kwa-larII} B. K. Kwa\'{s}niewski and N. S. Larsen, \emph{Nica-Toeplitz algebras associated with product systems over right LCM semigroups}, accepted in J. Math. Anal. Appl., DOI :  10.1016/j.jmaa.2018.10.020, arXiv:1706.04951

\bibitem{KM} B. K. Kwa\'sniewski, R. Meyer,
\emph{Aperiodicity, topological freeness and pure outerness: from group actions to Fell bundles}, Studia Math. 241 (2018), 257--302.

\bibitem{kwa-szym} B. K. Kwa\'sniewski, W. Szyma\'nski,
\emph{Topological aperiodicity for product systems over semigroups of Ore type},
 J. Funct. Anal. 270 (2016), no. 9, 3453-3504.


\bibitem{KS} B. K. Kwa\'sniewski, W. Szyma\'nski, \emph{Pure infiniteness and ideal structure of $C^*$-algebras associated to Fell bundles}, J. Math. Anal. Appl. 445 (2017), no. 1, 898-943.

\bibitem{KRVV} D. Kyed, S. Raum, S. Vaes,  M. Valvekens,
\emph{$L^2$-Betti numbers of rigid $C^*$-tensor categories and discrete quantum groups}, Anal. PDE
 10 (2017), 1757-1791.
\bibitem{LR}
M. Laca and I. Raeburn, \emph{Semigroup crossed products and the Toeplitz algebras of nonabelian
groups}, J. Funct. Anal., \textbf{139} (1996), 415--440.


\bibitem{lance}
 E.C. Lance, Hilbert $C^*$-Modules: A Toolkit for Operator Algebraists.   Cambridge University Press, Cambridge (1995).

\bibitem{Law2} M.V.~Lawson,  \emph{Non-commutative Stone duality: inverse semigroups, topological groupoids and C$^*$-algebras}, Internat. J. Algebra Comput. \textbf{22} (2012), no. 6, 1250058, 47 pp.

\bibitem{Li} X. Li, \emph{Semigroup {$C^*$}-algebras and amenability of semigroups},
 J. Funct. Anal. {\bf 262} (2012), no. 10,  4302--4340.


\bibitem{murphy} G.J. Murphy, \emph{Crossed products of $C^*$-algebras by semigroups of automorphisms},
Proc. Lond. Math. Soc. \textbf{3} (1994), 423--448.

\bibitem{NT}  S. Neshveyev,  M. Yamashita, \emph{Drinfeld center and representation theory for monoidal categories},
 Comm. Math. Phys.  345 (2016), 385--434

%S. Neshveyev and L. Tuset, \emph{Compact quantum groups and their representation categories}, Cours Sp\'ecialis'es 20. Soci\'et\'e Math\'ematique de France, Paris, 2013.

\bibitem{N} A. Nica,
\emph{$C^*$-algebras generated by isometries and Wiener-Hopf operators},
J. Operator Theory, \textbf{27} (1992), 17--52.

\bibitem{No0} M.D.~Norling, {\em Inverse semigroup $C^*$-algebras associated
with left cancellative semigroups}, Proc. Edinb. Math. Soc. \textbf{57} (2014), no. 2, 533--564.


\bibitem{MS}
 P. Muhly and B. Solel, \emph{On the Morita equivalence of tensor algebras},
Proc. London Math Soc. \textbf{81} (2000), 113--168.

\bibitem{PV} S. Popa and S. Vaes, \emph{Representation theory for subfactors, $\lambda$-lattices and $C^*$-tensor categories}, Commun. Math. Phys. 340 (2015), 1239-1280.

\bibitem{Pim} M. V. Pimsner, \emph{A class of {$C\sp*$}-algebras generalizing both {C}untz-{K}rieger algebras and crossed products by {${\bf Z}$}}, in Free probability theory (Waterloo, ON, 1995), Amer. Math. Soc., Providence, RI, 1997, 189--212.


\bibitem{SY} A. Sims and T. Yeend,
\emph{$C^*$-algebras associated to product systems of Hilbert bimodules},
J. Operator Theory {\bf 64} (2010), 349--376.

\end{thebibliography}
\end{document}